\documentclass{article}
\usepackage{silence}
\PassOptionsToPackage{numbers}{natbib}

\IfFileExists{neurips_2026.sty}{%
  \usepackage[preprint]{neurips_2026}
}{%
  \usepackage[a4paper, total={6in, 8in}]{geometry}
  \usepackage[numbers]{natbib}
}

\usepackage{graphicx}
\usepackage{makecell}

\usepackage[utf8]{inputenc}
\usepackage[T1]{fontenc}
\usepackage{hyperref}
\usepackage{url}
\usepackage{booktabs}
\usepackage{amsfonts}
\usepackage{nicefrac}
\usepackage{amsthm}
\usepackage{microtype}
\usepackage{xcolor}

\usepackage{amsmath,amssymb,amsthm,mathtools,bm}
\usepackage{enumitem}
\usepackage{algorithm}
\usepackage{algpseudocode}
\hypersetup{hidelinks}

\newtheorem{theorem}{Theorem}[section]
\newtheorem{lemma}[theorem]{Lemma}
\newtheorem{proposition}[theorem]{Proposition}
\newtheorem{corollary}[theorem]{Corollary}
\theoremstyle{definition}
\newtheorem{definition}[theorem]{Definition}
\newtheorem{assumption}[theorem]{Assumption}

\newcommand{\RR}{\mathbb R}
\newcommand{\EE}{\mathbb E}
\newcommand{\cF}{\mathcal F}
\newcommand{\cG}{\mathcal G}
\newcommand{\argmin}{\operatorname*{arg\,min}}
\newcommand{\norm}[1]{\left\lVert #1\right\rVert}
\newcommand{\ip}[2]{\left\langle #1,#2\right\rangle}
\newcommand{\one}{\mathbf 1}

\title{Heterogeneous-Horizon Exact-Weight Local SGD}

\author{Dmitry~Pasechnyuk-Vilensky \\ MBZUAI, UAE \\
\And
Martin~Tak\'a\v{c}  \\ MBZUAI, UAE
}

\begin{document}

\maketitle

\begin{abstract}
We study adaptive aggregation for heterogeneous local SGD in convex finite-sum optimization, allowing heterogeneous local horizons, minibatch sizes, gradient noise, and participation. We introduce HEW-Local SGD, a corrected local-SGD method that chooses nodewise server weights by minimizing an explicit one-round upper bound on the next objective value. This yields an exact local-control formulation with a threshold simplex update, separable amplitude updates, and a one-step guarantee under arbitrary predictable participation. We also introduce two post-local variants: a corrected heterogeneous method and a simpler homogeneous specialization. We establish one-step guarantees and global benchmark-style convergence results. In the regimes where comparison is appropriate, the theory matches the qualitative communication-efficient picture of recent LocalSGD/SCAFFOLD analyses, while also giving explicit guarantees for unequal local horizons.
\end{abstract}
\section{Introduction}

Local training reduces communication in federated and distributed optimization, but amplifies client heterogeneity \citep{gorbunov2021local,karimireddy2020scaffold,li2020federated,mcmahan2017communication,wang2020tackling}. Existing methods address this through proximal corrections, control variates, normalized aggregation, and refined analyses of corrected LocalSGD schemes \citep{gorbunov2021local,karimireddy2020scaffold,li2020federated,wang2020tackling}. Recent benchmark-style theory has further clarified when local methods improve over minibatch SGD, especially under stochastic and higher-order similarity effects \citep{luo2025revisiting,mangold2025scaffold}.

A key but usually fixed component in this literature is server aggregation. Standard choices include uniform averaging, data-size weighting, and normalized aggregation \citep{mcmahan2017communication,wang2020tackling}. We study adaptive aggregation when clients differ simultaneously in horizon, noise, and update scale.

Our starting point is \emph{Heterogeneous-Horizon Exact-Weight Local SGD} (HEW-Local SGD). It is a corrected local-SGD method with heterogeneous local horizons, nodewise amplitudes, and exact simplex-constrained server weights. Its central object is an explicit one-round certificate for the next expected objective value, obtained from endpoint mean--variance bounds for heterogeneous corrected local branches. This gives an exact local-control baseline in which the server weights are theorem-backed minimizers.

This baseline is conservative for practice because it optimizes before local randomness is realized and is driven by upper bounds. We therefore introduce two post-local variants. The first keeps corrected heterogeneous local branches and optimizes aggregation over realized corrected endpoints. The second further specializes to a homogeneous plain-local regime with a simpler post-local rule. The paper thus studies one idea --- endpoint-adaptive aggregation --- in three forms.

\paragraph{Contribution.} (1) Adaptive aggregation for heterogeneous local SGD as a single exact endpoint-based principle, and derive from it one local-control method and two post-local methods. (2) Convergence rates that, in the structured comparison regimes used in recent LocalSGD/SCAFFOLD theory, recover the same communication-efficient scaling picture. (3) Explicit guarantees for unequal local horizons, giving a new rate statement for heterogeneous numbers of local steps.

\paragraph{Limitations.} The strongest comparisons to prior SOTA hold in the structured regimes where such comparisons are standard; the general heterogeneous theory is less sharp. The local-control baseline is intentionally conservative, which motivates the post-local variants. The analysis is restricted to convex finite-sum models.

\section{Model}

\begin{assumption}[Convex smooth sum-of-sums model with invariant ball]
Let
\begin{equation}
F(x)=\frac1n\sum_{i=1}^n F_i(x),\qquad F_i(x)=\frac1{m_i}\sum_{j=1}^{m_i}\phi_{ij}(x),
\label{eq:model}\end{equation}
where each $\phi_{ij}:\RR^d\to\RR$ is convex and $L$-smooth:
\[
\norm{\nabla \phi_{ij}(x)-\nabla \phi_{ij}(y)}\le L\norm{x-y}\qquad\text{for all }x,y\in\RR^d.
\]
Assume there exists $x_\star\in\arg\min F$ and $R>0$ such that, almost surely, $x_t\in B(x_\star,R)$ and $y_{i,t}^{(\ell)}\in B(x_\star,R)$
for every round $t$, every active node $i$, and every local step $\ell=0,\dots,H_i$, where $B(x_\star,R):=\{x\in\RR^d:\ \norm{x-x_\star}\le R\}$.
\end{assumption}

\begin{assumption}[Minibatch stochastic gradients]
For each node $i$, define
\[
v_i^2:=\sup_{x\in B(x_\star,R)}\frac1{m_i}\sum_{j=1}^{m_i}\norm{\nabla \phi_{ij}(x)-\nabla F_i(x)}^2.
\]
At round $t$ and local step $\ell$, active node $i$ draws a minibatch $\mathcal B_{i,t,\ell}\subseteq[m_i]$ of size $b_i$ and sets
\[
g_{i,t,\ell}=\frac1{b_i}\sum_{j\in\mathcal B_{i,t,\ell}}\nabla \phi_{ij}(y_{i,t}^{(\ell)}).
\]
Let $\cG_{i,t,\ell}$ denote the sigma-field generated by $\cF_t$ and by all minibatches of node $i$ drawn strictly before local step $\ell$ in round $t$. Conditionally on $\cG_{i,t,\ell}$, the minibatches are independent across active nodes and across local steps, and
\[
\EE[g_{i,t,\ell}\mid \cG_{i,t,\ell}]=\nabla F_i(y_{i,t}^{(\ell)}),\qquad \EE\!\left[\norm{g_{i,t,\ell}-\nabla F_i(y_{i,t}^{(\ell)})}^2\middle|\cG_{i,t,\ell}\right]\le \frac{v_i^2}{b_i}.
\]
\end{assumption}

\begin{assumption}[Amplitude range and executable variance proxies]
There exist constants $0<\underline\theta\le \bar\theta$ such that every amplitude satisfies $\theta_{i,t}\in[\underline\theta,\bar\theta]$. The server may use executable variance proxies $\widehat v_{i,t}^2\ge v_i^2$. Small-step restrictions are imposed only in the theorem or proof blocks that actually use them.
\end{assumption}

Let $\cF_t$ be the sigma-field generated by the initial state and all minibatches drawn strictly before round $t$; the active set $\mathcal S_t\subseteq[n]$ and every feasible control pair $(w_t,\theta_t)$ are $\cF_t$-measurable. Write $f_t:=F(x_t)-F_\star$, $F_\star:=F(x_\star)$, $\bar f:=LR^2/2$, and $\Delta_{S_t}:=\{w\in\RR_+^{S_t}:\sum_{i\in\mathcal S_t}w_i=1\}$. A theorem-level upper state is any $\cF_t$-measurable pair $(U_t,Q_t)$ such that $f_t\le U_t$ and $\max_i\norm{c_{i,t}-\nabla F_i(x_t)}^2\le Q_t$ almost surely; we write $U_t^\sharp:=\min\{U_t,\bar f\}$.

All helper identities and proofs are deferred to the appendix.

\section{Local Viewpoint}

The local layer supplies the exact one-step controller and the benchmark-style rates used for comparison with recent LocalSGD/SCAFFOLD theory. The fully heterogeneous full-participation convex surrogate recursion is unchanged mathematically and is recorded in Appendix~B to keep the main text focused on the comparison regimes.
\begin{theorem}[One-step local certificate majorant]
\label{th:predictive-majorant}
Fix a round $t$ and any one-step admissible upper state $(U_t,Q_t)$. For every
predictable feasible pair $(w,\theta)$ with $w\in\Delta_{S_t}$ and
$\theta\in[\underline\theta,\bar\theta]^{S_t}$,
\begin{equation}
\EE[F(x_{t+1})-F_\star\mid \cF_t,w,\theta]
\le
\mathcal J_t^{\mathrm{id}}(w,\theta),
\label{eq:main-one-step}\end{equation}
where
\[
\mathcal J_t^{\mathrm{id}}(w,\theta)
:=
U_t^\sharp
-
\sum_{i\in\mathcal S_t}w_i \mu_i(\theta_i;U_t,Q_t)
+
\frac{L}{2}\sum_{i\in\mathcal S_t}w_i^2 \kappa_i(\theta_i;U_t,Q_t),
\]
with
\[
A_i(\theta_i):=\frac{\theta_i}{2LR^2},
\qquad
s_i(U_t^\sharp;\theta_i):=U_t^\sharp-T_{A_i(\theta_i)}(U_t^\sharp),
\]
\begin{equation}
\rho_i(\theta_i;U_t,Q_t)
:=
32 e^{2\theta_i}\theta_i^3 U_t
+
\frac{16\theta_i+64 e^{2\theta_i}\theta_i^3}{L}Q_t
+
8 e^{2\theta_i}\frac{\theta_i^3 v_i^2}{L H_i b_i},
\label{eq:rho-main}\end{equation}
\begin{equation}
\kappa_i(\theta_i;U_t,Q_t)
:=
16 e^{2\theta_i}\frac{\theta_i^4}{L}U_t
+
32 e^{2\theta_i}\frac{\theta_i^4}{L^2}Q_t
+
\left(2\theta_i^2+4 e^{2\theta_i}\theta_i^4\right)\frac{v_i^2}{L^2 H_i b_i},
\label{eq:kappa-main}\end{equation}
\[
\mu_i(\theta_i;U_t,Q_t)
:=
s_i(U_t^\sharp;\theta_i)-\rho_i(\theta_i;U_t,Q_t).
\]
Moreover, if the executable variance proxies $\widehat v_{i,t}$ are used, then
for every feasible $(w,\theta)$,
\[
\mathcal J_t^{\mathrm{id}}(w,\theta)
\le
\mathcal J_t^{\mathrm{ex}}(w,\theta),
\]
where $\mathcal J_t^{\mathrm{ex}}$ is obtained by replacing every occurrence of
$v_i^2$ in \eqref{eq:rho-main} and \eqref{eq:kappa-main} by
$\widehat v_{i,t}^2$.
\end{theorem}

\begin{theorem}[Exact one-step heterogeneous local control law]
\label{th:one-step-hetero}
Under the assumptions of Theorem~\ref{th:predictive-majorant}, let
\[
(w_t^{\mathrm{id}},\theta_t^{\mathrm{id}})
\in
\argmin_{\substack{w\in\Delta_{S_t}\\ \theta\in[\underline\theta,\bar\theta]^{S_t}}}
\mathcal J_t^{\mathrm{id}}(w,\theta),\quad(w_t^{\mathrm{ex}},\theta_t^{\mathrm{ex}})
\in
\argmin_{\substack{w\in\Delta_{S_t}\\ \theta\in[\underline\theta,\bar\theta]^{S_t}}}
\mathcal J_t^{\mathrm{ex}}(w,\theta).
\]
Then
\begin{equation}
\EE[F(x_{t+1})-F_\star\mid \cF_t]
\le
\mathcal J_t^{\mathrm{id}}(w_t^{\mathrm{id}},\theta_t^{\mathrm{id}})
\le
\mathcal J_t^{\mathrm{ex}}(w_t^{\mathrm{ex}},\theta_t^{\mathrm{ex}}).
\label{eq:one-step-control-law}\end{equation}
Moreover, for every predictable feasible benchmark pair
$(\bar w_t,\bar\theta_t)$,
\[
\EE[F(x_{t+1})-F_\star\mid \cF_t]
\le
\mathcal J_t^{\mathrm{id}}(\bar w_t,\bar\theta_t)-\Gamma_t,
\qquad
\Gamma_t\ge 0,
\]
where
\[
\Gamma_t
=
\mathcal J_t^{\mathrm{id}}(\bar w_t,\bar\theta_t)
-
\mathcal J_t^{\mathrm{id}}(w_t^{\mathrm{id}},\theta_t^{\mathrm{id}}).
\]
For fixed amplitudes, the weight subproblem is a strictly convex quadratic
program with the exact KKT law
\[
w_{i,t}^\star
=
\frac{\bigl(\mu_i(\theta_i;U_t,Q_t)-\lambda_t\bigr)_+}
{L\kappa_i(\theta_i;U_t,Q_t)},
\]
where $\lambda_t$ is chosen so that the weights sum to one.
\end{theorem}
For the full-participation uniform-controller specialization $\mathcal S_t=[n]$, $H_i\equiv H$, $b_i\equiv b$, $\theta_{i,t}\equiv\vartheta$, and $w_{i,t}\equiv 1/n$, write $v^2:=\max_i v_i^2$ and $q_t:=\max_i\EE\norm{e_{i,t}}^2$. The direct Bellman coefficients are
\begin{equation*}
a_{\mathrm{dir}}:=\frac{\vartheta}{2LR^2},\; \beta_{\mathrm{dir}}:=48\vartheta^3+\frac{32\vartheta^4}{n},\; \gamma_{\mathrm{dir}}:=\frac{96\vartheta^3}{L}+\frac{64\vartheta^4}{Ln},\; \delta_{\mathrm{dir}}:=\left(16\vartheta^3+\frac{16\vartheta^2}{n}\right)\frac{v^2}{L H b},
\end{equation*}
while the tracking recursion uses $A_q:=6v^2/(Hb)$, $B_q:=144L\vartheta^2$, and $C_q:=288\vartheta^2$.
\begin{theorem}[Global stochastic PL contraction in the uniform-controller branch]
\label{th:pl-global}
Assume Assumptions~\ref{ass:uniform-special} and \ref{ass:PL-benchmark}. Define
\begin{equation}
a_{\mathrm{PL}}:=\mu\frac{\vartheta}{L}-\beta_{\mathrm{dir}}.
\label{eq:aPL-main}\end{equation}
Suppose $a_{\mathrm{PL}}>0$, $a_{\mathrm{PL}}<1-C_q$, $a_{\mathrm{PL}}(1-C_q)>B_q\gamma_{\mathrm{dir}}$.
Define
\[
\rho_{\mathrm{PL}}
:=
\frac{
a_{\mathrm{PL}}+1-C_q
-
\sqrt{(1-C_q-a_{\mathrm{PL}})^2+4B_q\gamma_{\mathrm{dir}}}
}{2},
\]
\begin{equation}
\lambda_{\mathrm{PL}}
:=
\frac{
2\gamma_{\mathrm{dir}}
}{
1-C_q-a_{\mathrm{PL}}+
\sqrt{(1-C_q-a_{\mathrm{PL}})^2+4B_q\gamma_{\mathrm{dir}}}
}.
\label{eq:lambdaPL-main}\end{equation}
Then $s_t:=g_t+\lambda_{\mathrm{PL}}q_t$
satisfies
\[
s_{t+1}
\le
(1-\rho_{\mathrm{PL}})s_t
+
\delta_{\mathrm{dir}}+\lambda_{\mathrm{PL}}A_q,
\]
and therefore
\[
g_t
\le
(1-\rho_{\mathrm{PL}})^t\bigl(g_0+\lambda_{\mathrm{PL}}q_0\bigr)
+
\frac{\delta_{\mathrm{dir}}+\lambda_{\mathrm{PL}}A_q}{\rho_{\mathrm{PL}}}.
\]
In the readable safe regime $2 \mu \vartheta \le L$, $\vartheta^2\le \min\!\left\{\frac{\mu}{400L},\frac1{576}\right\}$, one has
\begin{equation}
\rho_{\mathrm{PL}}\ge \frac{\mu\vartheta}{8L},
\qquad
\frac{\delta_{\mathrm{dir}}+\lambda_{\mathrm{PL}}A_q}{\rho_{\mathrm{PL}}}
=
O\!\left(\frac{\vartheta v^2}{\mu n H b}+\frac{\vartheta^2 v^2}{\mu H b}\right).
\label{eq:pl-floor-main}\end{equation}
\end{theorem}

\begin{theorem}[Global higher-order convex benchmark bounds in the uniform-controller branch]
\label{th:ho-main}
Under Assumptions~\ref{ass:uniform-special} and \ref{ass:ho}, define
\[
K_{\mathrm{ho}}:=(\mathcal H+MR)^2,
\qquad
a_{\mathrm{ho}}:=\frac{\vartheta}{2LR^2},
\qquad
\underline a_{\mathrm{ho}}:=\frac{\vartheta}{2LR^2(1+\vartheta/4)}.
\]
Then the one-round recursion satisfies
\[
g_{t+1}
\le
T_{a_{\mathrm{ho}}}(g_t)
+
\frac{2\vartheta}{L}K_{\mathrm{ho}}\Delta_t^2
+
\frac{\vartheta^2}{2L n H}\frac{v^2}{b},
\]
where
\begin{equation}
\Delta_t^2
:=
\frac{96\vartheta^2}{L}g_t
+
\frac{32\vartheta^2 v^2}{L^2 H b}
+
\frac{192\vartheta^2}{L^2}q_t.
\label{eq:Delta-t-main}\end{equation}
Moreover, if $\bar q
:=
\max\left\{q_0,\frac{A_q+B_q\bar f}{1-C_q}\right\}$,
then $q_t\le \bar q$ for all $t$, and therefore
\begin{equation}
g_{t+1}
\le
g_t
-
\underline a_{\mathrm{ho}} g_t^2
+
\beta_{\mathrm{ho}} g_t
+
\delta_{\mathrm{ho}},
\label{eq:ho-rec-main}\end{equation}
where
\begin{equation}
\beta_{\mathrm{ho}}
:=
\frac{2\vartheta}{L}K_{\mathrm{ho}}\frac{96\vartheta^2}{L},
\qquad
\delta_{\mathrm{ho}}
:=
\frac{2\vartheta}{L}K_{\mathrm{ho}}
\left(
\frac{32\vartheta^2 v^2}{L^2 H b}+
\frac{192\vartheta^2}{L^2}\bar q
\right)
+
\frac{\vartheta^2}{2L n H}\frac{v^2}{b}.
\label{eq:ho-coeff-main}\end{equation}
Consequently, for every integer $T\ge 1$,
\begin{equation}
\min_{0\le t\le T-1} g_t
\le
\frac{\beta_{\mathrm{ho}}}{\underline a_{\mathrm{ho}}}
+
\sqrt{\frac{\delta_{\mathrm{ho}}}{\underline a_{\mathrm{ho}}}}
+
\sqrt{\frac{g_0}{\underline a_{\mathrm{ho}}T}}.
\label{eq:ho-explicit-main}\end{equation}
In the homogeneous quadratic benchmark subcase $\mathcal H=0$, $M=0$,
one has $K_{\mathrm{ho}}=0$ and therefore
\[
\min_{0\le t\le T-1} g_t
\le
C_1\sqrt{\frac{L R^2 g_0}{\vartheta T}}
+
C_2 R v\sqrt{\frac{\vartheta}{n H b}}
\]
for absolute constants $C_1,C_2>0$.
\end{theorem}

\begin{algorithm}[H]
\caption{HEW-Local SGD}
\label{alg:main}
\begin{algorithmic}[1]
\Require Current state $x_t$, active set $\mathcal S_t$, local controls $(c_{i,t})_{i=1}^n$, theorem-level state information
\State If $t=0$, initialize $c_0\gets \frac1n\sum_{i=1}^n c_{i,0}$
\State Solve the local control problem for $(w_t,\theta_t)$
\State Broadcast $(x_t,c_t,\theta_{i,t})$ to each active node $i\in\mathcal S_t$
\ForAll{$i\in\mathcal S_t$ in parallel}
    \State Set $\eta_{i,t}\gets \theta_{i,t}/(L H_i)$, $y_{i,t}^{(0)}\gets x_t$
    \For{$\ell=0,\dots,H_i-1$}
        \State Draw minibatch $\mathcal B_{i,t,\ell}$ and compute $g_{i,t,\ell}$
        \State Update $y_{i,t}^{(\ell+1)}\gets y_{i,t}^{(\ell)}-\eta_{i,t}(g_{i,t,\ell}-c_{i,t}+c_t)$
    \EndFor
    \State Set $\Delta_{i,t}\gets y_{i,t}^{(H_i)}-x_t$
    \State Update $c_{i,t+1}\gets c_{i,t}-c_t+\frac{1}{H_i\eta_{i,t}}(x_t-y_{i,t}^{(H_i)})$
    \State Send $(\Delta_{i,t},\Delta c_{i,t})$ to the server, where $\Delta c_{i,t}:=c_{i,t+1}-c_{i,t}$
\EndFor
\ForAll{$i\notin\mathcal S_t$}
    \State Set $c_{i,t+1}\gets c_{i,t}$
\EndFor
\State Set $x_{t+1}\gets x_t+\sum_{i\in\mathcal S_t}w_{i,t}\Delta_{i,t}$
\State Set $c_{t+1}\gets c_t+\frac1n\sum_{i\in\mathcal S_t}\Delta c_{i,t}$
\end{algorithmic}
\end{algorithm}

\section{Post-Local Aggregation Controllers}
\label{sec:postlocal-main}
\subsection{Heterogeneous corrected post-local controller}

Fix a common round amplitude $\vartheta_t\in[\underline\theta,\bar\theta]$ and let active nodes run the corrected local branches of Algorithm~\ref{alg:main} with $\eta_{i,t}=\vartheta_t/(L H_i)$. Extend inactive endpoints by $\widetilde\Delta_{i,t}=0$, define $d_t(w):=\sum_{i=1}^n w_i\widetilde\Delta_{i,t}$, and set
\begin{equation*}
\Psi_t^{\mathrm{het}}(w):=\ip{c_t}{d_t(w)}+\frac{\Lambda_t}{2}\norm{d_t(w)}^2,\; \Delta_n(\mathcal S_t):=\{w\in\RR_+^n:\ \sum_{i=1}^n w_i=1,\ w_i=0\ \text{for }i\notin\mathcal S_t\}.
\end{equation*}
The server then chooses $w_t^{\mathrm{het}}\in\arg\min_{w\in\Delta_n(\mathcal S_t)}\Psi_t^{\mathrm{het}}(w)$ and sets $x_{t+1}=x_t+d_t(w_t^{\mathrm{het}})$.
\begin{algorithm}[ht]
\caption{Heterogeneous Post-local HEW-Local SGD}
\label{alg:realized-het}
\begin{algorithmic}[1]
\Require $x_t$, $(c_{i,t})_{i=1}^n$, $c_t$, active set $\mathcal S_t$, amplitude $\vartheta_t$, curvature $\Lambda_t>L$
\State Broadcast $(x_t,c_t,\vartheta_t)$ to each active node $i\in\mathcal S_t$
\ForAll{$i\in\mathcal S_t$ in parallel}
    \State Set $\eta_{i,t}\gets \vartheta_t/(L H_i)$ and $y_{i,t}^{(0)}\gets x_t$
    \For{$\ell=0,\dots,H_i-1$}
        \State Draw minibatch $\mathcal B_{i,t,\ell}$ and compute $g_{i,t,\ell}$
        \State Update $y_{i,t}^{(\ell+1)}
        \gets
        y_{i,t}^{(\ell)}-\eta_{i,t}(g_{i,t,\ell}-c_{i,t}+c_t)$
    \EndFor
    \State Set $\Delta_{i,t}\gets y_{i,t}^{(H_i)}-x_t$
    \State Update $c_{i,t+1}
    \gets
    c_{i,t}-c_t+\frac{1}{H_i\eta_{i,t}}(x_t-y_{i,t}^{(H_i)})$
    \State Send $(\Delta_{i,t},\Delta c_{i,t})$ to the server, where $\Delta c_{i,t}:=c_{i,t+1}-c_{i,t}$
\EndFor
\ForAll{$i\notin\mathcal S_t$}
    \State Set $c_{i,t+1}\gets c_{i,t}$
\EndFor
\State Set $\widetilde\Delta_{i,t}\gets \Delta_{i,t}$ for $i\in\mathcal S_t$ and $\widetilde\Delta_{i,t}\gets 0$ for $i\notin\mathcal S_t$
\State Solve
\[
w_t^{\mathrm{het}}
\in
\argmin_{w\in\Delta_n(\mathcal S_t)}
\left\{
\ip{c_t}{\sum_{i=1}^n w_i\widetilde\Delta_{i,t}}
+\frac{\Lambda_t}{2}\norm{\sum_{i=1}^n w_i\widetilde\Delta_{i,t}}^2
\right\}
\]
\State Set $x_{t+1}\gets x_t+\sum_{i=1}^n w_{i,t}^{\mathrm{het}}\widetilde\Delta_{i,t}$
\State Set $c_{t+1}\gets c_t+\frac1n\sum_{i\in\mathcal S_t}\Delta c_{i,t}$
\end{algorithmic}
\end{algorithm}

\begin{theorem}[Post-local one-step certificate: heterogeneous corrected branch]
\label{th:realized-het-one-step}
Fix a round $t$ and define $\eta_t:=\Lambda_t-L>0$. Let
\[
\mathcal H_t
:=
\sigma\!\left(
\mathcal F_t,\ \{\mathcal B_{i,t,\ell}:i\in\mathcal S_t,\ 0\le \ell\le H_i-1\}
\right).
\]
Then, for every $\mathcal H_t$-measurable $w_t\in\Delta_n(\mathcal S_t)$,
\begin{equation}
F(x_t+d_t(w_t))-F_\star
\le
f_t+\Psi_t^{\mathrm{het}}(w_t)+\frac{1}{2\eta_t}\norm{\nabla F(x_t)-c_t}^2.
\label{eq:realized-het-pathwise}\end{equation}
If $w_t^{\mathrm{het}}$ is an exact minimizer, then
\[
F(x_{t+1})-F_\star\le f_t+\Psi_t^{\mathrm{het}}(w_t^{\mathrm{het}})+\frac{1}{2\eta_t}\norm{\nabla F(x_t)-c_t}^2.
\]
\end{theorem}
For the global heterogeneous rate, fix a deterministic comparator $a\in\Delta_n$ and define $V_1(a):=\sum_i a_i v_i^2/(H_i b_i)$ and $V_2(a):=\sum_i a_i^2 v_i^2/(H_i b_i)$. When
\begin{equation}
v_i^2>0\qquad\text{for every }i\in[n],
\label{eq:realized-het-positive-v}
\end{equation}
the noise-optimal comparator is $a_i^\star\sim H_i b_i/v_i^2$.

For every round $t$, set
\begin{equation}
A_t^{\mathrm{het}}:=\frac{\vartheta_t}{2LR^2},\qquad \underline A_t^{\mathrm{het}}:=\frac{A_t^{\mathrm{het}}}{1+A_t^{\mathrm{het}}\bar f}=\frac{\vartheta_t}{2LR^2(1+\vartheta_t/4)},
\label{eq:realized-het-A}\end{equation}
\[
\beta_t^{\mathrm{het}}:=192\vartheta_t^3+32\frac{\Lambda_t}{L}\vartheta_t^4,\qquad \gamma_t^{\mathrm{het}}:=\frac{1}{2(\Lambda_t-L)}+39\frac{\vartheta_t}{L}+64\frac{\Lambda_t}{L^2}\vartheta_t^4,
\]
\[
\delta_t^{\mathrm{het}}(a):=64\frac{\vartheta_t^3}{L}V_1(a)+16\frac{\Lambda_t\vartheta_t^2}{L^2}V_2(a).
\]
\begin{theorem}[Deterministic upper-state recursion: heterogeneous corrected post-local controller]
\label{th:realized-het-global}
Assume deterministic schedules, full participation $\mathcal S_t=[n]$, and $L<\Lambda_t\le 2L$, $\Lambda_t\vartheta_t\le L/2$ for every $t\ge 0$. Fix any deterministic comparator $a\in\Delta_n$. Let $(U_0,Q_0)$ satisfy
\begin{equation}
g_0\le U_0\le \bar f,
\qquad
Q_0\ge \max_{1\le i\le n}\EE\norm{c_{i,0}-\nabla F_i(x_0)}^2.
\label{eq:realized-het-init}\end{equation}
Define recursively
\begin{equation}
U_{t+1}
:=
\min\!\left\{
\bar f,\
T_{A_t^{\mathrm{het}}}(U_t)
+
\beta_t^{\mathrm{het}}U_t
+
\gamma_t^{\mathrm{het}}Q_t
+
\delta_t^{\mathrm{het}}(a)
\right\},
\label{eq:realized-het-U}\end{equation}
\begin{equation}
Q_{t+1}
:=
A_\chi+B_\chi U_t+C_\chi Q_t,
\label{eq:realized-het-Q}\end{equation}
where
\begin{equation}
A_\chi:=6\max_{1\le i\le n}\frac{v_i^2}{H_i b_i},
\qquad
B_\chi:=144L\bar\theta^2,
\qquad
C_\chi:=288\bar\theta^2.
\label{eq:realized-het-tracking-coeff}\end{equation}
Then
\begin{equation}
g_t\le U_t\le \bar f,
\qquad
\max_{1\le i\le n}\EE\norm{c_{i,t}-\nabla F_i(x_t)}^2\le Q_t
\qquad
\text{for every }t\ge 0.
\label{eq:realized-het-upper-state}\end{equation}
\end{theorem}

\begin{corollary}[Closed convex rate with unequal local horizons: heterogeneous corrected post-local controller]
\label{cor:realized-het-rate}
Under the hypotheses of Theorem~\ref{th:realized-het-global}, define
\[
\bar Q
:=
\max\left\{
Q_0,\
\frac{A_\chi+B_\chi\bar f}{1-C_\chi}
\right\}.
\]
Then $Q_t\le \bar Q$ for every $t\ge 0$, and therefore
\begin{equation}
U_{t+1}
\le
U_t-\underline A_t^{\mathrm{het}}U_t^2+\beta_t^{\mathrm{het}}U_t+d_t^{\mathrm{het}}(a),
\qquad
d_t^{\mathrm{het}}(a):=\gamma_t^{\mathrm{het}}\bar Q+\delta_t^{\mathrm{het}}(a).
\label{eq:realized-het-quadratic-linear}\end{equation}
If, in addition,
\begin{equation}
\vartheta_t\equiv \vartheta,
\qquad
\Lambda_t\equiv \Lambda,
\label{eq:realized-het-const-param}\end{equation}
define
\[
\underline A^{\mathrm{het}}
:=
\frac{\vartheta}{2LR^2(1+\vartheta/4)},
\qquad
\beta^{\mathrm{het}}
:=
192\vartheta^3+32\frac{\Lambda}{L}\vartheta^4,
\]
\begin{equation}
\gamma^{\mathrm{het}}
:=
\frac{1}{2(\Lambda-L)}
+
39\frac{\vartheta}{L}
+
64\frac{\Lambda}{L^2}\vartheta^4,
\qquad
\delta^{\mathrm{het}}(a)
:=
64\frac{\vartheta^3}{L}V_1(a)
+
16\frac{\Lambda\vartheta^2}{L^2}V_2(a),
\label{eq:realized-het-const-coeff-2}\end{equation}
\begin{equation}
d^{\mathrm{het}}(a):=\gamma^{\mathrm{het}}\bar Q+\delta^{\mathrm{het}}(a).
\label{eq:realized-het-const-coeff-3}\end{equation}
If $\underline A^{\mathrm{het}}m^{\mathrm{het}\,2}
-
\beta^{\mathrm{het}}m^{\mathrm{het}}
-
d^{\mathrm{het}}(a)
\ge 0$ and $2\underline A^{\mathrm{het}}m^{\mathrm{het}}\le 1+\beta^{\mathrm{het}}$,
then, for every $T\ge 0$,
\begin{equation}
g_T\le U_T
\le
m^{\mathrm{het}}
+
\frac{1}{((U_0-m^{\mathrm{het}})_+)^{-1}+\underline A^{\mathrm{het}}T}.
\label{eq:realized-het-rate}\end{equation}
If, in addition, \eqref{eq:realized-het-positive-v} holds, then
\[
V_2(a^\star)=\left(\sum_{i=1}^n \frac{H_i b_i}{v_i^2}\right)^{-1},
\quad
V_1(a^\star)=
n\left(\sum_{i=1}^n \frac{H_i b_i}{v_i^2}\right)^{-1},
\]
and therefore
\begin{equation}
d^{\mathrm{het}}(a^\star)
=
\gamma^{\mathrm{het}}\bar Q
+
\left(
64n\frac{\vartheta^3}{L}
+
16\frac{\Lambda\vartheta^2}{L^2}
\right)
\left(\sum_{i=1}^n \frac{H_i b_i}{v_i^2}\right)^{-1}.
\label{eq:realized-het-rate-astar}\end{equation}
\end{corollary}
\subsection{Homogeneous no-control-variate post-local controller}

For this branch we assume only a homogeneous post-local oracle along the plain local paths: for every node $i$ and local step $\ell$, with $\eta_{i,t}=\vartheta_t/(L H_i)$ and $y_{i,t}^{(\ell+1)}=y_{i,t}^{(\ell)}-\eta_{i,t}g_{i,t,\ell}$,
\begin{equation}
\EE[g_{i,t,\ell}\mid \cG_{i,t,\ell}]=\nabla F(y_{i,t}^{(\ell)}),\qquad \EE\!\left[\norm{g_{i,t,\ell}-\nabla F(y_{i,t}^{(\ell)})}^2\middle|\cG_{i,t,\ell}\right]\le \frac{v_i^2}{b_i}.
\label{eq:postlocal-hom-oracle-main}\end{equation}
This is implied by the stronger condition $F_i\equiv F$ and also by equal-size IID batching from a common finite sum; the appendix records the reduction explicitly.

Define $\Delta_{i,t}:=y_{i,t}^{(H_i)}-x_t$ and
\begin{equation}
g_{i,t}^{\mathrm{loc}}:=-(\eta_{i,t}H_i)^{-1}\Delta_{i,t},\qquad \bar g_t:=n^{-1}\sum_i g_{i,t}^{\mathrm{loc}}.
\label{eq:realized-hom-gbar}
\end{equation}
and
\[
\Psi_t^{\mathrm{hom}}(w):=\ip{\bar g_t}{\sum_{i=1}^n w_i\Delta_{i,t}}+\frac{\Lambda_t}{2}\norm{\sum_{i=1}^n w_i\Delta_{i,t}}^2,\qquad w_t^{\mathrm{hom}}\in\arg\min_{w\in\Delta_n}\Psi_t^{\mathrm{hom}}(w),
\]
with $x_{t+1}=x_t+\sum_{i=1}^n w_{i,t}^{\mathrm{hom}}\Delta_{i,t}$. Also write $\bar V_1:=n^{-1}\sum_i v_i^2/(H_i b_i)$ and $V_u:=n^{-2}\sum_i v_i^2/(H_i b_i)$, and define
\[
A_t^{\mathrm{hom}}:=\frac{\vartheta_t}{4LR^2},\qquad \underline A_t^{\mathrm{hom}}:=\frac{A_t^{\mathrm{hom}}}{1+A_t^{\mathrm{hom}}\bar f}=\frac{\vartheta_t}{4LR^2(1+\vartheta_t/8)},
\]
\[
\beta_t^{\mathrm{hom}}:=16\vartheta_t^3+16\frac{L\vartheta_t^2}{\Lambda_t-L},\qquad \delta_t^{\mathrm{hom}}:=\left(4\frac{\vartheta_t^3}{L}+4\frac{\vartheta_t^2}{\Lambda_t-L}\right)\bar V_1+\left(\frac{\vartheta_t}{L}+\frac{1}{\Lambda_t-L}\right)V_u.
\]
\begin{algorithm}[ht]
\caption{Homogeneous Post-local HEW-Local SGD}
\label{alg:realized-hom}
\begin{algorithmic}[1]
\Require $x_t$, amplitude $\vartheta_t$, curvature $\Lambda_t>L$
\ForAll{$i\in[n]$ in parallel}
    \State Set $\eta_{i,t}\gets \vartheta_t/(L H_i)$ and $y_{i,t}^{(0)}\gets x_t$
    \For{$\ell=0,\dots,H_i-1$}
        \State Draw minibatch $\mathcal B_{i,t,\ell}$ and compute $g_{i,t,\ell}$
        \State Update $y_{i,t}^{(\ell+1)}
        \gets
        y_{i,t}^{(\ell)}-\eta_{i,t}g_{i,t,\ell}$
    \EndFor
    \State Set $\Delta_{i,t}\gets y_{i,t}^{(H_i)}-x_t$
\EndFor
\State Form $g_{i,t}^{\mathrm{loc}}
\gets
-\frac{1}{\eta_{i,t}H_i}\Delta_{i,t}$,
$\bar g_t\gets \frac1n\sum_{i=1}^n g_{i,t}^{\mathrm{loc}}$
\State Solve
\[
w_t^{\mathrm{hom}}
\in
\argmin_{w\in\Delta_n}
\left\{
\ip{\bar g_t}{\sum_{i=1}^n w_i\Delta_{i,t}}
+\frac{\Lambda_t}{2}\norm{\sum_{i=1}^n w_i\Delta_{i,t}}^2
\right\}
\]
\State Set $x_{t+1}\gets x_t+\sum_{i=1}^n w_{i,t}^{\mathrm{hom}}\Delta_{i,t}$
\end{algorithmic}
\end{algorithm}

\begin{theorem}[Direct Bellman inequality: homogeneous no-control-variate post-local controller]
\label{th:realized-hom-bellman}
Assume deterministic schedules, $L<\Lambda_t\le 2L$, and $\Lambda_t\vartheta_t\le L/2$ for every $t\ge 0$. Then
\begin{equation}
g_{t+1}
\le
T_{A_t^{\mathrm{hom}}}(g_t)
+
\beta_t^{\mathrm{hom}}g_t
+
\delta_t^{\mathrm{hom}}.
\label{eq:realized-hom-bellman}\end{equation}
\end{theorem}

\begin{corollary}[Convex rate with unequal local horizons: homogeneous no-control-variate post-local controller]
\label{cor:realized-hom-rate}
Under the hypotheses of Theorem~\ref{th:realized-hom-bellman},
\begin{equation}
g_{t+1}
\le
g_t-\underline A_t^{\mathrm{hom}}g_t^2+\beta_t^{\mathrm{hom}}g_t+\delta_t^{\mathrm{hom}}.
\label{eq:realized-hom-quadratic-linear}\end{equation}
If, in addition, \eqref{eq:realized-het-const-param},
define
\[
\textstyle\underline A^{\mathrm{hom}}
:=
\frac{\vartheta}{4LR^2(1+\vartheta/8)},
\qquad
\beta^{\mathrm{hom}}
:=
16\vartheta^3+16\frac{L\vartheta^2}{\Lambda-L},
\]
\[
\delta^{\mathrm{hom}}
:=
\left(
4\frac{\vartheta^3}{L}
+
4\frac{\vartheta^2}{\Lambda-L}
\right)\bar V_1
+
\left(
\frac{\vartheta}{L}
+
\frac{1}{\Lambda-L}
\right)V_u.
\]
If $\underline A^{\mathrm{hom}}m^{\mathrm{hom}\,2}
-
\beta^{\mathrm{hom}}m^{\mathrm{hom}}
-
\delta^{\mathrm{hom}}
\ge 0$ and $2\underline A^{\mathrm{hom}}m^{\mathrm{hom}}\le 1+\beta^{\mathrm{hom}}$,
then, for every $T\ge 0$,
\begin{equation}
g_T
\le
m^{\mathrm{hom}}
+
\frac{1}{((g_0-m^{\mathrm{hom}})_+)^{-1}+\underline A^{\mathrm{hom}}T}.
\label{eq:realized-hom-rate}\end{equation}
\end{corollary}

\begin{corollary}[PL rate with unequal local horizons: homogeneous no-control-variate post-local controller]
\label{cor:realized-hom-PL}
Under the homogeneous post-local oracle condition \eqref{eq:postlocal-hom-oracle-main}, deterministic schedules, $L<\Lambda_t\le 2L$, $\Lambda_t\vartheta_t\le L/2$, and the PL condition $\norm{\nabla F(x)}^2\ge 2\mu(F(x)-F_\star)$ on $B(x_\star,R)$,
\begin{equation}
\textstyle g_{t+1}
\le
(1-\rho_t^{\mathrm{hom}})g_t+\delta_t^{\mathrm{hom}},
\qquad
\rho_t^{\mathrm{hom}}:=\frac{\mu\vartheta_t}{2L}-\beta_t^{\mathrm{hom}}.
\label{eq:realized-hom-PL-rec}\end{equation}
If, in addition, \eqref{eq:realized-het-const-param} holds and
\begin{equation}
\textstyle\rho^{\mathrm{hom}}
:=
\frac{\mu\vartheta}{2L}-\beta^{\mathrm{hom}}>0,
\label{eq:realized-hom-rho}\end{equation}
then
\begin{equation}
g_t
\le
(1-\rho^{\mathrm{hom}})^t g_0+\frac{\delta^{\mathrm{hom}}}{\rho^{\mathrm{hom}}}.
\label{eq:realized-hom-PL-rate}\end{equation}
\end{corollary}

\section{Experiments}

We evaluate on \textsc{Covertype} \cite{dua2017covertype} and \textsc{MNIST} \cite{lecun2010mnist} using a linear softmax model with $\ell_2$-regularized cross-entropy. Each dataset is split 80/20 into train and test sets, standardized, augmented with a bias feature. All methods start from the zero model. We simulate $n=20$ clients. In the homogeneous regimes, the training set is randomly shuffled and evenly split across clients. In the heterogeneous regimes, we use a class-wise Dirichlet partition with concentration $\alpha=0.2$. For equal horizons, $H=4$. For unequal horizons, each client is assigned a fixed value from $\{1,2,4,8\}$ uniformly at random. All rounds use full participation and minibatch size $32$. Stepsizes are normalized by an empirical estimate of the smoothness constant $L$. We run $90$ rounds and report averages over $7$ seeds. We compare \textsc{HEW}, \textsc{HEW-Fixed} against \textsc{Uniform-LocalSGD}, \textsc{FedAvg}, \textsc{FedNova}; in the equal-horizon setting we additionally include \textsc{SCAFFOLD}, \textsc{FedProx}, minibatch SGD. Hyperparameters are tuned separately for each dataset--regime pair by a short $20$-round sweep over the first three seeds. For \textsc{HEW}, \textsc{HEW-Fixed}, we tune $\vartheta \in \{0.5,1,2,4,8\}$ and $\Lambda/L \in \{1.1,1.25,1.5,1.75,2.0\}$; for the remaining methods, we tune the normalized stepsize scale in $\{0.2,0.4,0.8,1.6\}$, for \textsc{FedProx} also the proximal parameter in $\{0.01,0.1\}$. We use cumulative transmitted scalars as the communication budget and plot test accuracy and training objective against this axis (Figure~\ref{fig:main-exp-grid}). In the unequal-horizon regimes, we report the final \textsc{HEW} weight mass grouped by horizon $H$ (Figure~\ref{fig:hew-weight-mass}). 
% The code is at \url{https://anonymous.4open.science/r/Heterogeneous-Horizon-Exact-Weight-Local-SGD-BBD5}. More experiments are in Appendix~\ref{app:text_experiments}.
The code is at \url{https://github.com/dmivilensky/Heterogeneous-Horizon-Exact-Weight-Local-SGD}. More experiments are in Appendix~\ref{app:text_experiments}.

Figure~\ref{fig:main-exp-grid} shows that the benefit of adaptive weighting is regime-dependent. On \textsc{Covertype}, HEW clearly improves over both \textsc{HEW-Fixed} and all baselines only in the homogeneous random-$H$ setting; no similarly clear advantage is visible in the other \textsc{Covertype} regimes. On \textsc{MNIST}, the corresponding effect appears only in the heterogeneous random-$H$ setting, where HEW again dominates both \textsc{HEW-Fixed} and the baseline methods. This suggests that HEW is effective only when heterogeneity induces a useful ordering of post-local endpoints, rather than merely adding dispersion. On \textsc{Covertype}, this structure is exposed already by unequal local horizons, whereas on \textsc{MNIST} it becomes informative only in combination with statistical heterogeneity.
\newline
\section{Conclusion}

We studied adaptive aggregation for heterogeneous local SGD with unequal local horizons. On the theory side, we derived an exact one-step control formulation and global convergence guarantees under full participation, including convex, PL, and higher-order regimes. On the algorithmic side, we introduced post-local variants that retain the same endpoint-based viewpoint while avoiding the conservatism of the original local-control surrogate. The experiments indicate that adaptive weighting is most useful when heterogeneity induces a meaningful ordering of post-local client endpoints. 

\begin{figure*}[ht!]
\centering
\setlength{\tabcolsep}{2pt}

\begin{tabular}{ccc}
\makecell[c]{\footnotesize Homogeneous\\ \footnotesize equal $H$} &
\makecell[c]{\footnotesize Homogeneous\\ \footnotesize random $H$} &
\makecell[c]{\footnotesize Heterogeneous\\ \footnotesize random $H$} \\[-1mm]

\includegraphics[width=0.325\textwidth]{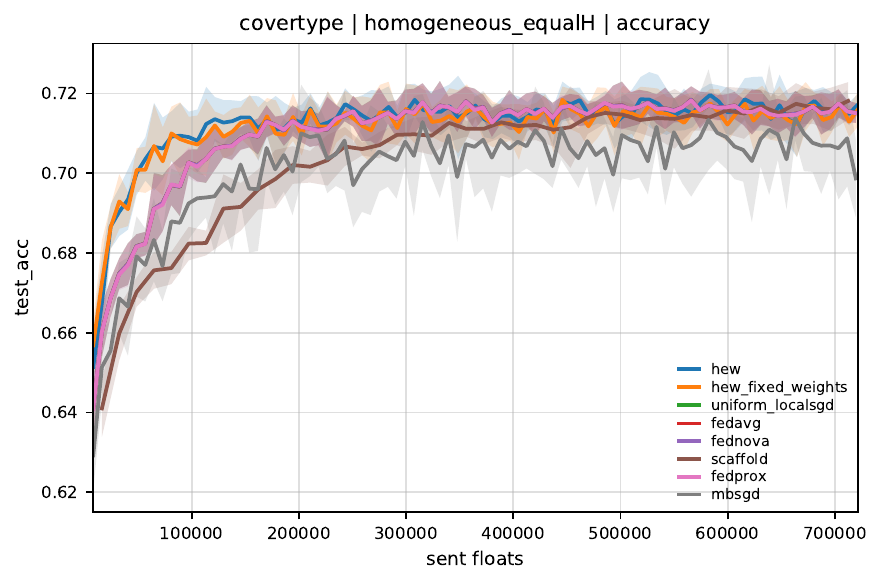} &
\includegraphics[width=0.325\textwidth]{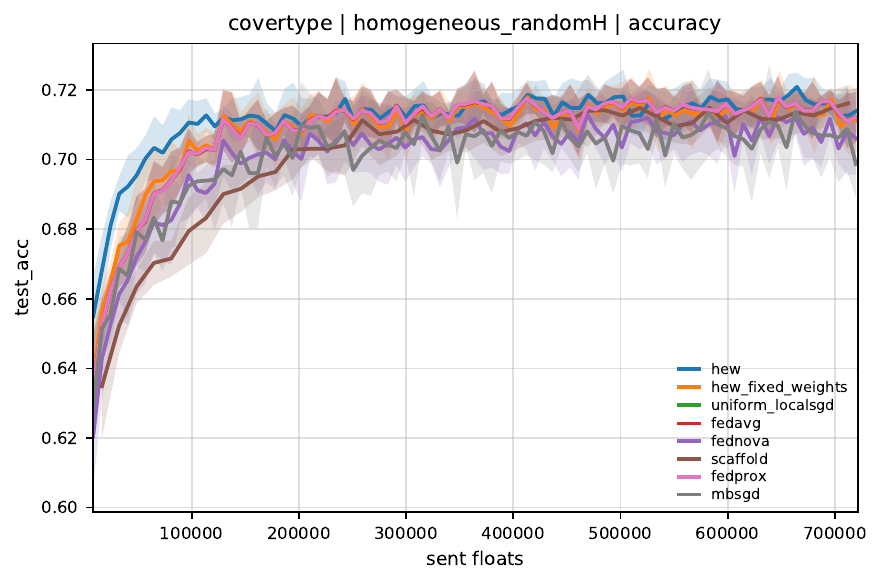} &
\includegraphics[width=0.325\textwidth]{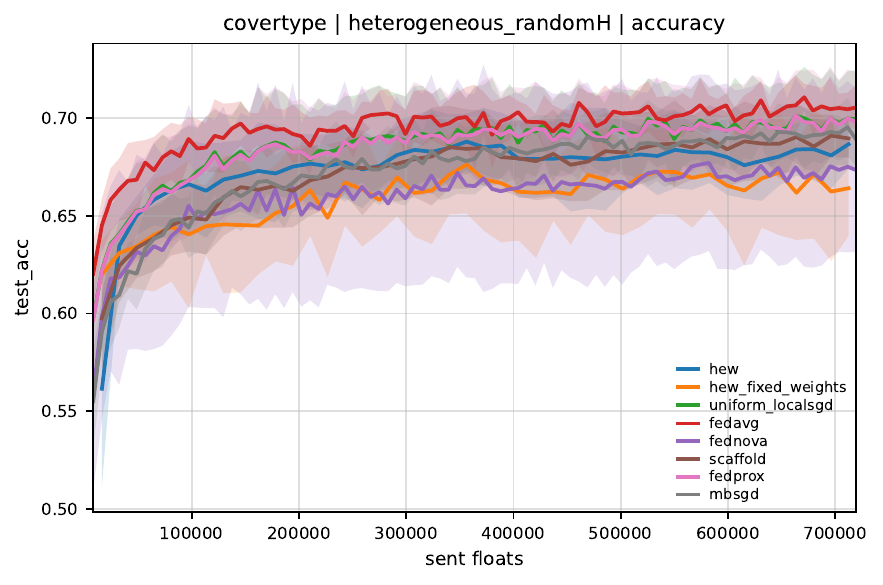} \\[-1mm]

\includegraphics[width=0.325\textwidth]{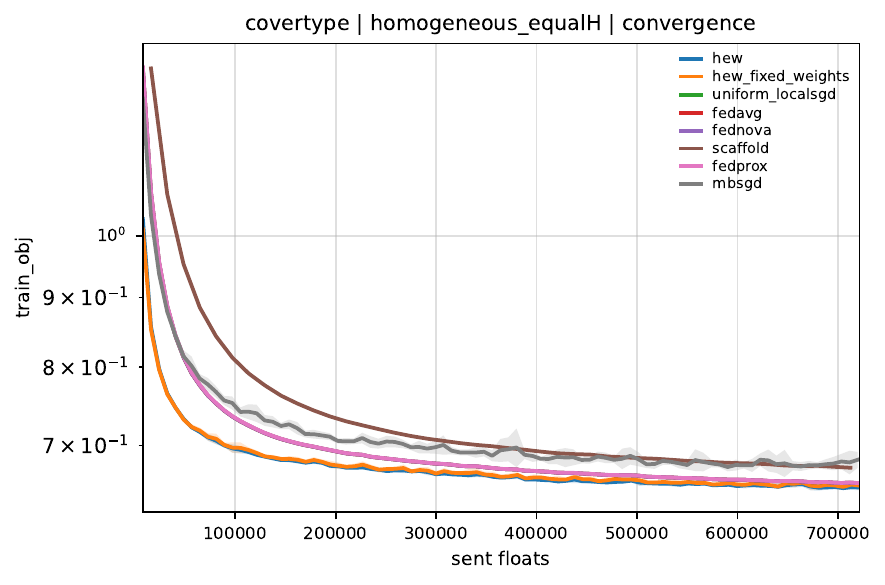} &
\includegraphics[width=0.325\textwidth]{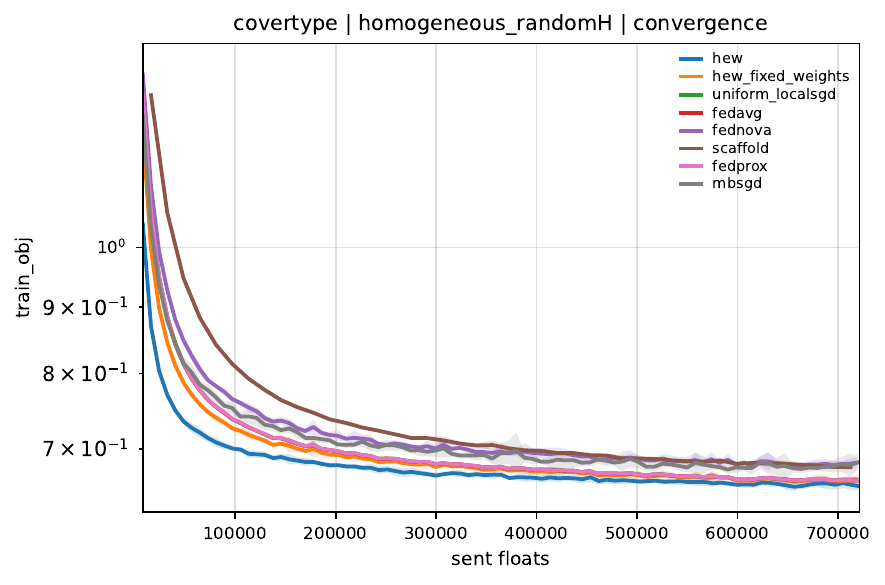} &
\includegraphics[width=0.325\textwidth]{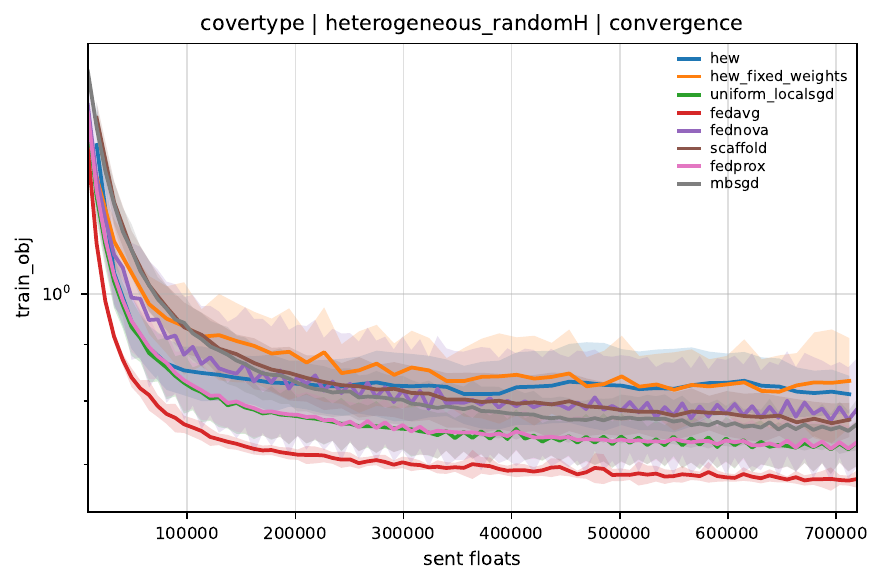} \\[-1mm]

\includegraphics[width=0.325\textwidth]{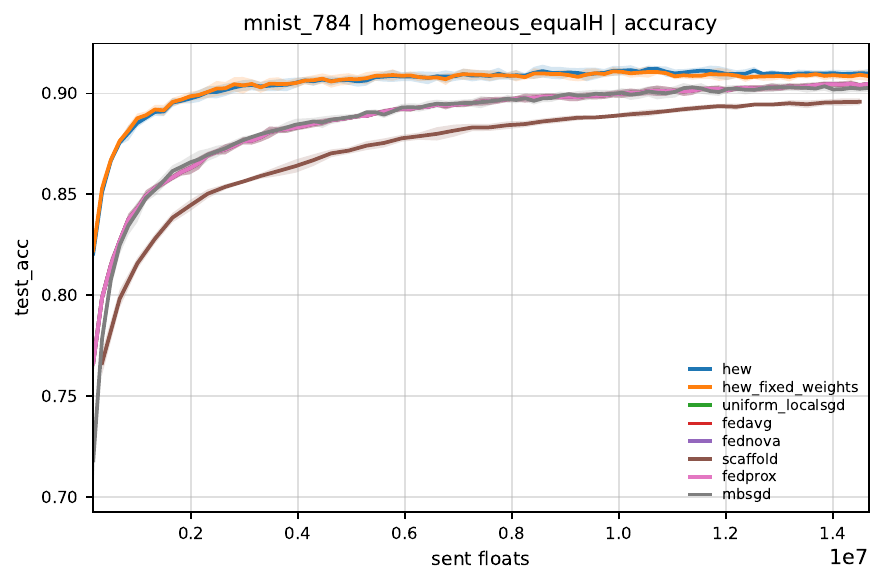} &
\includegraphics[width=0.325\textwidth]{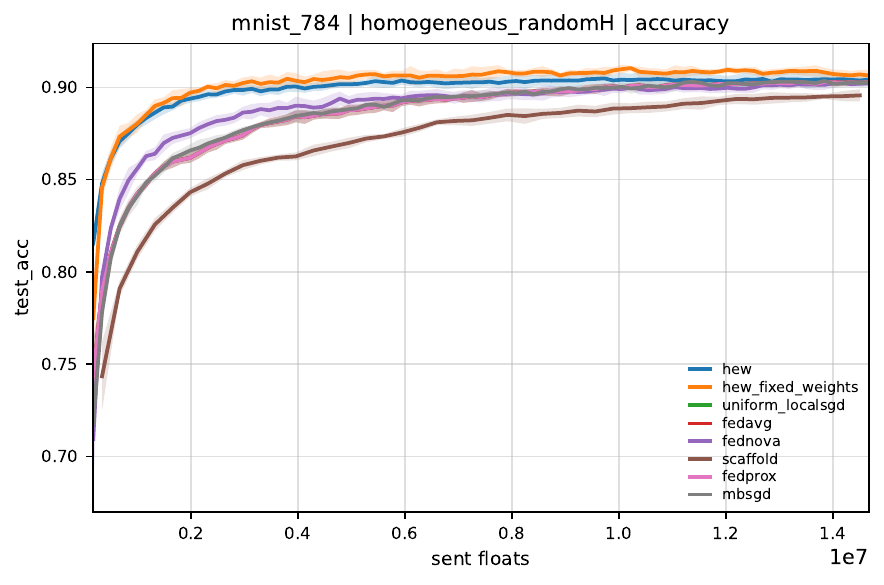} &
\includegraphics[width=0.325\textwidth]{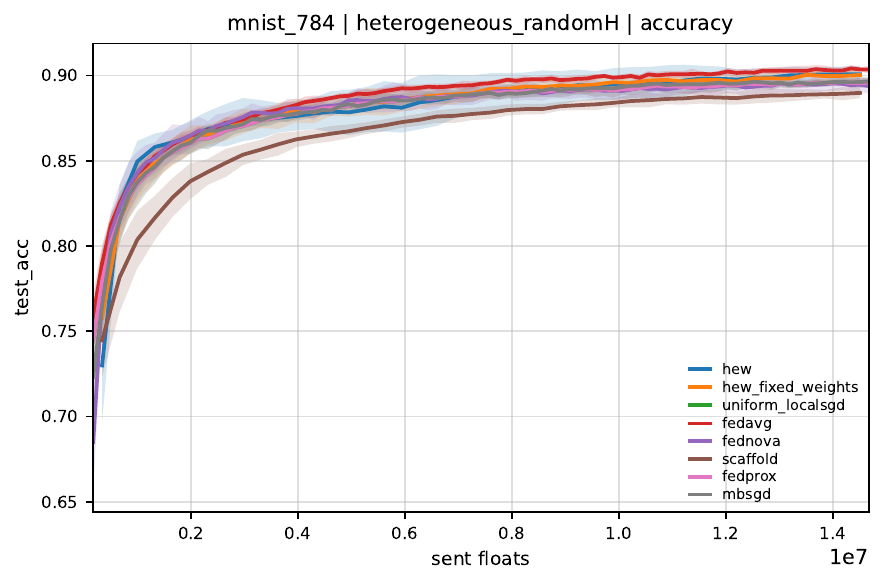} \\[-1mm]

\includegraphics[width=0.325\textwidth]{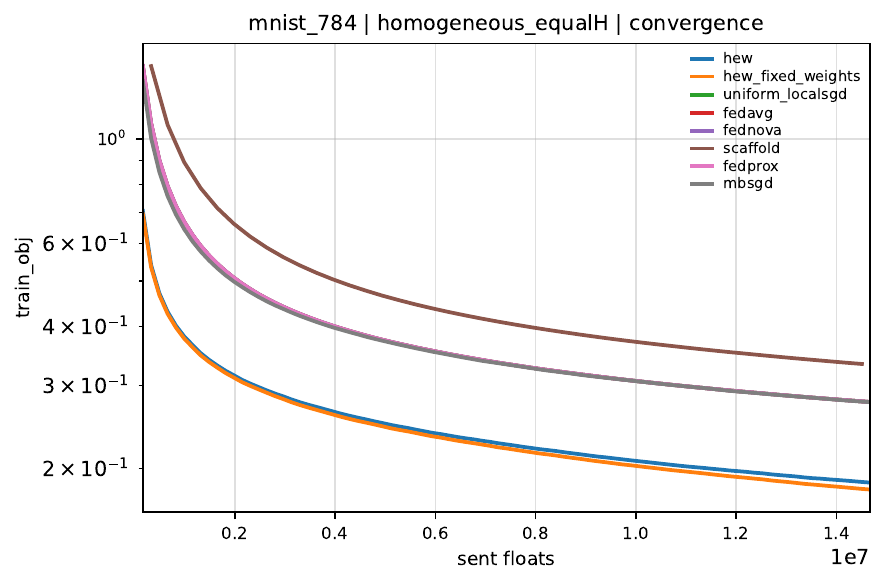} &
\includegraphics[width=0.325\textwidth]{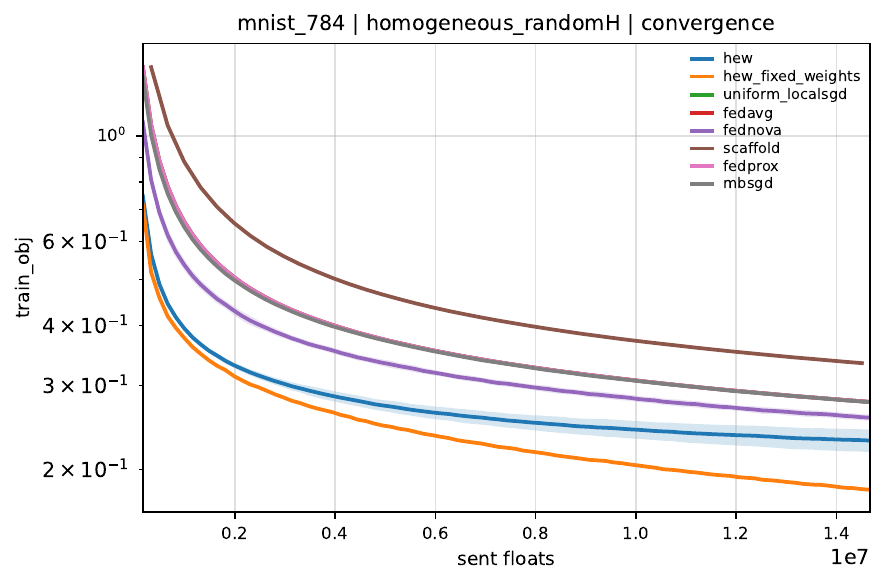} &
\includegraphics[width=0.325\textwidth]{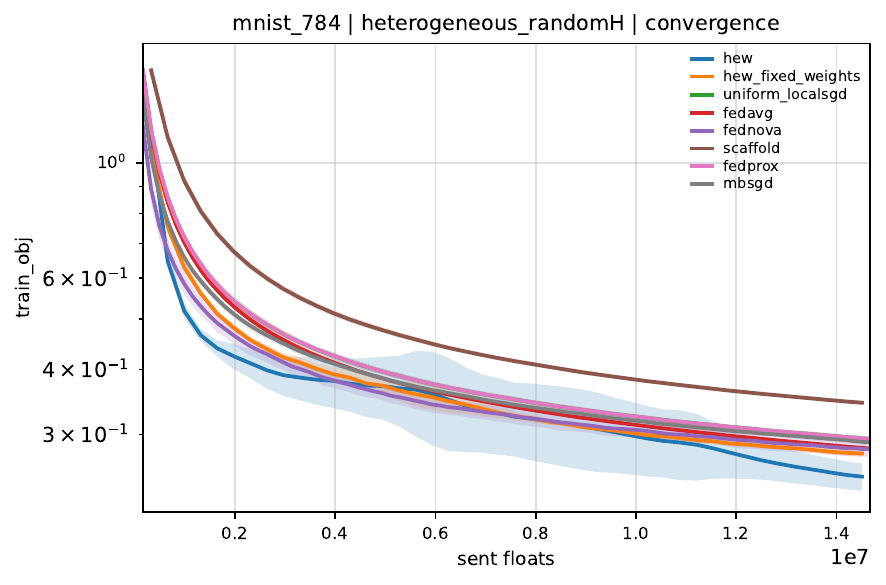}
\end{tabular}

\caption{
Communication--performance across three client regimes: homogeneous with equal local horizons, homogeneous with random local horizons, and heterogeneous with random local horizons. The first two rows correspond to Covertype and the last two to MNIST; within each dataset, the upper row reports test accuracy and the lower row reports training gap. The figure compares HEW, its fixed-weight variant, and standard local/federated baselines under matched communication budgets.
}
\label{fig:main-exp-grid}
\end{figure*}

\begin{figure*}[ht!]
\centering
\setlength{\tabcolsep}{2pt}

\begin{tabular}{cccc}
\makecell[c]{\footnotesize Covertype\\ \footnotesize Homogeneous random $H$} &
\makecell[c]{\footnotesize Covertype\\ \footnotesize Heterogeneous random $H$} &
\makecell[c]{\footnotesize MNIST\\ \footnotesize Homogeneous random $H$} &
\makecell[c]{\footnotesize MNIST\\ \footnotesize Heterogeneous random $H$} \\[-1mm]

\includegraphics[width=0.242\textwidth]{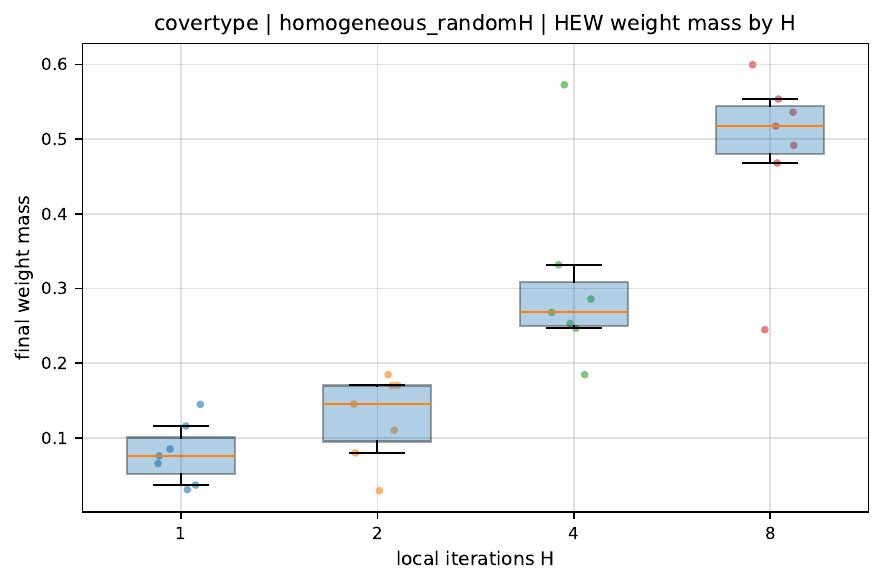} &
\includegraphics[width=0.242\textwidth]{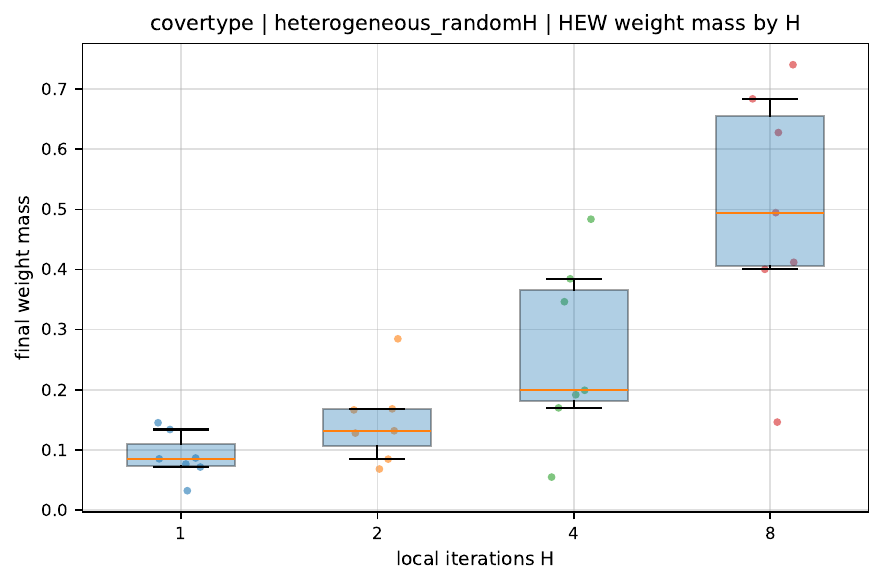} &
\includegraphics[width=0.242\textwidth]{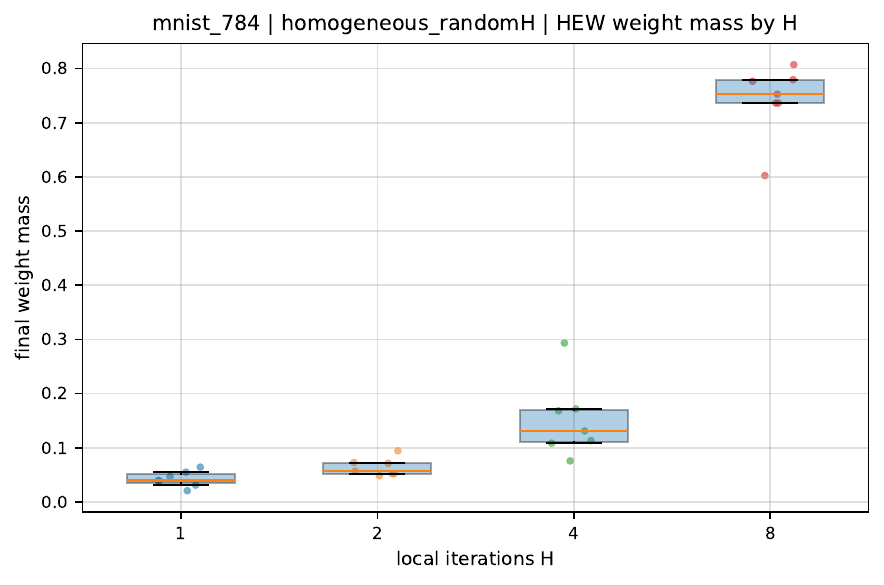} &
\includegraphics[width=0.242\textwidth]{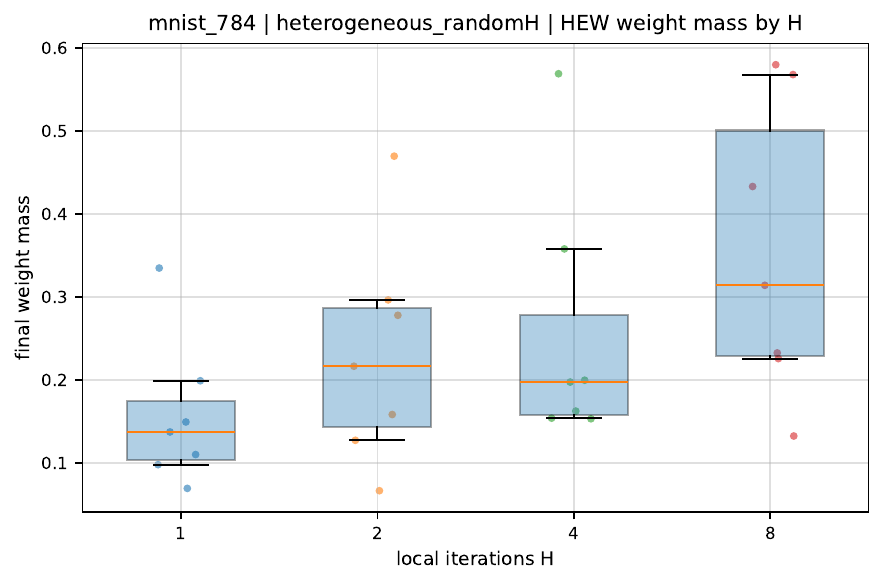}
\end{tabular}

\caption{
HEW aggregation weight mass grouped by local horizon $H$ in the random-horizon regimes. Across both datasets, the learned weights are visibly non-uniform across horizon groups, as HEW adapts aggregation to the realized distribution of local work.
}
\label{fig:hew-weight-mass}
\end{figure*}

\newpage
\bibliographystyle{plainnat}
\bibliography{references}

@article{fashionmnist,
  author       = {Han Xiao and Kashif Rasul and Roland Vollgraf},
  title        = {Fashion-MNIST: a Novel Image Dataset for Benchmarking Machine Learning Algorithms},
  journal      = {CoRR},
  volume       = {abs/1708.07747},
  year         = {2017},
  url          = {http://arxiv.org},
}

@article{lecun2010mnist,
  title={MNIST handwritten digit database},
  author={LeCun, Yann and Cortes, Corinna and Burges, CJ},
  journal={ATT Labs [Online]. Available: http://yann.lecun.com/exdb/mnist},
  volume={2},
  year={2010}
}

@misc{dua2017covertype,
  author = {Blackard, Jock A. and Dean, Denis J. and Anderson, Charles W.},
  title = {Covertype Data Set},
  year = {1998},
  url = {https://archive.ics.uci.edu/ml/datasets/covertype},
  institution = {UCI Machine Learning Repository}
}

@book{Bellman,
  author    = {Edwin F. Beckenbach and Richard Bellman},
  title     = {Inequalities},
  publisher = {Springer-Verlag},
  address   = {Berlin},
  year      = {1961},
  series    = {Ergebnisse der Mathematik und ihrer Grenzgebiete. Neue Folge},
  volume    = {30},
  doi       = {10.1007/978-3-642-64971-4}
}

@inproceedings{mcmahan2017communication,
  author    = {Brendan McMahan and Eider Moore and Daniel Ramage and Seth Hampson and Blaise Aguera y Arcas},
  title     = {Communication-Efficient Learning of Deep Networks from Decentralized Data},
  booktitle = {Proceedings of the 20th International Conference on Artificial Intelligence and Statistics},
  series    = {Proceedings of Machine Learning Research},
  volume    = {54},
  pages     = {1273--1282},
  editor    = {Aarti Singh and Jerry Zhu},
  year      = {2017},
  publisher = {PMLR}
}

@inproceedings{li2020federated,
  author    = {Tian Li and Anit Kumar Sahu and Manzil Zaheer and Maziar Sanjabi and Ameet Talwalkar and Virginia Smith},
  title     = {Federated Optimization in Heterogeneous Networks},
  booktitle = {Proceedings of Machine Learning and Systems},
  volume    = {2},
  pages     = {429--450},
  editor    = {I. Dhillon and D. Papailiopoulos and V. Sze},
  year      = {2020}
}

@inproceedings{karimireddy2020scaffold,
  author    = {Sai Praneeth Karimireddy and Satyen Kale and Mehryar Mohri and Sashank Reddi and Sebastian Stich and Ananda Theertha Suresh},
  title     = {{SCAFFOLD}: Stochastic Controlled Averaging for Federated Learning},
  booktitle = {Proceedings of the 37th International Conference on Machine Learning},
  series    = {Proceedings of Machine Learning Research},
  volume    = {119},
  pages     = {5132--5143},
  editor    = {Hal Daum{\'e} III and Aarti Singh},
  year      = {2020},
  publisher = {PMLR}
}

@inproceedings{wang2020tackling,
  author    = {Jianyu Wang and Qinghua Liu and Hao Liang and Gauri Joshi and H. Vincent Poor},
  title     = {Tackling the Objective Inconsistency Problem in Heterogeneous Federated Optimization},
  booktitle = {Advances in Neural Information Processing Systems 33},
  editor    = {Hugo Larochelle and Marc'Aurelio Ranzato and Raia Hadsell and Maria-Florina Balcan and Hsuan-Tien Lin},
  year      = {2020}
}

@inproceedings{gorbunov2021local,
  author    = {Eduard Gorbunov and Filip Hanzely and Peter Richtarik},
  title     = {Local SGD: Unified Theory and New Efficient Methods},
  booktitle = {Proceedings of the 24th International Conference on Artificial Intelligence and Statistics},
  series    = {Proceedings of Machine Learning Research},
  volume    = {130},
  pages     = {3556--3564},
  editor    = {Arindam Banerjee and Kenji Fukumizu},
  year      = {2021},
  publisher = {PMLR}
}

@inproceedings{luo2025revisiting,
  author    = {Ruichen Luo and Sebastian U. Stich and Samuel Horv{\'a}th and Martin Tak{\'a}{\v{c}}},
  title     = {Revisiting LocalSGD and SCAFFOLD: Improved Rates and Missing Analysis},
  booktitle = {Proceedings of the 28th International Conference on Artificial Intelligence and Statistics},
  series    = {Proceedings of Machine Learning Research},
  volume    = {258},
  pages     = {2539--2547},
  editor    = {Yingzhen Li and Stephan Mandt and Shipra Agrawal and Emtiyaz Khan},
  year      = {2025},
  publisher = {PMLR}
}

@misc{mangold2025scaffold,
  author       = {Paul Mangold and Alain Durmus and Aymeric Dieuleveut and Eric Moulines},
  title        = {Scaffold with Stochastic Gradients: New Analysis with Linear Speed-Up},
  year         = {2025},
  eprint       = {2503.07594},
  archivePrefix= {arXiv}
}
\newpage
\appendix
\section{Scalar lemmas and deterministic caps}

\begin{lemma}[Gradient-gap inequalities and deterministic cap]
\label{lem:gap}
Under Assumption~\ref{ass:model},
\begin{equation}
f_t\le R\norm{\nabla F(x_t)},
\qquad
\norm{\nabla F(x_t)}^2\ge \frac{f_t^2}{R^2}.
\label{eq:gap-lower-app}\end{equation}
Moreover, for every $x\in B(x_\star,R)$,
\begin{equation}
\norm{\nabla F(x)}^2\le 2L(F(x)-F_\star),
\label{eq:gap-upper-app}\end{equation}
and
\begin{equation}
F(x)-F_\star\le \frac{L}{2}\norm{x-x_\star}^2\le \frac{LR^2}{2}=\bar f.
\label{eq:gap-cap-app}\end{equation}
\end{lemma}

\begin{proof}
By convexity of $F$,
\[
F(x_t)-F_\star
\le
\ip{\nabla F(x_t)}{x_t-x_\star}
\le
\norm{\nabla F(x_t)}\norm{x_t-x_\star}
\le
R\norm{\nabla F(x_t)},
\]
which proves \eqref{eq:gap-lower-app}. For the upper gradient bound, define $
y:=x-\frac1L\nabla F(x)$.
By $L$-smoothness,
\[
F(y)\le F(x)+\ip{\nabla F(x)}{y-x}+\frac{L}{2}\norm{y-x}^2
=
F(x)-\frac1{2L}\norm{\nabla F(x)}^2.
\]
Since $F_\star\le F(y)$, rearranging proves \eqref{eq:gap-upper-app}.
Finally, because $x_\star$ minimizes the differentiable convex function $F$,
$\nabla F(x_\star)=0$. Applying $L$-smoothness between $x$ and $x_\star$ gives
\[
F(x)\le F(x_\star)+\ip{\nabla F(x_\star)}{x-x_\star}+\frac{L}{2}\norm{x-x_\star}^2
=
F_\star+\frac{L}{2}\norm{x-x_\star}^2.
\]
The ball condition then implies \eqref{eq:gap-cap-app}.
\end{proof}

\begin{definition}[$T_a$ operator]
For $a\ge 0$ and $u\ge 0$, define
\begin{equation}
T_a(u):=\frac{u}{1+a u}.\end{equation}
\end{definition}

\begin{lemma}[Quadratic descent is majorized by $T_a$]
\label{lem:Ta-majorizes}
For every $u,a\ge 0$,
\[
u-a u^2\le T_a(u).\]
\end{lemma}

\begin{proof}
\[
T_a(u)-(u-a u^2)
=
\frac{u}{1+a u}-u+a u^2
=
\frac{a^2u^3}{1+a u}\ge 0.
\]
\end{proof}

\begin{lemma}[Composition]
\label{lem:Ta-comp}
For all $a_1,a_2,u\ge 0$,
\[
T_{a_2}(T_{a_1}(u))=T_{a_1+a_2}(u).\]
\end{lemma}

\begin{proof}
\[
T_{a_2}(T_{a_1}(u))
=
\frac{u/(1+a_1u)}{1+a_2u/(1+a_1u)}
=
\frac{u}{1+(a_1+a_2)u}.
\]
\end{proof}

\begin{lemma}[$T_a$ is increasing and $1$-Lipschitz]
\label{lem:Ta-Lip}
For every $a\ge 0$, the map $T_a$ is increasing on $\RR_+$ and
\[
|T_a(u)-T_a(v)|\le |u-v|
\qquad
\text{for all }u,v\ge 0.\]
Hence
\[
T_a(u+b)\le T_a(u)+b
\qquad
\text{for all }u,b\ge 0.\]
\end{lemma}

\begin{proof}
\[
T_a'(u)=\frac{1}{(1+a u)^2},
\]
so $0\le T_a'(u)\le 1$.
\end{proof}

\begin{lemma}[Branch telescoping]
\label{lem:telescope}
Suppose a nonnegative sequence $(z_\ell)_{\ell=0}^H$ satisfies
\[
z_{\ell+1}\le T_{a_\ell}(z_\ell)+b_\ell,
\qquad
a_\ell,b_\ell\ge 0,
\qquad
\ell=0,\dots,H-1.\]
Then
\[
z_H
\le
T_{\sum_{\ell=0}^{H-1}a_\ell}(z_0)
+
\sum_{\ell=0}^{H-1}b_\ell.\]
\end{lemma}

\begin{proof}
We argue by induction on $H$. The case $H=1$ is immediate. Assume the claim
holds for some $H\ge 1$. Then
\[
z_{H+1}\le T_{a_H}(z_H)+b_H.
\]
By the induction hypothesis,
\[
z_H
\le
T_{\sum_{\ell=0}^{H-1}a_\ell}(z_0)
+
\sum_{\ell=0}^{H-1}b_\ell.
\]
Using Lemma~\ref{lem:Ta-Lip},
\[
T_{a_H}(z_H)
\le
T_{a_H}\!\left(T_{\sum_{\ell=0}^{H-1}a_\ell}(z_0)\right)
+
\sum_{\ell=0}^{H-1}b_\ell.
\]
Lemma~\ref{lem:Ta-comp} yields
\[
z_{H+1}
\le
T_{\sum_{\ell=0}^{H}a_\ell}(z_0)+\sum_{\ell=0}^{H}b_\ell.
\]
\end{proof}

\begin{lemma}[Quadratic-linear scalar recursion]
\label{lem:scalar}
Let $(x_t)_{t\ge 0}$ be a nonnegative sequence satisfying
\[
x_{t+1}\le x_t-a x_t^2+\beta x_t+\delta,
\qquad
a>0,\quad \beta\ge 0,\quad \delta\ge 0.\]
Let $m\ge 0$ satisfy
\begin{equation}
a m^2-\beta m-\delta\ge 0
\label{eq:scalar-super-app}\end{equation}
and
\begin{equation}
2am\le 1+\beta.
\label{eq:scalar-safe-app}\end{equation}
Then, for every $T\ge 0$,
\[
x_T
\le
m+\frac{1}{((x_0-m)_+)^{-1}+aT},\]
with the convention $1/\infty=0$. In particular,
\[
x_T\le m+\frac{1}{x_0^{-1}+aT}.\]
\end{lemma}

\begin{proof}
Define
\[
h(x):=x-a x^2+\beta x+\delta.
\]
Condition \eqref{eq:scalar-super-app} is equivalent to $h(m)\le m$. Moreover,
\[
h'(x)=1+\beta-2ax,
\]
and \eqref{eq:scalar-safe-app} implies $h'(x)\ge 0$ for all $x\in[0,m]$.
Hence $h$ is increasing on $[0,m]$, so
\begin{equation}
x\in[0,m]\quad\Longrightarrow\quad h(x)\le h(m)\le m.
\label{eq:h-below-m-app}\end{equation}
Define $
y_t:=(x_t-m)_+$.
We claim that
\begin{equation}
y_{t+1}\le y_t-a y_t^2.
\label{eq:y-rec-app}\end{equation}
If $x_t\le m$, then $y_t=0$, and \eqref{eq:h-below-m-app} gives
$x_{t+1}\le m$, hence $y_{t+1}=0$. So \eqref{eq:y-rec-app} holds.
Assume now that $x_t>m$. Then $y_t=x_t-m>0$, and
\begin{align*}
x_{t+1}-m
&\le
x_t-m-a(x_t^2-m^2)+\beta(x_t-m)+(h(m)-m)\\
&\le
(x_t-m)\bigl(1+\beta-a(x_t+m)\bigr),
\end{align*}
because $h(m)-m\le 0$. Since $x_t=m+y_t$,
\[
1+\beta-a(x_t+m)
=
1+\beta-2am-a y_t
\le
1-a y_t
\]
by \eqref{eq:scalar-safe-app}. Therefore
\[
x_{t+1}-m\le y_t(1-a y_t)=y_t-a y_t^2.
\]
Taking positive parts proves \eqref{eq:y-rec-app}. By
Lemma~\ref{lem:Ta-majorizes},
\[
y_{t+1}\le T_a(y_t).
\]
Iterating with Lemma~\ref{lem:Ta-comp}, $
y_T\le T_{aT}(y_0)=\frac{1}{y_0^{-1}+aT}$.
Hence
\[
x_T\le m+y_T\le m+\frac{1}{((x_0-m)_+)^{-1}+aT}.
\]
Finally, $(x_0-m)_+\le x_0$ implies the simpler bound.
\end{proof}

\begin{lemma}[Linear scalar recursion]
\label{lem:lin-scalar}
Let $(x_t)_{t\ge 0}$ be a nonnegative sequence satisfying
\[
x_{t+1}\le (1-a)x_t+\delta,
\qquad
0<a\le 1,\quad \delta\ge 0.\]
If $m\ge \delta/a$, then for every $T\ge 0$, $
x_T\le m+(1-a)^T (x_0-m)_+$.
\end{lemma}

\begin{proof}
If $x_t\le m$, then $
x_{t+1}\le (1-a)m+\delta\le m$.
Thus $m$ is invariant. Define $y_t:=(x_t-m)_+$. If $x_t\le m$, then
$y_{t+1}=0$. If $x_t>m$, then
\[
x_{t+1}-m\le (1-a)(x_t-m)+\delta-am\le (1-a)(x_t-m).
\]
Therefore $y_{t+1}\le (1-a)y_t$. Iterating yields the claim.
\end{proof}

\begin{lemma}[Discrete Gronwall for cumulative sums]
\label{lem:cum-gronwall}
Let $(x_\ell)_{\ell\ge 0}$ and $(a_\ell)_{\ell\ge 0}$ be nonnegative sequences with $a_\ell$ nondecreasing, and let $\beta\ge 0$. If
\[
x_\ell\le a_\ell+\beta\sum_{s=0}^{\ell-1}x_s
\qquad\text{for every }\ell\ge 0,
\]
then for every $\ell\ge 0$, $
x_\ell\le a_\ell(1+\beta)^\ell\le a_\ell e^{\beta\ell}$.
\end{lemma}

\begin{proof}
The proof is by induction. The claim is immediate for $\ell=0$. Assume it holds for all indices $<\ell$. Then
\[
x_\ell\le a_\ell+\beta\sum_{s=0}^{\ell-1}a_s(1+\beta)^s
\le a_\ell\left(1+\beta\sum_{s=0}^{\ell-1}(1+\beta)^s\right)
= a_\ell(1+\beta)^\ell.
\]
The exponential bound follows from $(1+\beta)^\ell\le e^{\beta\ell}$.
\end{proof}

\section{Local-viewpoint setup, deferred statements, and arbitrary-active-set proofs}

This appendix records the detailed assumptions, filtrations, state variables, and deferred theorem statements used by the proofs.

\begin{assumption}[Convex smooth sum-of-sums model with invariant ball]
\label{ass:model}
Let \eqref{eq:model} hold, where each $\phi_{ij}:\RR^d\to\RR$ is convex and $L$-smooth. Assume there exists $x_\star\in\arg\min F$ and $R>0$ such that, almost surely,
\[
x_t\in B(x_\star,R)
\qquad\text{and}\qquad
y_{i,t}^{(\ell)}\in B(x_\star,R)\]
for every round $t$, every active node $i$, and every local step $\ell=0,\dots,H_i$.
\end{assumption}

\begin{assumption}[Minibatch stochastic gradients]
\label{ass:oracle}
For each node $i$, define
\[
v_i^2
:=
\sup_{x\in B(x_\star,R)}
\frac1{m_i}\sum_{j=1}^{m_i}
\norm{\nabla \phi_{ij}(x)-\nabla F_i(x)}^2.\]
At round $t$ and local step $\ell$, active node $i$ draws a minibatch
$\mathcal B_{i,t,\ell}\subseteq[m_i]$ of size $b_i$ and sets
\[
g_{i,t,\ell}
=
\frac1{b_i}\sum_{j\in\mathcal B_{i,t,\ell}}
\nabla \phi_{ij}(y_{i,t}^{(\ell)}).\]
Let $\cG_{i,t,\ell}$ denote the sigma-field generated by $\cF_t$ and by all
minibatches of node $i$ drawn strictly before local step $\ell$ in round $t$.
Conditionally on $\cG_{i,t,\ell}$, the minibatches are independent across active
nodes and across local steps, and
\[
\EE[g_{i,t,\ell}\mid \cG_{i,t,\ell}]
=
\nabla F_i(y_{i,t}^{(\ell)}),\]
\[
\EE\!\left[
\norm{g_{i,t,\ell}-\nabla F_i(y_{i,t}^{(\ell)})}^2
\middle|
\cG_{i,t,\ell}
\right]
\le
\frac{v_i^2}{b_i}.\]
\end{assumption}

\begin{assumption}[Amplitude range]
There exist constants $0<\underline\theta\le \bar\theta\le 1$
such that every amplitude satisfies
$
\theta_{i,t}\in[\underline\theta,\bar\theta]$.

\end{assumption}

\begin{assumption}[Executable variance proxies]
For each node $i$ and round $t$, the server may have an upper proxy $\widehat v_{i,t}^2\ge v_i^2$.
The idealized regime uses $\widehat v_{i,t}=v_i$.
\end{assumption}

\begin{assumption}[Initialization and inactive-node convention]
\label{ass:init}
Assume
\begin{equation}
c_0=\frac1n\sum_{i=1}^n c_{i,0}.
\label{eq:init}
\end{equation}
For every inactive node $i\notin \mathcal S_t$, set $c_{i,t+1}=c_{i,t}$.
\end{assumption}

\begin{definition}[Global filtration and predictability]
Let $\cF_t$ be the sigma-field generated by the initial state and by all random
minibatches drawn strictly before round $t$. The active set
$\mathcal S_t\subseteq[n]$ and every local control pair $(w_t,\theta_t)$ are
assumed to be $\cF_t$-measurable.
\end{definition}

\begin{definition}[Gap and deterministic cap]
For every round $t$, define
\[
f_t:=F(x_t)-F_\star,
\qquad
F_\star:=F(x_\star),\]
and $\bar f:=\frac{LR^2}{2}$.
\end{definition}

\begin{definition}[One-step admissible upper state]
At round $t$, a one-step admissible upper state is an $\cF_t$-measurable pair
$(U_t,Q_t)$ such that
\[
f_t\le U_t,
\qquad
\max_{1\le i\le n}\norm{c_{i,t}-\nabla F_i(x_t)}^2\le Q_t
\qquad\text{almost surely.}\]
We also define the capped upper gap $U_t^\sharp:=\min\{U_t,\bar f\}$.
\end{definition}

\paragraph{Deferred statements and auxiliary definitions.}

\begin{assumption}[Full participation for the global exact-control branch]
\label{ass:full-global}
Assume
\[
\mathcal S_t=[n]
\qquad
\text{for every }t\ge 0.
\]
The local horizons $(H_i)_{i=1}^n$, minibatch sizes $(b_i)_{i=1}^n$, and
variance levels $(v_i)_{i=1}^n$ may remain heterogeneous and are fixed across
rounds.
\end{assumption}

\begin{definition}[Deterministic surrogate upper state]
\label{def:surrogate-state}
Under Assumption~\ref{ass:full-global}, define
\[
g_t:=\EE[F(x_t)-F_\star].
\]
A deterministic surrogate upper state at round $t$ is a deterministic pair
$(u_t,\chi_t)\in[0,\infty)^2$ such that
\[
g_t\le u_t\le \bar f,
\qquad
\max_{1\le i\le n}\EE\norm{e_{i,t}}^2\le \chi_t.
\]
\end{definition}

\begin{definition}[Deterministic surrogate local objectives]
\label{def:sur-obj}
For deterministic $(u,\chi)$ with $0\le u\le \bar f$, define
\[
A_i(\theta_i):=\frac{\theta_i}{2LR^2},
\qquad
s_i(u;\theta_i):=u-T_{A_i(\theta_i)}(u),
\]
\[
\rho_i(\theta_i;u,\chi)
:=
32 e^{2\theta_i}\theta_i^3 u
+
\frac{16\theta_i+64 e^{2\theta_i}\theta_i^3}{L}\chi
+
8 e^{2\theta_i}\frac{\theta_i^3 v_i^2}{L H_i b_i},
\]
\[
\kappa_i(\theta_i;u,\chi)
:=
16 e^{2\theta_i}\frac{\theta_i^4}{L}u
+
32 e^{2\theta_i}\frac{\theta_i^4}{L^2}\chi
+
\left(2\theta_i^2+4 e^{2\theta_i}\theta_i^4\right)\frac{v_i^2}{L^2 H_i b_i},
\]
\[
\mu_i(\theta_i;u,\chi):=s_i(u;\theta_i)-\rho_i(\theta_i;u,\chi),
\]
and
\[
\mathfrak J_t^{\mathrm{id}}(w,\theta;u,\chi)
:=
u
-
\sum_{i=1}^n w_i\mu_i(\theta_i;u,\chi)
+
\frac{L}{2}\sum_{i=1}^n w_i^2\kappa_i(\theta_i;u,\chi),
\]
The executable surrogate objective
$\mathfrak J_t^{\mathrm{ex}}(w,\theta;u,\chi)$ is defined by replacing every
occurrence of $v_i^2$ by $\widehat v_{i,t}^2$.
\end{definition}

\begin{theorem}[Global surrogate exact local system]
\label{th:sur-system}
Assume Assumption~\ref{ass:full-global}. Let
\begin{equation}
u_0\in[g_0,\bar f],
\qquad
\chi_0\ge \max_{1\le i\le n}\EE\norm{e_{i,0}}^2,
\label{eq:sur-init-main}\end{equation}
For every $t\ge 0$, let
\[
(w_t^{\mathrm{id}},\theta_t^{\mathrm{id}})
\in
\argmin_{\substack{w\in\Delta_n\\ \theta\in[\underline\theta,\bar\theta]^n}}
\mathfrak J_t^{\mathrm{id}}(w,\theta;u_t,\chi_t),
\]
and define
\begin{equation}
u_{t+1}
:=
\min\!\left\{
\bar f,
\mathfrak J_t^{\mathrm{id}}(w_t^{\mathrm{id}},\theta_t^{\mathrm{id}};u_t,\chi_t)
\right\}.
\label{eq:u-rec-main}\end{equation}
\[
\chi_{t+1}
:=
A_\chi+B_\chi u_t+C_\chi\chi_t,
\]
where
\[
A_\chi:=6\max_{1\le i\le n}\frac{v_i^2}{H_i b_i},
\qquad
B_\chi:=144L\bar\theta^2,
\qquad
C_\chi:=288\bar\theta^2.
\]
Then $(u_t,\chi_t)$ is a deterministic surrogate upper state for every
$t\ge 0$, i.e.
\begin{equation}
g_t\le u_t\le \bar f,
\qquad
\max_{1\le i\le n}\EE\norm{e_{i,t}}^2\le \chi_t.
\label{eq:sur-system-main}\end{equation}
The same upper-state domination holds if the idealized surrogate objective is
replaced by the executable surrogate objective.
\end{theorem}

\begin{definition}[Comparator convex coefficients]
\label{def:cvx-comp}
Fix a deterministic feasible comparator pair $(\bar w,\bar\theta)$. Define
\[
\underline A_i(\bar\theta_i)
:=
\frac{A_i(\bar\theta_i)}{1+A_i(\bar\theta_i)\bar f}
=
\frac{\bar\theta_i}{2LR^2(1+\bar\theta_i/4)},
\]
\[
a_{\mathrm{cvx}}(\bar w,\bar\theta)
:=
\sum_{i=1}^n \bar w_i\,\underline A_i(\bar\theta_i),\quad\beta_{\mathrm{cvx}}(\bar w,\bar\theta)
:=
96\sum_{i=1}^n \bar w_i \bar\theta_i^2
+
32\sum_{i=1}^n \bar w_i^2 \bar\theta_i^4,
\]
\[
\gamma_{\mathrm{cvx}}(\bar w,\bar\theta)
:=
\frac{256}{L}\sum_{i=1}^n \bar w_i \bar\theta_i
+
\frac{64}{L}\sum_{i=1}^n \bar w_i^2 \bar\theta_i^4,\;\delta_{\mathrm{cvx}}(\bar w,\bar\theta)
:=
\frac{32}{L}\sum_{i=1}^n \bar w_i \bar\theta_i^2\frac{v_i^2}{H_i b_i}
+
\frac{16}{L}\sum_{i=1}^n \bar w_i^2 \bar\theta_i^2\frac{v_i^2}{H_i b_i}.
\]
\end{definition}

\begin{theorem}[Comparator convex recursion for the exact optimized controller]
\label{th:sur-cvx}
Under the hypotheses of Theorem~\ref{th:sur-system}, let
$(\bar w,\bar\theta)$ be any deterministic feasible comparator pair. Then the
exact optimized controller satisfies
\begin{equation}
u_{t+1}
\le
u_t
-
a_{\mathrm{cvx}}(\bar w,\bar\theta)\,u_t^2
+
\beta_{\mathrm{cvx}}(\bar w,\bar\theta)\,u_t
+
\gamma_{\mathrm{cvx}}(\bar w,\bar\theta)\,\chi_t
+
\delta_{\mathrm{cvx}}(\bar w,\bar\theta).
\label{eq:cvx-rec-main}\end{equation}
Consequently, if
\begin{equation}
\bar\chi
:=
\max\left\{
\chi_0,
\frac{A_\chi+B_\chi \bar f}{1-C_\chi}
\right\},
\label{eq:chi-bar-main}\end{equation}
then $\chi_t\le \bar\chi$ for all $t$, and therefore
\begin{equation}
u_{t+1}
\le
u_t
-
a_{\mathrm{cvx}}\,u_t^2
+
\beta_{\mathrm{cvx}}\,u_t
+
d_{\mathrm{cvx}},
\qquad
d_{\mathrm{cvx}}:=\gamma_{\mathrm{cvx}}\bar\chi+\delta_{\mathrm{cvx}},
\label{eq:cvx-rec-closed-main}\end{equation}
with the comparator-dependent coefficients suppressed for readability.
If the positive root
\[
m_{\mathrm{cvx}}
:=
\frac{
\beta_{\mathrm{cvx}}
+
\sqrt{\beta_{\mathrm{cvx}}^2+4a_{\mathrm{cvx}}d_{\mathrm{cvx}}}
}{2a_{\mathrm{cvx}}}
\]
also satisfies
\[
2a_{\mathrm{cvx}}m_{\mathrm{cvx}}\le 1+\beta_{\mathrm{cvx}},
\]
then for every $T\ge 0$,
\[
g_T\le u_T\le
m_{\mathrm{cvx}}+
\frac{1}{((u_0-m_{\mathrm{cvx}})_+)^{-1}+a_{\mathrm{cvx}}T}.
\]
In particular,
\[
g_T\le u_T\le m_{\mathrm{cvx}}+\frac{1}{u_0^{-1}+a_{\mathrm{cvx}}T}.
\]
\end{theorem}

This is the main closed heterogeneous global theorem for the local viewpoint. It is stated
for the exact optimized controller itself, not for a frozen comparator or a
uniformized proxy. The price of this generality is that the statement is a
surrogate convex recursion rather than a sharp direct Bellman PL theorem.

\begin{assumption}[Uniform-controller common-schedule specialization]
\label{ass:uniform-special}
Assume, in addition to Assumption~\ref{ass:full-global}, that
\[
H_i\equiv H,
\qquad
b_i\equiv b,
\qquad
\theta_{i,t}\equiv \vartheta,
\qquad
w_{i,t}\equiv \frac1n,
\qquad
288\vartheta^2<1.
\]
Define $v^2:=\max_{1\le i\le n}v_i^2$.
\end{assumption}

\begin{definition}[Deterministic Bellman state of the uniform-controller branch]
\label{def:direct-state}
Under Assumption~\ref{ass:uniform-special}, define
\[
q_t:=\max_{1\le i\le n}\EE\norm{e_{i,t}}^2.
\]
The deterministic Bellman state of this branch is $(g_t,q_t)$.
\end{definition}

\begin{theorem}[Direct one-step Bellman inequality]
\label{th:direct-bellman}
Under Assumption~\ref{ass:uniform-special},
\[
g_{t+1}
\le
g_t
-
a_{\mathrm{dir}}g_t^2
+
\beta_{\mathrm{dir}}g_t
+
\gamma_{\mathrm{dir}}q_t
+
\delta_{\mathrm{dir}},
\]
where
\[
a_{\mathrm{dir}}:=\frac{\vartheta}{2LR^2},
\qquad
\beta_{\mathrm{dir}}:=48\vartheta^3+\frac{32\vartheta^4}{n},
\]
\[
\gamma_{\mathrm{dir}}:=\frac{96\vartheta^3}{L}+\frac{64\vartheta^4}{Ln},
\qquad
\delta_{\mathrm{dir}}:=\left(16\vartheta^3+\frac{16\vartheta^2}{n}\right)\frac{v^2}{L H b}.
\]
In the same regime,
\[
q_{t+1}\le A_q+B_q g_t+C_q q_t,
\]
with
\[
A_q:=6\frac{v^2}{H b},
\qquad
B_q:=144L\vartheta^2,
\qquad
C_q:=288\vartheta^2.
\]
\end{theorem}

The special structure of the uniform controller is what enables the direct
Bellman closure. In this regime, the aggregate control-variate mismatch cancels
exactly, and the local drift remainder can be analyzed directly at the
averaged-update level.

\begin{assumption}[PL benchmark condition]
\label{ass:PL-benchmark}
Assume
\[
\norm{\nabla F(x)}^2\ge 2\mu(F(x)-F_\star)
\qquad
\text{for every iterate }x.
\]
\end{assumption}

\begin{theorem}[Global stochastic PL contraction in the uniform-controller branch]

Assume Assumptions~\ref{ass:uniform-special} and \ref{ass:PL-benchmark}. Define
\[
a_{\mathrm{PL}}:=\mu\frac{\vartheta}{L}-\beta_{\mathrm{dir}}.
\]
Suppose
\[
a_{\mathrm{PL}}>0,
\qquad
a_{\mathrm{PL}}<1-C_q,
\qquad
a_{\mathrm{PL}}(1-C_q)>B_q\gamma_{\mathrm{dir}}.
\]
Define
\[
\rho_{\mathrm{PL}}
:=
\frac{
a_{\mathrm{PL}}+1-C_q
-
\sqrt{(1-C_q-a_{\mathrm{PL}})^2+4B_q\gamma_{\mathrm{dir}}}
}{2},
\]
\[
\lambda_{\mathrm{PL}}
:=
\frac{
2\gamma_{\mathrm{dir}}
}{
1-C_q-a_{\mathrm{PL}}+
\sqrt{(1-C_q-a_{\mathrm{PL}})^2+4B_q\gamma_{\mathrm{dir}}}
}.
\]
Then
\[
s_t:=g_t+\lambda_{\mathrm{PL}}q_t
\]
satisfies
\[
s_{t+1}
\le
(1-\rho_{\mathrm{PL}})s_t
+
\delta_{\mathrm{dir}}+\lambda_{\mathrm{PL}}A_q,
\]
and therefore
\[
g_t
\le
(1-\rho_{\mathrm{PL}})^t\bigl(g_0+\lambda_{\mathrm{PL}}q_0\bigr)
+
\frac{\delta_{\mathrm{dir}}+\lambda_{\mathrm{PL}}A_q}{\rho_{\mathrm{PL}}}.
\]
In the readable safe regime
\[
\mu\frac{\vartheta}{L}\le \frac12,
\qquad
\vartheta^2\le \min\!\left\{\frac{\mu}{400L},\frac1{576}\right\},
\]
one has
\[
\rho_{\mathrm{PL}}\ge \frac{\mu\vartheta}{8L},
\qquad
\frac{\delta_{\mathrm{dir}}+\lambda_{\mathrm{PL}}A_q}{\rho_{\mathrm{PL}}}
=
O\!\left(\frac{\vartheta v^2}{\mu n H b}+\frac{\vartheta^2 v^2}{\mu H b}\right).
\]
\end{theorem}

This is the sharpest PL statement in the paper for the local viewpoint. The comparison with the recent
stochastic SCAFFOLD benchmark line is made only in this uniform-controller
regime and under the same PL-type benchmark condition. In that comparison
class, the theorem gives a geometric contraction to an explicit stochastic
floor of order \eqref{eq:pl-floor-main}, which is the same benchmark-style
communication-efficient scaling emphasized by recent stochastic SCAFFOLD
analyses, including the linear-speedup / higher-order bias perspective of
\citet{mangold2025scaffold}; see also the broader benchmark program of
\citet{luo2025revisiting}.

\begin{assumption}[Higher-order benchmark regularity]
\label{ass:ho}
For every node $i\in[n]$, the function $F_i$ is twice continuously
differentiable on $B(x_\star,R)$. There exist constants
$\mathcal H\ge 0$ and $M\ge 0$ such that, for all $x,y\in B(x_\star,R)$,
\[
\|\nabla^2 F_i(x)-\nabla^2 F(x)\|\le \mathcal H,\quad\|\nabla^2 F_i(x)-\nabla^2 F_i(y)\|\le M\|x-y\|.
\]
\end{assumption}

\begin{theorem}[Global higher-order convex benchmark bounds in the uniform-controller branch]

Under Assumptions~\ref{ass:uniform-special} and \ref{ass:ho}, define
\[
K_{\mathrm{ho}}:=(\mathcal H+MR)^2,
\qquad
a_{\mathrm{ho}}:=\frac{\vartheta}{2LR^2},
\qquad
\underline a_{\mathrm{ho}}:=\frac{\vartheta}{2LR^2(1+\vartheta/4)}.
\]
Then the one-round recursion satisfies
\[
g_{t+1}
\le
T_{a_{\mathrm{ho}}}(g_t)
+
\frac{2\vartheta}{L}K_{\mathrm{ho}}\Delta_t^2
+
\frac{\vartheta^2}{2L n H}\frac{v^2}{b},
\]
where
\[
\Delta_t^2
:=
\frac{96\vartheta^2}{L}g_t
+
\frac{32\vartheta^2 v^2}{L^2 H b}
+
\frac{192\vartheta^2}{L^2}q_t.
\]
Moreover, if
\[
\bar q
:=
\max\left\{q_0,\frac{A_q+B_q\bar f}{1-C_q}\right\},
\]
then $q_t\le \bar q$ for all $t$, and therefore
\[
g_{t+1}
\le
g_t
-
\underline a_{\mathrm{ho}} g_t^2
+
\beta_{\mathrm{ho}} g_t
+
\delta_{\mathrm{ho}},
\]
where
\[
\beta_{\mathrm{ho}}
:=
\frac{2\vartheta}{L}K_{\mathrm{ho}}\frac{96\vartheta^2}{L},
\qquad
\delta_{\mathrm{ho}}
:=
\frac{2\vartheta}{L}K_{\mathrm{ho}}
\left(
\frac{32\vartheta^2 v^2}{L^2 H b}+
\frac{192\vartheta^2}{L^2}\bar q
\right)
+
\frac{\vartheta^2}{2L n H}\frac{v^2}{b}.
\]
Consequently, for every integer $T\ge 1$,
\[
\min_{0\le t\le T-1} g_t
\le
\frac{\beta_{\mathrm{ho}}}{\underline a_{\mathrm{ho}}}
+
\sqrt{\frac{\delta_{\mathrm{ho}}}{\underline a_{\mathrm{ho}}}}
+
\sqrt{\frac{g_0}{\underline a_{\mathrm{ho}}T}}.
\]
In the homogeneous quadratic benchmark subcase
$\mathcal H=0$, $M=0$,
one has $K_{\mathrm{ho}}=0$ and therefore
\[
\min_{0\le t\le T-1} g_t
\le
C_1\sqrt{\frac{L R^2 g_0}{\vartheta T}}
+
C_2 R v\sqrt{\frac{\vartheta}{n H b}}
\]
for absolute constants $C_1,C_2>0$.
\end{theorem}

This higher-order branch is where the paper aligns most closely with the
recent higher-order benchmark program for LocalSGD / SCAFFOLD, especially the
Hessian-similarity and Lipschitz-Hessian regimes emphasized by
\citet{luo2025revisiting}. The comparison is again only for the
uniform-controller specialization together with the higher-order assumptions
stated above. Under those conditions, the theorem yields the same qualitative
benchmark picture: a higher-order correction term plus a best-iterate convex
rate, with the homogeneous quadratic subcase reducing to the clean
$T^{-1/2}$-plus-noise-floor form. The difference is that the present route uses
a direct averaged-branch Bellman mechanism tailored to the corrected-local-SGD
specialization.

\paragraph{Arbitrary-active-set identities and certificate proofs.}

\begin{lemma}[Server-average identity]
\label{lem:c-average}
Under Assumption~\ref{ass:init},
\[
c_t=\frac1n\sum_{i=1}^n c_{i,t}
\qquad\text{for all }t\ge 0.
\]
\end{lemma}

\begin{proof}
We prove the claim by induction. At $t=0$, it is exactly
\eqref{eq:init}. Assume it holds at time $t$. By the inactive-node convention,
\[
\sum_{i=1}^n(c_{i,t+1}-c_{i,t})
=
\sum_{i\in\mathcal S_t}(c_{i,t+1}-c_{i,t}).
\]
Hence the server update gives
\[
c_{t+1}
=
c_t+\frac1n\sum_{i=1}^n(c_{i,t+1}-c_{i,t})
=
\frac1n\sum_{i=1}^n c_{i,t}
+
\frac1n\sum_{i=1}^n(c_{i,t+1}-c_{i,t})
=
\frac1n\sum_{i=1}^n c_{i,t+1}.
\]
\end{proof}

\begin{lemma}[Control-variates average identity]
\label{lem:avg}
For every active node $i\in\mathcal S_t$,
\[
c_{i,t+1}
=
\frac1{H_i}\sum_{\ell=0}^{H_i-1}g_{i,t,\ell}.
\]
\end{lemma}

\begin{proof}
Summing the branch recursion over $\ell=0,\dots,H_i-1$ gives
\[
x_t-y_{i,t}^{(H_i)}
=
\eta_{i,t}\sum_{\ell=0}^{H_i-1}(g_{i,t,\ell}-c_{i,t}+c_t).
\]
Substitute this identity into the control-variate update:
\[
c_{i,t+1}
=
c_{i,t}-c_t+\frac{1}{H_i\eta_{i,t}}(x_t-y_{i,t}^{(H_i)})
=
\frac1{H_i}\sum_{\ell=0}^{H_i-1}g_{i,t,\ell}.
\]
\end{proof}

\begin{lemma}[Tracking mismatch]
\label{lem:delta}
Define
\[
\delta_{i,t}
:=
\nabla F_i(x_t)-\nabla F(x_t)-(c_{i,t}-c_t).
\]
Then
$\norm{\delta_{i,t}}^2\le 4Q_t$ almost surely.
\end{lemma}

\begin{proof}
Since
$c_t-\nabla F(x_t)=\frac1n\sum_{j=1}^n e_{j,t}$,
one has 
$\delta_{i,t}=-e_{i,t}+\frac1n\sum_{j=1}^n e_{j,t}$.
Hence
\[
\norm{\delta_{i,t}}^2
\le
2\norm{e_{i,t}}^2+2\left\|\frac1n\sum_{j=1}^n e_{j,t}\right\|^2
\le
2\norm{e_{i,t}}^2+2\frac1n\sum_{j=1}^n\norm{e_{j,t}}^2
\le 4Q_t.
\]
\end{proof}

\begin{definition}[Conditional branch radius]
\label{def:radius}
For every active node $i\in\mathcal S_t$, define
\[
R_{i,t}^2
:=
\max_{0\le \ell\le H_i}
\EE\!\left[\norm{y_{i,t}^{(\ell)}-x_t}^2\middle|\cF_t\right].
\]
\end{definition}

\begin{lemma}[Uniform branch radius]
\label{lem:radius}
For every active node $i\in\mathcal S_t$,
\[
R_{i,t}^2
\le
\frac{24\theta_{i,t}^2}{L}U_t
+
\frac{8\theta_{i,t}^2 v_i^2}{L^2 H_i b_i}
+
\frac{48\theta_{i,t}^2}{L^2}Q_t
\qquad\text{almost surely.}
\]
\end{lemma}

\begin{proof}
Fix $i\in\mathcal S_t$ and $t$. For every $\ell\in\{0,\dots,H_i\}$,
\[
y_{i,t}^{(\ell)}-x_t
=
-\eta_{i,t}\sum_{s=0}^{\ell-1}
\left(
\nabla F(x_t)+\delta_{i,t}
+\bigl(\nabla F_i(y_{i,t}^{(s)})-\nabla F_i(x_t)\bigr)
+\varepsilon_{i,t,s}
\right),
\]
where $\varepsilon_{i,t,s}:=g_{i,t,s}-\nabla F_i(y_{i,t}^{(s)})$.
Hence
\begin{align}
\EE\!\left[\norm{y_{i,t}^{(\ell)}-x_t}^2\middle|\cF_t\right]
&\le
2\eta_{i,t}^2
\EE\!\left[
\left\|
\sum_{s=0}^{\ell-1}
\left(
\nabla F(x_t)+\delta_{i,t}
+\nabla F_i(y_{i,t}^{(s)})-\nabla F_i(x_t)
\right)
\right\|^2
\middle|\cF_t
\right]
\notag\\
&\quad+
2\eta_{i,t}^2
\EE\!\left[
\left\|\sum_{s=0}^{\ell-1}\varepsilon_{i,t,s}\right\|^2
\middle|\cF_t
\right].
\label{eq:radius-split-app}
\end{align}
For the deterministic part,
\[
\left\|\sum_{s=0}^{\ell-1}a_s\right\|^2
\le
\ell\sum_{s=0}^{\ell-1}\norm{a_s}^2
\le
H_i^2\max_{0\le s\le H_i-1}\norm{a_s}^2.
\]
Thus
\begin{align}
&\EE\!\left[
\left\|
\sum_{s=0}^{\ell-1}
\left(
\nabla F(x_t)+\delta_{i,t}
+\nabla F_i(y_{i,t}^{(s)})-\nabla F_i(x_t)
\right)
\right\|^2
\middle|\cF_t
\right]
\notag\\
&\qquad\qquad\le
H_i^2\left(
3\norm{\nabla F(x_t)}^2
+
3\norm{\delta_{i,t}}^2
+
3L^2R_{i,t}^2
\right).
\label{eq:radius-det-app}
\end{align}
For the stochastic part, the sequence $(\varepsilon_{i,t,s})_{s=0}^{H_i-1}$ is a
martingale-difference sequence with respect to $(\cG_{i,t,s})_{s=0}^{H_i-1}$,
so the cross terms vanish conditionally on $\cF_t$. Therefore
\begin{equation}
\EE\!\left[
\left\|\sum_{s=0}^{\ell-1}\varepsilon_{i,t,s}\right\|^2
\middle|\cF_t
\right]
=
\sum_{s=0}^{\ell-1}
\EE\!\left[\norm{\varepsilon_{i,t,s}}^2\middle|\cF_t\right]
\le
H_i\frac{v_i^2}{b_i}.
\label{eq:radius-sto-app}\end{equation}
Substituting \eqref{eq:radius-det-app} and \eqref{eq:radius-sto-app} into
\eqref{eq:radius-split-app}, and using $\eta_{i,t}=\theta_{i,t}/(L H_i)$,
produces
\[
\EE\!\left[\norm{y_{i,t}^{(\ell)}-x_t}^2\middle|\cF_t\right]
\le
\frac{6\theta_{i,t}^2}{L^2}\norm{\nabla F(x_t)}^2
+
\frac{6\theta_{i,t}^2}{L^2}\norm{\delta_{i,t}}^2
+
6\theta_{i,t}^2R_{i,t}^2
+
\frac{2\theta_{i,t}^2v_i^2}{L^2 H_i b_i}.
\]
Define
\[
r_{i,t,\ell}:=\Bigl(\EE[\norm{y_{i,t}^{(\ell)}-x_t}^2\mid\cF_t]\Bigr)^{1/2}.
\]
By Minkowski's inequality, the decomposition above, and the conditional orthogonality of the noise terms,
\[
r_{i,t,\ell}
\le
\eta_{i,t}\ell\bigl(\norm{\nabla F(x_t)}+\norm{\delta_{i,t}}\bigr)
+
\eta_{i,t}L\sum_{s=0}^{\ell-1} r_{i,t,s}
+
\eta_{i,t}\sqrt{\ell}\,\frac{v_i}{\sqrt{b_i}}.
\]
Since $\ell\le H_i$ and $\eta_{i,t}=\theta_{i,t}/(L H_i)$,
\[
r_{i,t,\ell}
\le
\frac{\theta_{i,t}}{L}\bigl(\norm{\nabla F(x_t)}+\norm{\delta_{i,t}}\bigr)
+
\frac{\theta_{i,t}v_i}{L\sqrt{H_i b_i}}
+
\frac{\theta_{i,t}}{H_i}\sum_{s=0}^{\ell-1} r_{i,t,s}.
\]
Apply Lemma~\ref{lem:cum-gronwall} with $\beta=\theta_{i,t}/H_i$ to obtain
\[
r_{i,t,\ell}
\le
\exp(\theta_{i,t})\left[
\frac{\theta_{i,t}}{L}\bigl(\norm{\nabla F(x_t)}+\norm{\delta_{i,t}}\bigr)
+
\frac{\theta_{i,t}v_i}{L\sqrt{H_i b_i}}
\right].
\]
Squaring, using $(a+b+c)^2\le 4a^2+4b^2+2c^2$, and then Lemmas~\ref{lem:gap} and \ref{lem:delta}, yields
\[
\EE\!\left[\norm{y_{i,t}^{(\ell)}-x_t}^2\middle|\cF_t\right]
\le
8 e^{2\theta_{i,t}}\frac{\theta_{i,t}^2}{L}U_t
+
16 e^{2\theta_{i,t}}\frac{\theta_{i,t}^2}{L^2}Q_t
+
2 e^{2\theta_{i,t}}\frac{\theta_{i,t}^2v_i^2}{L^2 H_i b_i}.
\]
Taking the maximum over $\ell$ proves the claim.
\end{proof}

\begin{lemma}[Endpoint decomposition]
\label{lem:endpoint}
For every active node $i\in\mathcal S_t$,
\begin{equation}
\Delta_{i,t}
=
-\frac{\theta_{i,t}}{L}\nabla F(x_t)
-\frac{\theta_{i,t}}{L}\delta_{i,t}
+r_{i,t}
+\xi_{i,t},
\label{eq:endpoint-app}\end{equation}
where
\[
r_{i,t}
:=
-\eta_{i,t}\sum_{\ell=0}^{H_i-1}\bigl(\nabla F_i(y_{i,t}^{(\ell)})-\nabla F_i(x_t)\bigr),\quad\xi_{i,t}
:=
-\eta_{i,t}\sum_{\ell=0}^{H_i-1}\bigl(g_{i,t,\ell}-\nabla F_i(y_{i,t}^{(\ell)})\bigr).
\]
Define
\[
m_{i,t}:=\EE[\Delta_{i,t}\mid \cF_t],
\qquad
\zeta_{i,t}:=\Delta_{i,t}-m_{i,t}.
\]
Then
\begin{equation}
m_{i,t}
=
-\frac{\theta_{i,t}}{L}\nabla F(x_t)
-\frac{\theta_{i,t}}{L}\delta_{i,t}
+\bar r_{i,t},
\qquad
\bar r_{i,t}:=\EE[r_{i,t}\mid \cF_t],
\label{eq:mean-decomp-app}\end{equation}
\begin{equation}
\EE[\zeta_{i,t}\mid \cF_t]=0,
\label{eq:zeta-center-app}\end{equation}
and for $i\neq j$,
\begin{equation}
\EE\!\left[\ip{\zeta_{i,t}}{\zeta_{j,t}}\middle|\cF_t\right]=0.
\label{eq:zeta-orth-app}\end{equation}
\end{lemma}

\begin{proof}
Summing the branch recursion over $\ell=0,\dots,H_i-1$ yields
\[
\Delta_{i,t}
=
-\eta_{i,t}\sum_{\ell=0}^{H_i-1}(g_{i,t,\ell}-c_{i,t}+c_t).
\]
Insert and subtract $\nabla F_i(y_{i,t}^{(\ell)})$ and $\nabla F_i(x_t)$ to obtain
\eqref{eq:endpoint-app}. Taking conditional expectation gives
\eqref{eq:mean-decomp-app}. Equation \eqref{eq:zeta-center-app} is immediate
from the definition, and \eqref{eq:zeta-orth-app} follows from conditional
independence of the nodewise randomness given $\cF_t$.
\end{proof}

\begin{lemma}[Conditional mean remainder bound]
\label{lem:rbar}
For every active node $i\in\mathcal S_t$,
\begin{equation}
\norm{\bar r_{i,t}}^2
\le
\EE[\norm{r_{i,t}}^2\mid \cF_t]
\le
\theta_{i,t}^2 R_{i,t}^2
\qquad\text{almost surely.}
\label{eq:rbar-app}\end{equation}
Consequently,
\[
\norm{\bar r_{i,t}}^2
\le
24\frac{\theta_{i,t}^4}{L}U_t
+
8\frac{\theta_{i,t}^4v_i^2}{L^2 H_i b_i}
+
48\frac{\theta_{i,t}^4}{L^2}Q_t.
\]
\end{lemma}

\begin{proof}
By Jensen's inequality for conditional expectation,
\[
\norm{\bar r_{i,t}}^2
=
\norm{\EE[r_{i,t}\mid \cF_t]}^2
\le
\EE[\norm{r_{i,t}}^2\mid \cF_t].
\]
By the definition of $r_{i,t}$,
\[
\norm{r_{i,t}}^2
\le
\eta_{i,t}^2 H_i\sum_{\ell=0}^{H_i-1}
\norm{\nabla F_i(y_{i,t}^{(\ell)})-\nabla F_i(x_t)}^2
\le
\eta_{i,t}^2 H_i\sum_{\ell=0}^{H_i-1}L^2\norm{y_{i,t}^{(\ell)}-x_t}^2.
\]
Using $\eta_{i,t}=\theta_{i,t}/(L H_i)$ yields
\[
\norm{r_{i,t}}^2
\le
\frac{\theta_{i,t}^2}{H_i}\sum_{\ell=0}^{H_i-1}\norm{y_{i,t}^{(\ell)}-x_t}^2
\le
\theta_{i,t}^2\max_{0\le \ell\le H_i}\norm{y_{i,t}^{(\ell)}-x_t}^2.
\]
Taking conditional expectation proves \eqref{eq:rbar-app}; the expanded bound
follows from Lemma~\ref{lem:radius}.
\end{proof}

\begin{lemma}[Centered endpoint variance]
\label{lem:kappa}
For every active node $i\in\mathcal S_t$,
\[
\EE[\norm{\zeta_{i,t}}^2\mid \cF_t]
\le
\kappa_i(\theta_{i,t};U_t,Q_t)
\qquad\text{almost surely.}
\]
\end{lemma}

\begin{proof}
Since
\[
\zeta_{i,t}=(r_{i,t}-\EE[r_{i,t}\mid \cF_t])+\xi_{i,t},
\]
one has
\[
\EE[\norm{\zeta_{i,t}}^2\mid \cF_t]
\le
2\EE[\norm{r_{i,t}}^2\mid \cF_t]
+
2\EE[\norm{\xi_{i,t}}^2\mid \cF_t].
\]
The first term is controlled by Lemma~\ref{lem:rbar}. For the second term,
conditional orthogonality across local steps gives
\[
\EE[\norm{\xi_{i,t}}^2\mid \cF_t]
\le
\eta_{i,t}^2 H_i\frac{v_i^2}{b_i}
=
\frac{\theta_{i,t}^2v_i^2}{L^2H_i b_i}.
\]
Combining the two estimates and using $\theta_{i,t}\le 1$ gives the stated
bound.
\end{proof}

\begin{definition}[Aggregated mean and centered residual]
\label{def:agg}
For a predictable feasible pair $(w,\theta)$, define
\[
\bar m_t(w,\theta):=\sum_{i\in\mathcal S_t}w_i m_{i,t},
\qquad
\bar \zeta_t(w,\theta):=\sum_{i\in\mathcal S_t}w_i \zeta_{i,t}.
\]
\end{definition}

\begin{theorem}[Exact aggregated server one-step inequality]
\label{th:agg-server}
For every predictable feasible pair $(w,\theta)$ with $w\in\Delta_{S_t}$,
\begin{align*}
\EE[F(x_{t+1})-F_\star\mid \cF_t,w,\theta]
&\le
f_t
+
\ip{\nabla F(x_t)}{\bar m_t(w,\theta)}
+
\frac{L}{2}\norm{\bar m_t(w,\theta)}^2
\notag\\
&\quad+
\frac{L}{2}\sum_{i\in\mathcal S_t}w_i^2
\EE[\norm{\zeta_{i,t}}^2\mid \cF_t].
\end{align*}
In particular,
\begin{align*}
\EE[F(x_{t+1})-F_\star\mid \cF_t,w,\theta]
&\le
f_t
+
\ip{\nabla F(x_t)}{\bar m_t(w,\theta)}
+
\frac{L}{2}\norm{\bar m_t(w,\theta)}^2
\notag\\
&\quad+
\frac{L}{2}\sum_{i\in\mathcal S_t}w_i^2 \kappa_i(\theta_{i,t};U_t,Q_t).
\end{align*}
\end{theorem}

\begin{proof}
Write $
x_{t+1}
=
x_t+\sum_{i\in\mathcal S_t}w_i\Delta_{i,t}
=
x_t+\bar m_t+\bar \zeta_t$.
By $L$-smoothness,
\[
F(x_t+\bar m_t+\bar \zeta_t)
\le
F(x_t)
+
\ip{\nabla F(x_t)}{\bar m_t+\bar \zeta_t}
+
\frac{L}{2}\norm{\bar m_t+\bar \zeta_t}^2.
\]
Taking $\EE[\cdot\mid\cF_t,w,\theta]$ and using
$\EE[\bar \zeta_t\mid\cF_t]=0$ gives
\[
\EE[F(x_{t+1})\mid\cF_t,w,\theta]
\le
F(x_t)
+
\ip{\nabla F(x_t)}{\bar m_t}
+
\frac{L}{2}\norm{\bar m_t}^2
+
\frac{L}{2}\EE[\norm{\bar\zeta_t}^2\mid\cF_t].
\]
Now
\[
\EE[\norm{\bar\zeta_t}^2\mid\cF_t]
=
\sum_i w_i^2\EE[\norm{\zeta_{i,t}}^2\mid\cF_t]
+
\sum_{i\neq j}w_iw_j\EE\!\left[\ip{\zeta_{i,t}}{\zeta_{j,t}}\middle|\cF_t\right].
\]
The cross terms vanish by Lemma~\ref{lem:endpoint}. Applying
Lemma~\ref{lem:kappa} yields the second display.
\end{proof}

\begin{proposition}[Direct nodewise mean certificate]
\label{prop:mean}
For every active node $i\in\mathcal S_t$,
\[
\ip{\nabla F(x_t)}{m_{i,t}}
+
\frac{L}{2}\norm{m_{i,t}}^2
\le
-
\frac{\theta_{i,t}}{2LR^2}f_t^2
+
\rho_i(\theta_{i,t};U_t,Q_t)
\qquad\text{almost surely.}
\]
\end{proposition}

\begin{proof}
Fix $i\in\mathcal S_t$. Write
\[
\theta:=\theta_{i,t},
\qquad
g:=\nabla F(x_t),
\qquad
\delta:=\delta_{i,t},
\qquad
\bar r:=\bar r_{i,t},
\qquad
m:=m_{i,t}.
\]
By Lemma~\ref{lem:endpoint}, $
m=-\frac{\theta}{L}g-\frac{\theta}{L}\delta+\bar r$.
Set $
\nu:=-\frac{\theta}{L}\delta+\bar r$,
so that $m=-(\theta/L)g+\nu$. Then
\begin{align*}
\ip{g}{m}+\frac{L}{2}\norm{m}^2
&=
-\frac{\theta-\theta^2/2}{L}\norm{g}^2
+
(1-\theta)\ip{g}{\nu}
+
\frac{L}{2}\norm{\nu}^2.
\end{align*}
By Young's inequality,
\[
(1-\theta)\ip{g}{\nu}
\le
\frac{\theta}{4L}\norm{g}^2+\frac{L}{\theta}\norm{\nu}^2.
\]
Since $0<\theta\le 1$,
\[
-\frac{\theta-\theta^2/2}{L}+\frac{\theta}{4L}\le -\frac{\theta}{2L},
\qquad
\frac{L}{\theta}+\frac{L}{2}\le \frac{2L}{\theta}.
\]
Thus
\begin{equation}
\ip{g}{m}+\frac{L}{2}\norm{m}^2
\le
-\frac{\theta}{2L}\norm{g}^2
+
\frac{2L}{\theta}\norm{\nu}^2.
\label{eq:mean-after-young-app}\end{equation}
Next,
\[
\norm{\nu}^2
\le
2\frac{\theta^2}{L^2}\norm{\delta}^2+2\norm{\bar r}^2.
\]
Apply Lemma~\ref{lem:delta} and Lemma~\ref{lem:rbar}, then use
$\theta\le 1$ to obtain
\[
\frac{2L}{\theta}\norm{\nu}^2
\le
256\frac{\theta}{L}Q_t
+
96\theta^2 U_t
+
32\frac{\theta^2v_i^2}{L H_i b_i}.
\]
Substituting into \eqref{eq:mean-after-young-app} and invoking
$\norm{g}^2\ge f_t^2/R^2$ from Lemma~\ref{lem:gap} proves the claim.
\end{proof}

\begin{lemma}[Nodewise upper certificate]
\label{lem:node-cert}
For every active node $i\in\mathcal S_t$,
\[
f_t
+
\ip{\nabla F(x_t)}{m_{i,t}}
+
\frac{L}{2}\norm{m_{i,t}}^2
\le
T_{A_i(\theta_{i,t})}(U_t^\sharp)
+
\rho_i(\theta_{i,t};U_t,Q_t)
\qquad\text{almost surely.}
\]
\end{lemma}

\begin{proof}
By Proposition~\ref{prop:mean},
\[
f_t
+
\ip{\nabla F(x_t)}{m_{i,t}}
+
\frac{L}{2}\norm{m_{i,t}}^2
\le
f_t-A_i(\theta_{i,t})f_t^2+\rho_i(\theta_{i,t};U_t,Q_t).
\]
Define $\varphi(z):=z-A_i(\theta_{i,t})z^2$. Then
\[
\varphi'(z)=1-2A_i(\theta_{i,t})z.
\]
For $0\le z\le \bar f$,
\[
2A_i(\theta_{i,t})z
\le
2A_i(\theta_{i,t})\bar f
=
\frac{\theta_{i,t}}{2}\le \frac12,
\]
so $\varphi'(z)\ge 1/2>0$. Thus $\varphi$ is increasing on $[0,\bar f]$. Since
$0\le f_t\le U_t^\sharp\le \bar f$,
\[
f_t-A_i(\theta_{i,t})f_t^2
=\varphi(f_t)
\le \varphi(U_t^\sharp)
=U_t^\sharp-A_i(\theta_{i,t})(U_t^\sharp)^2.
\]
Lemma~\ref{lem:Ta-majorizes} gives
\[
U_t^\sharp-A_i(\theta_{i,t})(U_t^\sharp)^2
\le
T_{A_i(\theta_{i,t})}(U_t^\sharp),
\]
which proves the claim.
\end{proof}

\begin{proposition}[Existence and measurable selection of exact one-step minimizers]
\label{prop:existence}
For every realization of the round-$t$ information, the feasible set
$\Delta_{S_t}\times[\underline\theta,\bar\theta]^{S_t}$ is compact and the
one-step local objectives are continuous. Hence exact one-step minimizers
exist. Moreover, because the objective is a Caratheodory function of the state
and the feasible correspondence is measurable with compact values, a predictable
exact minimizer may be selected by the measurable maximum theorem.
\end{proposition}

\begin{proof}
Compactness and continuity imply existence by the Weierstrass theorem. For the
predictable selection statement, the feasible correspondence
$\omega\mapsto \Delta_{S_t(\omega)}\times[\underline\theta,\bar\theta]^{S_t(\omega)}$
is measurable with nonempty compact values, and the objective is measurable in
$\omega$ and continuous in the control variables. The measurable maximum theorem
therefore yields a measurable exact minimizer, which is predictable because the
state is $\cF_t$-measurable.
\end{proof}

\begin{theorem}[Proof of Theorem~\ref{th:predictive-majorant}]
\label{th:predictive-majorant-proof}
The conclusions of Theorem~\ref{th:predictive-majorant} hold.
\end{theorem}

\begin{proof}
By Theorem~\ref{th:agg-server},
\begin{align}
\EE[F(x_{t+1})-F_\star\mid \cF_t,w,\theta]
&\le
f_t
+
\ip{\nabla F(x_t)}{\bar m_t(w,\theta)}
+
\frac{L}{2}\norm{\bar m_t(w,\theta)}^2
\notag\\
&\quad+
\frac{L}{2}\sum_{i\in\mathcal S_t}w_i^2\kappa_i(\theta_i;U_t,Q_t).
\label{eq:maj-start-app}
\end{align}
Now
\[
\bar m_t(w,\theta)=\sum_{i\in\mathcal S_t} w_i m_{i,t},
\qquad
\ip{\nabla F(x_t)}{\bar m_t(w,\theta)}
=
\sum_{i\in\mathcal S_t} w_i \ip{\nabla F(x_t)}{m_{i,t}}.
\]
By convexity of $\norm{\cdot}^2$, we have $
\norm{\bar m_t(w,\theta)}^2
\le
\sum_{i\in\mathcal S_t}w_i\norm{m_{i,t}}^2$.
Hence
\[
f_t
+
\ip{\nabla F(x_t)}{\bar m_t(w,\theta)}
+
\frac{L}{2}\norm{\bar m_t(w,\theta)}^2
\le
\sum_{i\in\mathcal S_t}w_i
\left(
f_t+\ip{\nabla F(x_t)}{m_{i,t}}+\frac{L}{2}\norm{m_{i,t}}^2
\right).
\]
Apply Lemma~\ref{lem:node-cert}:
\[
f_t+\ip{\nabla F(x_t)}{m_{i,t}}+\frac{L}{2}\norm{m_{i,t}}^2
\le
T_{A_i(\theta_i)}(U_t^\sharp)+\rho_i(\theta_i;U_t,Q_t).
\]
Therefore
\begin{align*}
f_t
+
\ip{\nabla F(x_t)}{\bar m_t(w,\theta)}
+
\frac{L}{2}\norm{\bar m_t(w,\theta)}^2
&\le
\sum_{i\in\mathcal S_t}w_i
\left(
T_{A_i(\theta_i)}(U_t^\sharp)+\rho_i(\theta_i;U_t,Q_t)
\right)
\\
&=
\sum_{i\in\mathcal S_t}w_i
\left(
U_t^\sharp-s_i(U_t^\sharp;\theta_i)+\rho_i(\theta_i;U_t,Q_t)
\right)
\\
&=
U_t^\sharp-\sum_{i\in\mathcal S_t}w_i\mu_i(\theta_i;U_t,Q_t).
\end{align*}
Substituting into \eqref{eq:maj-start-app} yields
\eqref{eq:main-one-step}. Executable domination follows by replacing every
occurrence of $v_i^2$ by the larger value $\widehat v_{i,t}^2$ in the
$\rho$- and $\kappa$-terms.
\end{proof}

\begin{proposition}[KKT law in the weights]
\label{prop:KKT}
Fix $\theta$. If $\kappa_i(\theta_i;U_t,Q_t)>0$ for all $i\in\mathcal S_t$, then
the unique minimizer of $\mathcal J_t^{\mathrm{id}}(\cdot,\theta)$ over
$\Delta_{S_t}$ is
\[
w_{i,t}^\star
=
\frac{\bigl(\mu_i(\theta_i;U_t,Q_t)-\lambda_t\bigr)_+}
{L\kappa_i(\theta_i;U_t,Q_t)},
\qquad
i\in\mathcal S_t,
\]
where $\lambda_t\in\RR$ is the unique threshold such that
$\sum_{i\in\mathcal S_t}w_{i,t}^\star=1$.
\end{proposition}

\begin{proof}
For fixed $\theta$, the objective is
\[
w\mapsto
U_t^\sharp-\sum_{i\in\mathcal S_t}w_i \mu_i(\theta_i;U_t,Q_t)
+
\frac{L}{2}\sum_{i\in\mathcal S_t}\kappa_i(\theta_i;U_t,Q_t) w_i^2
\]
on the simplex. This is strictly convex because every $\kappa_i>0$. The KKT
conditions are
\[
L\kappa_i w_i-\mu_i+\lambda-\nu_i=0,
\qquad
\nu_i\ge 0,
\qquad
\nu_i w_i=0,
\qquad
\sum_i w_i=1.
\]
Hence $
w_i=\frac{(\mu_i-\lambda)_+}{L\kappa_i}$.
The threshold $\lambda_t$ is uniquely determined by the simplex constraint
because the map $
\lambda\mapsto \sum_i \frac{(\mu_i-\lambda)_+}{L\kappa_i}$
is continuous and strictly decreasing from $+\infty$ to $0$.
\end{proof}

\begin{theorem}[Proof of Theorem~\ref{th:one-step-hetero}]
\label{th:one-step-hetero-proof}
The conclusions of Theorem~\ref{th:one-step-hetero} hold.
\end{theorem}

\begin{proof}
The first inequality in \eqref{eq:one-step-control-law} is
Theorem~\ref{th:predictive-majorant} evaluated at the exact minimizer
$(w_t^{\mathrm{id}},\theta_t^{\mathrm{id}})$. The second inequality follows from
exact minimization and executable domination:
\[
\mathcal J_t^{\mathrm{id}}(w_t^{\mathrm{id}},\theta_t^{\mathrm{id}})
\le
\mathcal J_t^{\mathrm{id}}(w_t^{\mathrm{ex}},\theta_t^{\mathrm{ex}})
\le
\mathcal J_t^{\mathrm{ex}}(w_t^{\mathrm{ex}},\theta_t^{\mathrm{ex}}).
\]
The benchmark inequality \eqref{eq:one-step-control-law} follows by defining
\[
\Gamma_t
=
\mathcal J_t^{\mathrm{id}}(\bar w_t,\bar\theta_t)
-
\mathcal J_t^{\mathrm{id}}(w_t^{\mathrm{id}},\theta_t^{\mathrm{id}})\ge 0
\]
and substituting into the first inequality. The fixed-$\theta$ KKT law is
Proposition~\ref{prop:KKT}.
\end{proof}

\section{Full-participation and uniform-controller local proofs}
\begin{lemma}[Expectation-level mismatch bound under a surrogate upper state]
\label{lem:delta-sur}
Assume Assumption~\ref{ass:full-global}, and let $(u_t,\chi_t)$ be a surrogate
upper state. Then, for every node $i\in[n]$,
$\EE\norm{\delta_{i,t}}^2\le 4\chi_t$.
\end{lemma}

\begin{proof}
By Lemma~\ref{lem:c-average}, $
\delta_{i,t}=-e_{i,t}+\frac1n\sum_{j=1}^n e_{j,t}$.
Therefore
\[
\norm{\delta_{i,t}}^2
\le
2\norm{e_{i,t}}^2+2\left\|\frac1n\sum_{j=1}^n e_{j,t}\right\|^2
\le
2\norm{e_{i,t}}^2+2\frac1n\sum_{j=1}^n \norm{e_{j,t}}^2.
\]
Taking expectations and using
$\max_j\EE\norm{e_{j,t}}^2\le \chi_t$ yields the result.
\end{proof}

\begin{definition}[Deterministic branch radius under a surrogate upper state]
\label{def:sur-radius}
For every node $i\in[n]$, define
\[
\widetilde R_{i,t}^2
:=
\max_{0\le \ell\le H_i}
\EE\norm{y_{i,t}^{(\ell)}-x_t}^2.
\]
\end{definition}

\begin{lemma}[Deterministic branch radius bound under a surrogate upper state]
\label{lem:sur-radius}
Assume Assumption~\ref{ass:full-global}, and let $(u_t,\chi_t)$ be a surrogate
upper state. Then, for every deterministic feasible amplitude vector
$\theta\in[\underline\theta,\bar\theta]^n$ and every node $i\in[n]$,
\begin{equation}
\widetilde R_{i,t}^2
\le
\frac{24\theta_i^2}{L}u_t
+
\frac{8\theta_i^2 v_i^2}{L^2 H_i b_i}
+
\frac{48\theta_i^2}{L^2}\chi_t.
\label{eq:sur-radius-app}\end{equation}
\end{lemma}

\begin{proof}
Fix $i\in[n]$ and a deterministic feasible amplitude vector $\theta$. For
$\ell\in\{0,\dots,H_i\}$,
\[
y_{i,t}^{(\ell)}-x_t
=
-\eta_i\sum_{s=0}^{\ell-1}
\left(
\nabla F(x_t)+\delta_{i,t}
+\bigl(\nabla F_i(y_{i,t}^{(s)})-\nabla F_i(x_t)\bigr)
+\varepsilon_{i,t,s}
\right),
\]
where $
\eta_i:=\frac{\theta_i}{L H_i}$,
$\varepsilon_{i,t,s}:=g_{i,t,s}-\nabla F_i(y_{i,t}^{(s)})$.
Hence
\begin{align}
\EE\norm{y_{i,t}^{(\ell)}-x_t}^2
&\le
2\eta_i^2
\EE\left\|
\sum_{s=0}^{\ell-1}
\left(
\nabla F(x_t)+\delta_{i,t}
+\nabla F_i(y_{i,t}^{(s)})-\nabla F_i(x_t)
\right)
\right\|^2
\notag\\
&\quad+
2\eta_i^2
\EE\left\|
\sum_{s=0}^{\ell-1}\varepsilon_{i,t,s}
\right\|^2.
\label{eq:sur-radius-split-app-detail}
\end{align}
For the deterministic part,
\[
\left\|\sum_{s=0}^{\ell-1}a_s\right\|^2
\le
\ell\sum_{s=0}^{\ell-1}\norm{a_s}^2
\le
H_i^2\max_{0\le s\le H_i-1}\norm{a_s}^2.
\]
Therefore
\begin{align*}
&\EE\left\|
\sum_{s=0}^{\ell-1}
\left(
\nabla F(x_t)+\delta_{i,t}
+\nabla F_i(y_{i,t}^{(s)})-\nabla F_i(x_t)
\right)
\right\|^2\le
H_i^2\left(
3\EE\norm{\nabla F(x_t)}^2
+
3\EE\norm{\delta_{i,t}}^2
+
3L^2\widetilde R_{i,t}^2
\right),
\end{align*}
where the last term uses the definition of $\widetilde R_{i,t}^2$.

For the stochastic term, the sequence
$(\varepsilon_{i,t,s})_{s=0}^{H_i-1}$ is a martingale-difference sequence with
respect to the local filtration within the round, and hence its cross terms
vanish. Consequently,
\[
\EE\left\|\sum_{s=0}^{\ell-1}\varepsilon_{i,t,s}\right\|^2
=
\sum_{s=0}^{\ell-1}\EE\norm{\varepsilon_{i,t,s}}^2
\le
H_i\frac{v_i^2}{b_i}.
\]
Substituting the last two bounds into
\eqref{eq:sur-radius-split-app-detail} and using
$\eta_i=\theta_i/(L H_i)$, we obtain
\[
\EE\norm{y_{i,t}^{(\ell)}-x_t}^2
\le
\frac{6\theta_i^2}{L^2}\EE\norm{\nabla F(x_t)}^2
+
\frac{6\theta_i^2}{L^2}\EE\norm{\delta_{i,t}}^2
+
6\theta_i^2\widetilde R_{i,t}^2
+
\frac{2\theta_i^2v_i^2}{L^2 H_i b_i}.
\]
Take the maximum over $\ell$. By Lemma~\ref{lem:gap}, $
\EE\norm{\nabla F(x_t)}^2\le 2L g_t\le 2L u_t$,
and by Lemma~\ref{lem:delta-sur}, $
\EE\norm{\delta_{i,t}}^2\le 4\chi_t$.
Define
\[
r_{i,t,\ell}:=\bigl(\EE\norm{y_{i,t}^{(\ell)}-x_t}^2\bigr)^{1/2}.
\]
By Minkowski's inequality, the decomposition above, and orthogonality of the local noise,
\[
r_{i,t,\ell}
\le
\eta_i\ell\Bigl((\EE\norm{\nabla F(x_t)}^2)^{1/2}+(\EE\norm{\delta_{i,t}}^2)^{1/2}\Bigr)
+
\eta_i L\sum_{s=0}^{\ell-1} r_{i,t,s}
+
\eta_i\sqrt{\ell}\,\frac{v_i}{\sqrt{b_i}}.
\]
Since $\ell\le H_i$ and $\eta_i=\theta_i/(L H_i)$,
\[
r_{i,t,\ell}
\le
\frac{\theta_i}{L}\Bigl((\EE\norm{\nabla F(x_t)}^2)^{1/2}+(\EE\norm{\delta_{i,t}}^2)^{1/2}\Bigr)
+
\frac{\theta_i v_i}{L\sqrt{H_i b_i}}
+
\frac{\theta_i}{H_i}\sum_{s=0}^{\ell-1} r_{i,t,s}.
\]
Apply Lemma~\ref{lem:cum-gronwall} with $\beta=\theta_i/H_i$ to obtain
\[
r_{i,t,\ell}
\le
\exp(\theta_i)\left[
\frac{\theta_i}{L}\Bigl((\EE\norm{\nabla F(x_t)}^2)^{1/2}+(\EE\norm{\delta_{i,t}}^2)^{1/2}\Bigr)
+
\frac{\theta_i v_i}{L\sqrt{H_i b_i}}
\right].
\]
Using $(a+b+c)^2\le 4a^2+4b^2+2c^2$, Lemma~\ref{lem:gap}, and Lemma~\ref{lem:delta-sur}, we conclude that
\[
\EE\norm{y_{i,t}^{(\ell)}-x_t}^2
\le
8 e^{2\theta_i}\frac{\theta_i^2}{L}u_t
+
16 e^{2\theta_i}\frac{\theta_i^2}{L^2}\chi_t
+
2 e^{2\theta_i}\frac{\theta_i^2v_i^2}{L^2 H_i b_i}.
\]
Taking the maximum over $\ell$ proves \eqref{eq:sur-radius-app}.
\end{proof}

\begin{lemma}[Deterministic mean remainder bound under a surrogate upper state]
\label{lem:sur-rbar}
Assume Assumption~\ref{ass:full-global}, and let $(u_t,\chi_t)$ be a surrogate
upper state. Then, for every deterministic feasible amplitude vector
$\theta\in[\underline\theta,\bar\theta]^n$ and every node $i\in[n]$,
\[
\EE\norm{\bar r_{i,t}}^2
\le
24\frac{\theta_i^4}{L}u_t
+
8\frac{\theta_i^4 v_i^2}{L^2 H_i b_i}
+
48\frac{\theta_i^4}{L^2}\chi_t.
\]
\end{lemma}

\begin{proof}
By Jensen,
\[
\EE\norm{\bar r_{i,t}}^2
=
\EE\norm{\EE[r_{i,t}\mid \cF_t]}^2
\le
\EE\norm{r_{i,t}}^2.
\]
Exactly as in Lemma~\ref{lem:rbar},
\[
\norm{r_{i,t}}^2
\le
\theta_i^2\max_{0\le \ell\le H_i}\norm{y_{i,t}^{(\ell)}-x_t}^2.
\]
Taking expectations and applying Lemma~\ref{lem:sur-radius} yields the result.
\end{proof}

\begin{lemma}[Deterministic centered endpoint variance under a surrogate upper state]
\label{lem:sur-kappa}
Assume Assumption~\ref{ass:full-global}, and let $(u_t,\chi_t)$ be a surrogate
upper state. Then, for every deterministic feasible amplitude vector
$\theta\in[\underline\theta,\bar\theta]^n$ and every node $i\in[n]$,
\begin{equation}
\EE\norm{\zeta_{i,t}}^2
\le
\kappa_i(\theta_i;u_t,\chi_t).
\label{eq:sur-kappa-app}\end{equation}
\end{lemma}

\begin{proof}
By Lemma~\ref{lem:endpoint},
\[
\zeta_{i,t}
=
\bigl(r_{i,t}-\EE[r_{i,t}\mid \cF_t]\bigr)+\xi_{i,t}.
\]
Therefore
\[
\EE\norm{\zeta_{i,t}}^2
\le
2\EE\norm{r_{i,t}}^2+2\EE\norm{\xi_{i,t}}^2.
\]
By Lemma~\ref{lem:sur-rbar},
\[
\EE\norm{r_{i,t}}^2
\le
24\frac{\theta_i^4}{L}u_t
+
8\frac{\theta_i^4v_i^2}{L^2H_i b_i}
+
48\frac{\theta_i^4}{L^2}\chi_t.
\]
For the stochastic term,
\[
\xi_{i,t}
=
-\eta_i\sum_{\ell=0}^{H_i-1}
\bigl(g_{i,t,\ell}-\nabla F_i(y_{i,t}^{(\ell)})\bigr),
\qquad
\eta_i=\frac{\theta_i}{L H_i}.
\]
The summands are conditionally centered and conditionally orthogonal across the
local steps, so
\[
\EE\norm{\xi_{i,t}}^2
\le
\eta_i^2 H_i\frac{v_i^2}{b_i}
=
\frac{\theta_i^2v_i^2}{L^2H_i b_i}.
\]
Combining the last two bounds,
\begin{align*}
\EE\norm{\zeta_{i,t}}^2
&\le
48\frac{\theta_i^4}{L}u_t
+
16\frac{\theta_i^4v_i^2}{L^2H_i b_i}
+
96\frac{\theta_i^4}{L^2}\chi_t
+
2\frac{\theta_i^2v_i^2}{L^2H_i b_i}\\
&\le
64\frac{\theta_i^4}{L}u_t
+
32\frac{\theta_i^2v_i^2}{L^2H_i b_i}
+
128\frac{\theta_i^4}{L^2}\chi_t,
\end{align*}
because $\theta_i\le 1$. This is exactly \eqref{eq:sur-kappa-app}.
\end{proof}

\begin{proposition}[Deterministic nodewise mean certificate under a surrogate upper state]
\label{prop:sur-mean}
Assume Assumption~\ref{ass:full-global}, and let $(u_t,\chi_t)$ be a surrogate
upper state. Then, for every deterministic feasible amplitude vector
$\theta\in[\underline\theta,\bar\theta]^n$ and every node $i\in[n]$,
\begin{equation}
\EE\!\left[
\ip{\nabla F(x_t)}{m_{i,t}}
+
\frac{L}{2}\norm{m_{i,t}}^2
\right]
\le
-
\frac{\theta_i}{2LR^2}g_t^2
+
\rho_i(\theta_i;u_t,\chi_t).
\label{eq:sur-mean-app}\end{equation}
\end{proposition}

\begin{proof}
Fix $i\in[n]$ and write
\[
\theta:=\theta_i,
\qquad
g:=\nabla F(x_t),
\qquad
\delta:=\delta_{i,t},
\qquad
\bar r:=\bar r_{i,t},
\qquad
m:=m_{i,t}.
\]
By Lemma~\ref{lem:endpoint}, $m
=
-\frac{\theta}{L}g
-
\frac{\theta}{L}\delta
+
\bar r$.
Set $
\nu:=-\frac{\theta}{L}\delta+\bar r$,
so that $m=-(\theta/L)g+\nu$. Then
\begin{align}
\ip{g}{m}+\frac{L}{2}\norm{m}^2
&=
-\frac{\theta}{L}\norm{g}^2+\ip{g}{\nu}
+
\frac{L}{2}
\left(
\frac{\theta^2}{L^2}\norm{g}^2
-
2\frac{\theta}{L}\ip{g}{\nu}
+
\norm{\nu}^2
\right)
\notag\\
&=
-\frac{\theta-\theta^2/2}{L}\norm{g}^2
+
(1-\theta)\ip{g}{\nu}
+
\frac{L}{2}\norm{\nu}^2.
\label{eq:sur-mean-expand-app-detail}
\end{align}
By Young's inequality,
\[
(1-\theta)\ip{g}{\nu}
\le
\frac{\theta}{4L}\norm{g}^2+\frac{L}{\theta}\norm{\nu}^2.
\]
Since $0<\theta\le 1$,
\[
-\frac{\theta-\theta^2/2}{L}+\frac{\theta}{4L}
\le
-\frac{\theta}{2L},
\qquad
\frac{L}{\theta}+\frac{L}{2}\le \frac{2L}{\theta}.
\]
Hence \eqref{eq:sur-mean-expand-app-detail} gives
\begin{equation}
\ip{g}{m}+\frac{L}{2}\norm{m}^2
\le
-\frac{\theta}{2L}\norm{g}^2
+
\frac{2L}{\theta}\norm{\nu}^2.
\label{eq:sur-mean-mid-app-detail}\end{equation}
Taking expectations,
\begin{equation}
\EE\!\left[
\ip{g}{m}+\frac{L}{2}\norm{m}^2
\right]
\le
-\frac{\theta}{2L}\EE\norm{g}^2
+
\frac{2L}{\theta}\EE\norm{\nu}^2.
\label{eq:sur-mean-mid-exp-app-detail}\end{equation}
Next, $
\norm{\nu}^2
\le
2\frac{\theta^2}{L^2}\norm{\delta}^2+2\norm{\bar r}^2$.
Taking expectations and using Lemma~\ref{lem:delta-sur} and
Lemma~\ref{lem:sur-rbar},
\[
\EE\norm{\nu}^2
\le
2\frac{\theta^2}{L^2}\cdot 4\chi_t
+
2\left(
24\frac{\theta^4}{L}u_t
+
8\frac{\theta^4v_i^2}{L^2H_i b_i}
+
48\frac{\theta^4}{L^2}\chi_t
\right).
\]
Therefore
\[
\EE\norm{\nu}^2
\le
8\frac{\theta^2}{L^2}\chi_t
+
48\frac{\theta^4}{L}u_t
+
16\frac{\theta^4v_i^2}{L^2H_i b_i}
+
96\frac{\theta^4}{L^2}\chi_t.
\]
Multiplying by $2L/\theta$, and using $\theta\le 1$, yields
\[
\frac{2L}{\theta}\EE\norm{\nu}^2
\le
256\frac{\theta}{L}\chi_t
+
96\theta^2u_t
+
32\frac{\theta^2v_i^2}{L H_i b_i}.
\]
Substituting into \eqref{eq:sur-mean-mid-exp-app-detail}, we obtain
\[
\EE\!\left[
\ip{g}{m}+\frac{L}{2}\norm{m}^2
\right]
\le
-\frac{\theta}{2L}\EE\norm{g}^2
+
256\frac{\theta}{L}\chi_t
+
96\theta^2u_t
+
32\frac{\theta^2v_i^2}{L H_i b_i}.
\]
Finally, Lemma~\ref{lem:gap} and Jensen imply $
\EE\norm{\nabla F(x_t)}^2
\ge
\frac{1}{R^2}\EE[f_t^2]
\ge
\frac{g_t^2}{R^2}$.
Substituting this lower bound proves \eqref{eq:sur-mean-app}.
\end{proof}

\begin{corollary}[Deterministic nodewise upper certificate under a surrogate upper state]
\label{cor:sur-node}
Assume Assumption~\ref{ass:full-global}, and let $(u_t,\chi_t)$ be a surrogate
upper state. Then, for every deterministic feasible amplitude vector
$\theta\in[\underline\theta,\bar\theta]^n$ and every node $i\in[n]$,
\begin{equation}
\EE\!\left[
f_t
+
\ip{\nabla F(x_t)}{m_{i,t}}
+
\frac{L}{2}\norm{m_{i,t}}^2
\right]
\le
T_{A_i(\theta_i)}(u_t)+\rho_i(\theta_i;u_t,\chi_t).
\label{eq:sur-node-app}\end{equation}
\end{corollary}

\begin{proof}
By Proposition~\ref{prop:sur-mean},
\[
\EE\!\left[
f_t+\ip{\nabla F(x_t)}{m_{i,t}}+\frac{L}{2}\norm{m_{i,t}}^2
\right]
\le
g_t-A_i(\theta_i)g_t^2+\rho_i(\theta_i;u_t,\chi_t).
\]
By Lemma~\ref{lem:Ta-majorizes},
\[
g_t-A_i(\theta_i)g_t^2\le T_{A_i(\theta_i)}(g_t).
\]
Since $g_t\le u_t$ and $T_{A_i(\theta_i)}$ is increasing on $\RR_+$,
\[
T_{A_i(\theta_i)}(g_t)\le T_{A_i(\theta_i)}(u_t).
\]
Combining the last two inequalities yields \eqref{eq:sur-node-app}.
\end{proof}

\begin{proposition}[Deterministic surrogate executable domination]
\label{prop:sur-domination}
For every deterministic feasible pair $(w,\theta)$,
\begin{equation}
\mathfrak J_t^{\mathrm{id}}(w,\theta;u,\chi)
\le
\mathfrak J_t^{\mathrm{ex}}(w,\theta;u,\chi).
\label{eq:sur-domination-app}\end{equation}
\end{proposition}

\begin{proof}
Since $\widehat v_{i,t}^2\ge v_i^2$, one has
\[
\widehat\rho_i(\theta_i;u,\chi,\widehat v_{i,t})
\ge
\rho_i(\theta_i;u,\chi),
\qquad
\widehat\kappa_i(\theta_i;u,\chi,\widehat v_{i,t})
\ge
\kappa_i(\theta_i;u,\chi).
\]
Because the score $s_i(u;\theta_i)$ is unchanged,
\[
\widehat\mu_i(\theta_i;u,\chi,\widehat v_{i,t})
=
s_i(u;\theta_i)-\widehat\rho_i(\theta_i;u,\chi,\widehat v_{i,t})
\le
s_i(u;\theta_i)-\rho_i(\theta_i;u,\chi)
=
\mu_i(\theta_i;u,\chi).
\]
Substituting these inequalities into the definitions of
$\mathfrak J_t^{\mathrm{id}}$ and $\mathfrak J_t^{\mathrm{ex}}$ proves
\eqref{eq:sur-domination-app}.
\end{proof}

\begin{theorem}[Deterministic one-step surrogate majorant]
\label{th:sur-majorant}
Assume Assumption~\ref{ass:full-global}, and let $(u_t,\chi_t)$ be a surrogate
upper state. Fix any deterministic feasible pair
$w\in\Delta_n$, $\theta\in[\underline\theta,\bar\theta]^n$. Then
\[
g_{t+1}
\le
\mathfrak J_t^{\mathrm{id}}(w,\theta;u_t,\chi_t).
\]
Consequently, $
g_{t+1}
\le
\mathfrak J_t^{\mathrm{ex}}(w,\theta;u_t,\chi_t)$.
\end{theorem}

\begin{proof}
Under full participation,
\[
x_{t+1}=x_t+\sum_{i=1}^n w_i\Delta_{i,t}=x_t+\bar m_t+\bar\zeta_t.
\]
Applying $L$-smoothness and taking total expectation yields
\begin{equation}
g_{t+1}
\le
g_t
+
\EE\ip{\nabla F(x_t)}{\bar m_t}
+
\frac{L}{2}\EE\norm{\bar m_t}^2
+
\frac{L}{2}\EE\norm{\bar\zeta_t}^2.
\label{eq:sur-majorant-start-app}\end{equation}
By convexity of $\norm{\cdot}^2$, $
\norm{\bar m_t}^2\le \sum_{i=1}^n w_i\norm{m_{i,t}}^2$,
so
\[
g_t+\EE\ip{\nabla F(x_t)}{\bar m_t}+\frac{L}{2}\EE\norm{\bar m_t}^2
\le
\sum_{i=1}^n w_i
\EE\!\left[
f_t+\ip{\nabla F(x_t)}{m_{i,t}}+\frac{L}{2}\norm{m_{i,t}}^2
\right].
\]
Apply Corollary~\ref{cor:sur-node} to each node:
\[
\sum_{i=1}^n w_i
\EE\!\left[
f_t+\ip{\nabla F(x_t)}{m_{i,t}}+\frac{L}{2}\norm{m_{i,t}}^2
\right]
\le
\sum_{i=1}^n w_i\left(T_{A_i(\theta_i)}(u_t)+\rho_i(\theta_i;u_t,\chi_t)\right).
\]
Since $T_{A_i(\theta_i)}(u_t)=u_t-s_i(u_t;\theta_i)$, the right-hand side is
\[
u_t-\sum_{i=1}^n w_i\mu_i(\theta_i;u_t,\chi_t),
\]
For the centered term,
\[
\EE\norm{\bar\zeta_t}^2
=
\sum_{i=1}^n w_i^2\EE\norm{\zeta_{i,t}}^2,
\]
because the cross terms vanish by Lemma~\ref{lem:endpoint}. Apply
Lemma~\ref{lem:sur-kappa} to get
\[
\EE\norm{\bar\zeta_t}^2\le \sum_{i=1}^n w_i^2\kappa_i(\theta_i;u_t,\chi_t).
\]
Substituting into \eqref{eq:sur-majorant-start-app} proves the idealized bound,
and executable domination proves the second statement.
\end{proof}

\begin{proposition}[Tracking recursion under a surrogate upper state]
\label{prop:sur-track}
Assume Assumption~\ref{ass:full-global}, and let $(u_t,\chi_t)$ be a surrogate
upper state. Let $w\in\Delta_n$ and $\theta\in[\underline\theta,\bar\theta]^n$
be any deterministic feasible control pair at round $t$. Then
\[
\max_{1\le i\le n}\EE\norm{e_{i,t+1}}^2
\le
A_\chi+B_\chi u_t+C_\chi\chi_t.
\]
\end{proposition}

\begin{proof}
For each $i\in[n]$,
\begin{align}
\EE\norm{e_{i,t+1}}^2
&\le
3\EE\left\|
\frac1{H_i}\sum_{\ell=0}^{H_i-1}
\bigl(g_{i,t,\ell}-\nabla F_i(y_{i,t}^{(\ell)})\bigr)
\right\|^2
\notag\\
&\quad+
3\EE\left\|
\frac1{H_i}\sum_{\ell=0}^{H_i-1}
\bigl(\nabla F_i(y_{i,t}^{(\ell)})-\nabla F_i(x_t)\bigr)
\right\|^2
\notag\\
&\quad+
3\EE\norm{\nabla F_i(x_t)-\nabla F_i(x_{t+1})}^2.
\label{eq:sur-track-split-app}
\end{align}
The first term is bounded by $\max_j v_j^2/(H_j b_j)$ by conditional
orthogonality across local steps. The second term is bounded by
$L^2\widetilde R_{i,t}^2$. The third term is bounded by
$L^2\EE\norm{x_{t+1}-x_t}^2$, and
\[
\EE\norm{x_{t+1}-x_t}^2
\le
\sum_{j=1}^n w_j\EE\norm{\Delta_{j,t}}^2
\le
\sum_{j=1}^n w_j\widetilde R_{j,t}^2.
\]
Applying Lemma~\ref{lem:sur-radius}, then using $\theta_i\le \bar\theta$ and
$\sum_j w_j=1$, yields the common upper bound
\[
24L\bar\theta^2u_t
+
8\bar\theta^2\max_j\frac{v_j^2}{H_j b_j}
+
48\bar\theta^2\chi_t
\]
for both $L^2\widetilde R_{i,t}^2$ and $L^2\EE\norm{x_{t+1}-x_t}^2$. Plugging
this into \eqref{eq:sur-track-split-app} gives
\[
\EE\norm{e_{i,t+1}}^2
\le
3\max_j\frac{v_j^2}{H_j b_j}
+
48\bar\theta^2\max_j\frac{v_j^2}{H_j b_j}
+
144L\bar\theta^2u_t
+
288\bar\theta^2\chi_t.
\]
Since $48\bar\theta^2\le 1$, the first two terms are bounded by
$6\max_j v_j^2/(H_jb_j)$, which proves the claim.
\end{proof}

\begin{proposition}[Existence of deterministic exact minimizers]
\label{prop:sur-existence}
For every deterministic state pair $(u_t,\chi_t)$, the feasible set
$\Delta_n\times[\underline\theta,\bar\theta]^n$ is compact and the surrogate
objectives $\mathfrak J_t^{\mathrm{id}}$ and $\mathfrak J_t^{\mathrm{ex}}$ are
continuous. Hence deterministic exact minimizers exist.
\end{proposition}

\begin{proof}
This is the deterministic analogue of Proposition~\ref{prop:existence}.
\end{proof}

\begin{theorem}[Proof of Theorem~\ref{th:sur-system}]
\label{th:sur-system-proof}
The conclusions of Theorem~\ref{th:sur-system} hold.
\end{theorem}

\begin{proof}
We argue by induction on $t$. At $t=0$, \eqref{eq:sur-init-main} is exactly the
required upper-state condition.

Assume now that $(u_t,\chi_t)$ is a surrogate upper state. By
Theorem~\ref{th:sur-majorant},
\[
g_{t+1}
\le
\mathfrak J_t^{\mathrm{id}}(w_t^{\mathrm{id}},\theta_t^{\mathrm{id}};u_t,\chi_t).
\]
Since also $g_{t+1}\le \bar f$ by Lemma~\ref{lem:gap},
\[
g_{t+1}
\le
\min\!\left\{
\bar f,
\mathfrak J_t^{\mathrm{id}}(w_t^{\mathrm{id}},\theta_t^{\mathrm{id}};u_t,\chi_t)
\right\}
=u_{t+1},
\]
where the right-hand side is exactly the recursive definition in
\eqref{eq:u-rec-main}. Therefore $g_{t+1}\le u_{t+1}$. By construction,
$u_{t+1}\le \bar f$.

Next, Proposition~\ref{prop:sur-track} applied to the deterministic control
$(w_t^{\mathrm{id}},\theta_t^{\mathrm{id}})$ yields
\[
\max_{1\le i\le n}\EE\norm{e_{i,t+1}}^2
\le
A_\chi+B_\chi u_t+C_\chi\chi_t
=
\chi_{t+1}.
\]
Thus $(u_{t+1},\chi_{t+1})$ is again a surrogate upper state. This proves
\eqref{eq:sur-system-main} for all $t$.

For the executable statement, Theorem~\ref{th:sur-majorant} gives
\[
g_{t+1}
\le
\mathfrak J_t^{\mathrm{ex}}(w_t^{\mathrm{ex}},\theta_t^{\mathrm{ex}};u_t,\chi_t).
\]
Combining this inequality with $g_{t+1}\le \bar f$ yields
\[
g_{t+1}
\le
\min\!\left\{
\bar f,
\mathfrak J_t^{\mathrm{ex}}(w_t^{\mathrm{ex}},\theta_t^{\mathrm{ex}};u_t,\chi_t)
\right\},
\]
which is exactly the executable upper-state domination claimed in
Theorem~\ref{th:sur-system}.
\end{proof}
\begin{corollary}[Exact optimizer is no worse than any deterministic benchmark under the surrogate system]
\label{cor:sur-benchmark}
Under the hypotheses of Theorem~\ref{th:sur-system}, for every deterministic
feasible benchmark pair $(\bar w,\bar\theta)$,
\begin{equation}
u_{t+1}
\le
\mathfrak J_t^{\mathrm{id}}(\bar w,\bar\theta;u_t,\chi_t).
\label{eq:sur-benchmark-app}\end{equation}
The same statement holds for the executable surrogate objective.
\end{corollary}

\begin{proof}
By exact minimization,
\[
\mathfrak J_t^{\mathrm{id}}(w_t^{\mathrm{id}},\theta_t^{\mathrm{id}};u_t,\chi_t)
\le
\mathfrak J_t^{\mathrm{id}}(\bar w,\bar\theta;u_t,\chi_t).
\]
By definition,
\[u_{t+1}
=
\min\!\left\{\bar f,\mathfrak J_t^{\mathrm{id}}(w_t^{\mathrm{id}},\theta_t^{\mathrm{id}};u_t,\chi_t)\right\}.
\]
Therefore
\[u_{t+1}
\le
\mathfrak J_t^{\mathrm{id}}(w_t^{\mathrm{id}},\theta_t^{\mathrm{id}};u_t,\chi_t)
\le
\mathfrak J_t^{\mathrm{id}}(\bar w,\bar\theta;u_t,\chi_t).
\]
This proves \eqref{eq:sur-benchmark-app}. The executable case is identical.
\end{proof}

\begin{lemma}[Global deterministic tracking cap for the surrogate system]
\label{lem:sur-chi-cap}
Under the hypotheses of Theorem~\ref{th:sur-system}, $
\chi_t\le \bar\chi$
for all $t\ge 0$,
where $\bar\chi$ is given by \eqref{eq:chi-bar-main}.
\end{lemma}

\begin{proof}
We argue by induction. At $t=0$, the claim follows from the definition of
$\bar\chi$. Assume $\chi_t\le \bar\chi$. Since $u_t\le \bar f$,
\[
\chi_{t+1}
=
A_\chi+B_\chi u_t+C_\chi\chi_t
\le
A_\chi+B_\chi\bar f+C_\chi\bar\chi.
\]
By the definition of $\bar\chi$,
$(1-C_\chi)\bar\chi\ge A_\chi+B_\chi\bar f$, hence $\chi_{t+1}\le \bar\chi$.
\end{proof}

\begin{theorem}[Proof of Theorem~\ref{th:sur-cvx}]
\label{th:sur-cvx-proof}
The conclusions of Theorem~\ref{th:sur-cvx} hold.
\end{theorem}

\begin{proof}
By Corollary~\ref{cor:sur-benchmark},
\[
u_{t+1}
\le
\mathfrak J_t^{\mathrm{id}}(\bar w,\bar\theta;u_t,\chi_t).
\]
Expanding the surrogate objective,
\[
u_{t+1}
\le
u_t
-
\sum_{i=1}^n \bar w_i s_i(u_t;\bar\theta_i)
+
\sum_{i=1}^n \bar w_i \rho_i(\bar\theta_i;u_t,\chi_t)
+
\frac{L}{2}\sum_{i=1}^n \bar w_i^2 \kappa_i(\bar\theta_i;u_t,\chi_t).
\]
Now
\[
s_i(u_t;\bar\theta_i)
=
\frac{A_i(\bar\theta_i)u_t^2}{1+A_i(\bar\theta_i)u_t}
\ge
\frac{A_i(\bar\theta_i)}{1+A_i(\bar\theta_i)\bar f}u_t^2
=
\underline A_i(\bar\theta_i)u_t^2,
\]
because $u_t\le \bar f$. Summing with weights $\bar w_i$ gives the negative
quadratic term $a_{\mathrm{cvx}}u_t^2$. Expanding the $\rho$- and
$\kappa$-contributions produces \eqref{eq:cvx-rec-main}. The uniform cap on
$\chi_t$ is Lemma~\ref{lem:sur-chi-cap}. Substituting this cap into
\eqref{eq:cvx-rec-main} gives \eqref{eq:cvx-rec-closed-main}. The final rate
bound follows from Lemma~\ref{lem:scalar}.
\end{proof}

\paragraph{Uniform-controller Bellman proofs.}

\begin{lemma}[Aggregate mismatch cancellation]
\label{lem:cancellation}
Under Assumption~\ref{ass:uniform-special}, $
\frac1n\sum_{i=1}^n \delta_{i,t}=0$.
\end{lemma}

\begin{proof}
By the definition of $\delta_{i,t}$,
\[
\frac1n\sum_{i=1}^n \delta_{i,t}
=
\frac1n\sum_{i=1}^n
\bigl(\nabla F_i(x_t)-\nabla F(x_t)\bigr)
-
\frac1n\sum_{i=1}^n(c_{i,t}-c_t).
\]
The first average vanishes because
$\nabla F(x_t)=n^{-1}\sum_i \nabla F_i(x_t)$, and the second average vanishes by
Lemma~\ref{lem:c-average}.
\end{proof}

\begin{lemma}[Uniform direct radius bound]
\label{lem:direct-radius}
Under Assumption~\ref{ass:uniform-special}, for every node $i\in[n]$,
\begin{equation}
R_{i,t}^2
:=
\max_{0\le \ell\le H}\EE\norm{y_{i,t}^{(\ell)}-x_t}^2
\le
\frac{24\vartheta^2}{L}g_t
+
\frac{8\vartheta^2 v^2}{L^2 H b}
+
\frac{48\vartheta^2}{L^2}q_t.
\label{eq:direct-radius-app}\end{equation}
\end{lemma}

\begin{proof}
Fix $i\in[n]$. Under Assumption~\ref{ass:uniform-special}, one has
$H_i\equiv H$, $b_i\equiv b$, and $\theta_i\equiv \vartheta$. Therefore
Lemma~\ref{lem:sur-radius} yields
\[
R_{i,t}^2
\le
\frac{24\vartheta^2}{L}g_t
+
\frac{8\vartheta^2v_i^2}{L^2Hb}
+
\frac{48\vartheta^2}{L^2}q_t.
\]
Since $v_i^2\le v^2$ for every $i$, we may bound the middle term by
$8\vartheta^2v^2/(L^2Hb)$. This gives exactly
\eqref{eq:direct-radius-app}.
\end{proof}

\begin{definition}[Aggregate mean remainder]
\label{def:rbart}
Under Assumption~\ref{ass:uniform-special}, define
\[
\bar r_{i,t}:=\EE[r_{i,t}\mid \cF_t],
\qquad
\bar r_t:=\frac1n\sum_{i=1}^n \bar r_{i,t}.
\]
\end{definition}

\begin{lemma}[Aggregate mean decomposition]
\label{lem:agg-mean}
Under Assumption~\ref{ass:uniform-special},
\[
\bar m_t
:=
\frac1n\sum_{i=1}^n m_{i,t}
=
-\frac{\vartheta}{L}\nabla F(x_t)+\bar r_t.
\]
\end{lemma}

\begin{proof}
Take conditional expectations in Lemma~\ref{lem:endpoint} and average over
$i\in[n]$. The $\delta_{i,t}$ terms cancel by Lemma~\ref{lem:cancellation}.
\end{proof}

\begin{lemma}[Uniform-controller tracking decomposition]
\label{lem:tracking-direct}
Under Assumption~\ref{ass:uniform-special}, for every node $i\in[n]$,
\begin{align*}
e_{i,t+1}
&=
\frac1H\sum_{\ell=0}^{H-1}
\bigl(g_{i,t,\ell}-\nabla F_i(y_{i,t}^{(\ell)})\bigr)
\notag\\
&\quad+
\frac1H\sum_{\ell=0}^{H-1}
\bigl(\nabla F_i(y_{i,t}^{(\ell)})-\nabla F_i(x_t)\bigr)
+
\bigl(\nabla F_i(x_t)-\nabla F_i(x_{t+1})\bigr).
\end{align*}
\end{lemma}

\begin{proof}
Because every node is active and $H_i\equiv H$, Lemma~\ref{lem:avg} gives
$ c_{i,t+1}=H^{-1}\sum_{\ell=0}^{H-1}g_{i,t,\ell}$. Therefore
\[
e_{i,t+1}
=
\frac1H\sum_{\ell=0}^{H-1}g_{i,t,\ell}-\nabla F_i(x_{t+1}).
\]
Insert and subtract $\nabla F_i(y_{i,t}^{(\ell)})$ and $\nabla F_i(x_t)$.
\end{proof}

\begin{proposition}[Direct tracking recursion for the uniform-controller branch]
\label{prop:direct-q}
Under Assumption~\ref{ass:uniform-special}, $q_{t+1}\le A_q+B_q g_t+C_q q_t$,
where
$A_q:=6\frac{v^2}{H b}$,
$B_q:=144L\vartheta^2$,
$C_q:=288\vartheta^2$.
\end{proposition}

\begin{proof}
Starting from Lemma~\ref{lem:tracking-direct}, apply the same three-term split
as in Proposition~\ref{prop:sur-track}. The first term is bounded by
$v^2/(Hb)$, the second by $L^2R_{i,t}^2$, and the third by
$L^2\EE\norm{x_{t+1}-x_t}^2$. Under uniform averaging,
\[
\EE\norm{x_{t+1}-x_t}^2
=
\EE\left\|\frac1n\sum_{j=1}^n \Delta_{j,t}\right\|^2
\le
\frac1n\sum_{j=1}^n \EE\norm{\Delta_{j,t}}^2
\le
\max_j R_{j,t}^2.
\]
Applying Lemma~\ref{lem:direct-radius} to both $R_{i,t}^2$ and $\max_jR_{j,t}^2$
and substituting into the three-term decomposition yields the stated recursion.
\end{proof}

\begin{lemma}[Aggregate mean remainder bound]
\label{lem:agg-r}
Under Assumption~\ref{ass:uniform-special},
\[
\EE\norm{\bar r_t}^2
\le
24\frac{\vartheta^4}{L}g_t
+
8\frac{\vartheta^4}{L^2 H b}v^2
+
48\frac{\vartheta^4}{L^2}q_t.
\]
\end{lemma}

\begin{proof}
By Jensen,
\[
\EE\norm{\bar r_t}^2
\le
\frac1n\sum_{i=1}^n \EE\norm{\bar r_{i,t}}^2
\le
\frac1n\sum_{i=1}^n \EE\norm{r_{i,t}}^2.
\]
As in Lemma~\ref{lem:rbar},
$\EE\norm{r_{i,t}}^2\le \vartheta^2 R_{i,t}^2$. Averaging over $i$ and using
Lemma~\ref{lem:direct-radius} proves the claim.
\end{proof}

\begin{lemma}[Aggregate centered noise bound]
\label{lem:agg-noise}
Under Assumption~\ref{ass:uniform-special},
\[
\frac{L}{2}\EE\!\left\|
\frac1n\sum_{i=1}^n \zeta_{i,t}
\right\|^2
\le
16\frac{\vartheta^2 v^2}{L n H b}
+
32\frac{\vartheta^4}{n}g_t
+
64\frac{\vartheta^4}{L n}q_t.
\]
\end{lemma}

\begin{proof}
By Lemma~\ref{lem:endpoint},
\[
\EE\left\|\frac1n\sum_{i=1}^n \zeta_{i,t}\right\|^2
=
\frac1{n^2}\sum_{i=1}^n \EE\norm{\zeta_{i,t}}^2,
\]
because the cross terms vanish. Apply Lemma~\ref{lem:sur-kappa} specialized to
the uniform controller and multiply by $L/2$.
\end{proof}

\begin{theorem}[Proof of Theorem~\ref{th:direct-bellman}]
\label{th:direct-bellman-proof}
The conclusions of Theorem~\ref{th:direct-bellman} hold.
\end{theorem}

\begin{proof}
Under Assumption~\ref{ass:uniform-special},
$x_{t+1}=x_t+\frac1n\sum_{i=1}^n \Delta_{i,t}$.
Applying Theorem~\ref{th:agg-server} with $w_i=1/n$, then taking total
expectations, yields
\begin{equation}
g_{t+1}
\le
g_t
+
\EE\!\left[
\ip{\nabla F(x_t)}{\bar m_t}
+
\frac{L}{2}\norm{\bar m_t}^2
\right]
+
\frac{L}{2}\EE\!\left\|
\frac1n\sum_{i=1}^n \zeta_{i,t}
\right\|^2.
\label{eq:direct-bellman-start-app}\end{equation}
By Lemma~\ref{lem:agg-mean}, $
\bar m_t=-\frac{\vartheta}{L}\nabla F(x_t)+\bar r_t$.
Therefore
\begin{align*}
\ip{\nabla F(x_t)}{\bar m_t}+\frac{L}{2}\norm{\bar m_t}^2
&=
-\frac{\vartheta-\vartheta^2/2}{L}\norm{\nabla F(x_t)}^2
+
(1-\vartheta)\ip{\nabla F(x_t)}{\bar r_t}
+
\frac{L}{2}\norm{\bar r_t}^2.
\end{align*}
Using Young's inequality,
\[
(1-\vartheta)\ip{\nabla F(x_t)}{\bar r_t}
\le
\frac{\vartheta}{4L}\norm{\nabla F(x_t)}^2
+
\frac{L}{\vartheta}\norm{\bar r_t}^2.
\]
Since $0<\vartheta\le 1$,
\[
-\frac{\vartheta-\vartheta^2/2}{L}+\frac{\vartheta}{4L}\le -\frac{\vartheta}{4L},
\qquad
\frac{L}{\vartheta}+\frac{L}{2}\le \frac{3L}{2\vartheta}.
\]
Hence
\[
\ip{\nabla F(x_t)}{\bar m_t}+\frac{L}{2}\norm{\bar m_t}^2
\le
-\frac{\vartheta}{2L}\norm{\nabla F(x_t)}^2
+
\frac{2L}{\vartheta}\norm{\bar r_t}^2.
\]
Take expectations. By Lemma~\ref{lem:gap},
\[
\EE\norm{\nabla F(x_t)}^2\ge \frac{g_t^2}{R^2}.
\]
By Lemma~\ref{lem:agg-r},
\[
\frac{2L}{\vartheta}\EE\norm{\bar r_t}^2
\le
48\vartheta^3 g_t
+
96\frac{\vartheta^3}{L}q_t
+
16\vartheta^3\frac{v^2}{L H b}.
\]
Finally, invoke Lemma~\ref{lem:agg-noise}. Substituting all three bounds into
\eqref{eq:direct-bellman-start-app} yields the stated coefficients.
\end{proof}
\begin{corollary}[Direct PL Bellman inequality]
\label{cor:direct-pl}
Under Assumptions~\ref{ass:uniform-special} and \ref{ass:PL-benchmark},
\[
g_{t+1}
\le
(1-a_{\mathrm{PL}})g_t
+
\gamma_{\mathrm{dir}}q_t
+
\delta_{\mathrm{dir}},
\]
where $a_{\mathrm{PL}}$ is defined in \eqref{eq:aPL-main}.
\end{corollary}

\begin{proof}
By Assumption~\ref{ass:PL-benchmark}, $\EE\norm{\nabla F(x_t)}^2\ge 2\mu g_t$.
Substitute this lower bound into Theorem~\ref{th:direct-bellman}.
\end{proof}

\begin{lemma}[Auxiliary coefficient bounds for the PL branch]
\label{lem:pl-coeff-bounds}
Under Assumption~\ref{ass:uniform-special},
\begin{equation}
\beta_{\mathrm{dir}}\le 50\vartheta^3,
\qquad
\gamma_{\mathrm{dir}}\le \frac{100\vartheta^3}{L},
\label{eq:pl-coeff-bounds-1-app}\end{equation}
and
\begin{equation}
C_q=288\vartheta^2<1.
\label{eq:pl-coeff-bounds-2-app}\end{equation}
Consequently,
\begin{equation}
1-C_q>0.
\label{eq:pl-coeff-bounds-3-app}\end{equation}
\end{lemma}

\begin{proof}
By definition,
\[
\beta_{\mathrm{dir}}
=
48\vartheta^3+\frac{32\vartheta^4}{n}
\le
48\vartheta^3+32\vartheta^4
\le
48\vartheta^3+\frac{32}{\sqrt{288}}\vartheta^3
<50\vartheta^3,
\]
because $n\ge 1$ and $288\vartheta^2<1$. Similarly,
\[
\gamma_{\mathrm{dir}}
=
\frac{96\vartheta^3}{L}+\frac{64\vartheta^4}{Ln}
\le
\frac{96\vartheta^3}{L}+\frac{64\vartheta^4}{L}
\le
\frac{96\vartheta^3}{L}+\frac{64}{\sqrt{288}}\frac{\vartheta^3}{L}
<
\frac{100\vartheta^3}{L}.
\]
This proves \eqref{eq:pl-coeff-bounds-1-app}. Also, $C_q=288\vartheta^2<1$,
which proves \eqref{eq:pl-coeff-bounds-2-app}. Consequently $1-C_q>0$, which is \eqref{eq:pl-coeff-bounds-3-app}.
\end{proof}

\begin{theorem}[Proof of Theorem~\ref{th:pl-global}]
\label{th:pl-global-proof}
The conclusions of Theorem~\ref{th:pl-global} hold.
\end{theorem}

\begin{proof}
Let $
\nu:=1-C_q$,
$a:=a_{\mathrm{PL}}$,
$B:=B_q$ and
$\gamma:=\gamma_{\mathrm{dir}}$.
The assumptions imply $
a>0$, $
a<\nu$,
$a\nu>B\gamma$.
Define $\rho:=\rho_{\mathrm{PL}}$ and
$\lambda:=\lambda_{\mathrm{PL}}$.
First we verify that $0<\rho<\min\{a,\nu\}$. Since $a\nu>B\gamma$,
\[
(\nu-a)^2+4B\gamma
<
(\nu-a)^2+4a\nu
=
(\nu+a)^2.
\]
Therefore
\[
\sqrt{(\nu-a)^2+4B\gamma}<\nu+a,
\]
which implies $\rho>0$. Since the square root is also strictly larger than
$|\nu-a|$,
\[
\rho
=
\frac{\nu+a-\sqrt{(\nu-a)^2+4B\gamma}}{2}
<
\frac{\nu+a-|\nu-a|}{2}
=
\min\{a,\nu\}.
\]

Next we prove the identities
\begin{equation}
a-B\lambda=\rho,
\qquad
\nu-\frac{\gamma}{\lambda}=\rho.
\label{eq:lambda-rho-identities-app}\end{equation}
From \eqref{eq:lambdaPL-main},
\[
B\lambda
=
\frac{2B\gamma}{\nu-a+\sqrt{(\nu-a)^2+4B\gamma}}.
\]
Multiply numerator and denominator by
$-\nu+a+\sqrt{(\nu-a)^2+4B\gamma}$. Since
\[
\bigl(\nu-a+\sqrt{(\nu-a)^2+4B\gamma}\bigr)
\bigl(-\nu+a+\sqrt{(\nu-a)^2+4B\gamma}\bigr)
=
4B\gamma,
\]
we obtain
\[
B\lambda
=
\frac{-\nu+a+\sqrt{(\nu-a)^2+4B\gamma}}{2}.
\]
Hence
\[
a-B\lambda
=
a-\frac{-\nu+a+\sqrt{(\nu-a)^2+4B\gamma}}{2}
=
\frac{\nu+a-\sqrt{(\nu-a)^2+4B\gamma}}{2}
=
\rho.
\]
This proves the first identity. Also, again from \eqref{eq:lambdaPL-main},
\[
\frac{\gamma}{\lambda}
=
\frac{\nu-a+\sqrt{(\nu-a)^2+4B\gamma}}{2}.
\]
Therefore
\[
\nu-\frac{\gamma}{\lambda}
=
\nu-\frac{\nu-a+\sqrt{(\nu-a)^2+4B\gamma}}{2}
=
\frac{\nu+a-\sqrt{(\nu-a)^2+4B\gamma}}{2}
=
\rho,
\]
which proves the second identity.

By Corollary~\ref{cor:direct-pl} and Proposition~\ref{prop:direct-q},
\[
g_{t+1}\le (1-a)g_t+\gamma q_t+\delta_{\mathrm{dir}},
\qquad
q_{t+1}\le A_q+B g_t+C_q q_t.
\]
Multiply the second inequality by $\lambda$ and add it to the first:
\[
s_{t+1}
\le
(1-a+B\lambda)g_t+(\gamma+C_q\lambda)q_t+\delta_{\mathrm{dir}}+\lambda A_q.
\]
By \eqref{eq:lambda-rho-identities-app}, $1-a+B\lambda=1-\rho$.
Also,
\[
\gamma+C_q\lambda
=
\lambda\left(C_q+\frac{\gamma}{\lambda}\right)
=
\lambda\bigl(C_q+\nu-\rho\bigr)
=
\lambda(1-\rho),
\]
because $\nu=1-C_q$. Therefore
\[
s_{t+1}
\le
(1-\rho)g_t+(1-\rho)\lambda q_t+\delta_{\mathrm{dir}}+\lambda A_q
=
(1-\rho)s_t+\delta_{\mathrm{dir}}+\lambda A_q.
\]
Applying Lemma~\ref{lem:lin-scalar} to the scalar sequence $(s_t)_{t\ge 0}$
with
\[
a=\rho,
\qquad
\delta=\delta_{\mathrm{dir}}+\lambda A_q,
\qquad
m=\frac{\delta_{\mathrm{dir}}+\lambda A_q}{\rho},
\]
yields the explicit rate bound.
\end{proof}

\begin{lemma}[A convenient lower bound on $\rho_{\mathrm{PL}}$]
\label{lem:rho-lower}
Assume the hypotheses of Theorem~\ref{th:pl-global}. If
\[
B_q\gamma_{\mathrm{dir}}\le \frac14 a_{\mathrm{PL}}(1-C_q),
\]
then $\rho_{\mathrm{PL}}\ge \frac14 a_{\mathrm{PL}}$.
\end{lemma}

\begin{proof}
Let $u:=1-C_q$, $a:=a_{\mathrm{PL}}$, and $x:=B_q\gamma_{\mathrm{dir}}$. By
assumption, $x\le au/4$ and $a<u$. We claim that
\[
\sqrt{(u-a)^2+4x}\le u+\frac{a}{2}.
\]
Indeed,
\[
\left(u+\frac{a}{2}\right)^2-\bigl((u-a)^2+4x\bigr)
=
3au-\frac34 a^2-4x.
\]
Since $a\le u$,
\[
3au-\frac34a^2
\ge
3a\left(u-\frac{u}{4}\right)=\frac94 au.
\]
Therefore
\[
\left(u+\frac{a}{2}\right)^2-\bigl((u-a)^2+4x\bigr)
\ge
\frac94au-4x
\ge
\frac94au-au
=
\frac54au>0.
\]
Substituting into the definition of $\rho_{\mathrm{PL}}$ yields the claim.
\end{proof}

\begin{corollary}[Readable safe hyperparameter regime for the stochastic PL branch]
\label{cor:pl-safe}
Assume Assumptions~\ref{ass:uniform-special} and \ref{ass:PL-benchmark}, and
suppose in addition that $2 \mu\vartheta\le L$ and $\vartheta^2\le \frac{\mu}{98L}$.
Then all assumptions of Theorem~\ref{th:pl-global} hold, $
\rho_{\mathrm{PL}}\ge \frac{\mu\vartheta}{8L}$ and
$\lambda_{\mathrm{PL}}\le 4\,\gamma_{\mathrm{dir}}$.
Consequently,
\[
g_t
\le
\left(1-\frac{\mu\vartheta}{8L}\right)^t
\left(g_0+\lambda_{\mathrm{PL}}q_0\right)
+
\frac{\delta_{\mathrm{dir}}+\lambda_{\mathrm{PL}}A_q}{\rho_{\mathrm{PL}}},
\]
and
\[
\frac{\delta_{\mathrm{dir}}+\lambda_{\mathrm{PL}}A_q}{\rho_{\mathrm{PL}}}
=
O\!\left(\frac{\vartheta v^2}{\mu n H b}+\frac{\vartheta^2 v^2}{\mu H b}\right).
\]
\end{corollary}

\begin{proof}
By Lemma~\ref{lem:pl-coeff-bounds}, $
\beta_{\mathrm{dir}}\le 50\vartheta^3$.
Using $\vartheta^2\le \min\!\left\{\mu/(400L),1/576\right\}$, we have $
50\vartheta^3=50\vartheta\vartheta^2\le \frac{\mu\vartheta}{8L}$.
Hence
\[
a_{\mathrm{PL}}=\mu\frac{\vartheta}{L}-\beta_{\mathrm{dir}}\ge \frac{7\mu\vartheta}{8L}>0.
\]
This proves the first condition of Theorem~\ref{th:pl-global}. The second
condition follows from $a_{\mathrm{PL}}\le \mu\vartheta/L\le 1/2$ and $C_q\le 288/576=1/2$, so $a_{\mathrm{PL}}<1-C_q$. For the product condition, Lemma~\ref{lem:pl-coeff-bounds}
implies
\[
B_q\gamma_{\mathrm{dir}}\le 144L\vartheta^2\cdot \frac{100\vartheta^3}{L}=14400\vartheta^5.
\]
Using $\vartheta^2\le \min\!\left\{\mu/(400L),1/576\right\}$ gives
\[
B_q\gamma_{\mathrm{dir}}\le 14400\vartheta^3\frac{\mu}{400L}<\frac{1}{16}\,\frac{\mu\vartheta}{L}.
\]
On the other hand,
\[
a_{\mathrm{PL}}(1-C_q)\ge \frac{7\mu\vartheta}{8L}\cdot \frac12=\frac{7}{16}\,\frac{\mu\vartheta}{L}.
\]
Thus
\[
B_q\gamma_{\mathrm{dir}}<\frac14 a_{\mathrm{PL}}(1-C_q),
\]
so Theorem~\ref{th:pl-global} applies and Lemma~\ref{lem:rho-lower} yields
$\rho_{\mathrm{PL}}\ge a_{\mathrm{PL}}/4\ge \mu\vartheta/(8L)$. Also, since $a_{\mathrm{PL}}\le \mu\vartheta/L\le 1/2$ and $C_q\le 288/576=1/2$,
\[
\lambda_{\mathrm{PL}}
\le
\frac{2\gamma_{\mathrm{dir}}}{1-C_q-a_{\mathrm{PL}}}
\le
4\,\gamma_{\mathrm{dir}}.
\]
The rate formula is then Theorem~\ref{th:pl-global}. Finally,
\[
\delta_{\mathrm{dir}}=
\left(16\vartheta^3+\frac{16\vartheta^2}{n}\right)\frac{v^2}{L H b},
\]
so dividing by $\rho_{\mathrm{PL}}\ge \mu\vartheta/(8L)$ gives
\[
\frac{\delta_{\mathrm{dir}}}{\rho_{\mathrm{PL}}}
=
O\!\left(\frac{\vartheta v^2}{\mu n H b}+\frac{\vartheta^2 v^2}{\mu H b}\right).
\]
Since $\lambda_{\mathrm{PL}}=O(\vartheta^3/L)$ and $A_q=6v^2/(Hb)$,
\[
\frac{\lambda_{\mathrm{PL}}A_q}{\rho_{\mathrm{PL}}}
=
O\!\left(\frac{\vartheta^2 v^2}{\mu H b}\right),
\]
which proves the floor estimate.
\end{proof}
\begin{theorem}[Proof of Theorem~\ref{th:ho-main}]
\label{th:ho-main-proof}
The conclusions of Theorem~\ref{th:ho-main} hold.
\end{theorem}

\begin{proof}
Under Assumption~\ref{ass:uniform-special}, define
\[
\bar y_t^{(\ell)}:=\frac1n\sum_{i=1}^n y_{i,t}^{(\ell)},
\qquad
d_{i,t}^{(\ell)}:=y_{i,t}^{(\ell)}-\bar y_t^{(\ell)}.
\]
Averaging the branch recursion and using Lemma~\ref{lem:c-average} yields
\[
\bar y_t^{(\ell+1)}=\bar y_t^{(\ell)}-\eta\bar g_{t,\ell},
\qquad
\eta:=\frac{\vartheta}{LH},
\qquad
\bar g_{t,\ell}:=\frac1n\sum_{i=1}^n g_{i,t,\ell}.
\]
At $\ell=H$, one has $x_{t+1}=\bar y_t^{(H)}$.

Define
\[
b_{t,\ell}:=\frac1n\sum_{i=1}^n \nabla F_i(y_{i,t}^{(\ell)})-\nabla F(\bar y_t^{(\ell)}),
\qquad
\xi_{t,\ell}:=\bar g_{t,\ell}-\frac1n\sum_{i=1}^n \nabla F_i(y_{i,t}^{(\ell)}).
\]
Conditional on the sigma-field generated by the past within the round,
$\EE[\xi_{t,\ell}\mid \mathcal G_{t,\ell}]=0$ and
$\EE[\norm{\xi_{t,\ell}}^2\mid \mathcal G_{t,\ell}]\le v^2/(nb)$ by
nodewise independence and Assumption~\ref{ass:oracle}. Also,
\[
\bar y_t^{(\ell+1)}
=
\bar y_t^{(\ell)}-\eta\bigl(\nabla F(\bar y_t^{(\ell)})+b_{t,\ell}+\xi_{t,\ell}\bigr).
\]
The disagreement level
\[
\mathfrak D_{t,\ell}^2:=\frac1n\sum_{i=1}^n\norm{d_{i,t}^{(\ell)}}^2
\]
satisfies
\[
\EE[\mathfrak D_{t,\ell}^2]\le 4\max_i R_{i,t}^2
\le
\Delta_t^2,
\]
where $\Delta_t^2$ is defined in \eqref{eq:Delta-t-main}; the first inequality
is Jensen's inequality and the second is Lemma~\ref{lem:direct-radius}. By the
higher-order assumptions,
\[
\EE\norm{b_{t,\ell}}^2\le K_{\mathrm{ho}}\Delta_t^2.
\]
Now apply $L$-smoothness of $F$ to the averaged step
$\bar y_t^{(\ell+1)}=\bar y_t^{(\ell)}-\eta(\nabla F+b+\xi)$, take conditional
expectation, use
\[
-\eta\ip{u}{v}\le \frac{\eta}{4}\norm{u}^2+\eta\norm{v}^2,
\]
and
\[
\EE[\norm{u+v+\xi}^2\mid\mathcal G_{t,\ell}]
\le
2\norm{u}^2+2\norm{v}^2+\frac{v^2}{nb},
\]
then use $L\eta=\vartheta/H\le 1$ together with a fixed numerical absorption constant. This
produces the one-step averaged-branch recursion
\[
z_{t,\ell+1}
\le
z_{t,\ell}-\frac{\eta}{2R^2}z_{t,\ell}^2+2\eta K_{\mathrm{ho}}\Delta_t^2+
\frac{L\eta^2}{2}\frac{v^2}{nb},
\]
where $z_{t,\ell}:=\EE[F(\bar y_t^{(\ell)})-F_\star]$ and Jensen plus
Lemma~\ref{lem:gap} were used to turn
$\EE\norm{\nabla F(\bar y_t^{(\ell)})}^2$ into $z_{t,\ell}^2/R^2$. Applying
Lemma~\ref{lem:Ta-majorizes} and then Lemma~\ref{lem:telescope} over the $H$
local steps yields
\[
g_{t+1}
\le
T_{a_{\mathrm{ho}}}(g_t)
+
\frac{2\vartheta}{L}K_{\mathrm{ho}}\Delta_t^2
+
\frac{\vartheta^2}{2L n H}\frac{v^2}{b},
\]
which is the one-round higher-order recursion.

Next, $g_t\le \bar f$ for every $t$ by Lemma~\ref{lem:gap}. Therefore
\[
T_{a_{\mathrm{ho}}}(g_t)
\le
g_t-\underline a_{\mathrm{ho}}g_t^2,
\qquad
\underline a_{\mathrm{ho}}:=\frac{a_{\mathrm{ho}}}{1+a_{\mathrm{ho}}\bar f}
=
\frac{\vartheta}{2LR^2(1+\vartheta/4)}.
\]
Also, by Proposition~\ref{prop:direct-q}, if
\[
\bar q:=\max\left\{q_0,\frac{A_q+B_q\bar f}{1-C_q}\right\},
\]
then $q_t\le \bar q$ for all $t$ by induction. Hence
\[
\Delta_t^2
\le
\frac{96\vartheta^2}{L}g_t
+
\frac{32\vartheta^2v^2}{L^2 H b}
+
\frac{192\vartheta^2}{L^2}\bar q.
\]
Substituting this into the one-round recursion yields
\eqref{eq:ho-rec-main} with the coefficients from \eqref{eq:ho-coeff-main}.

To derive the best-iterate bound, sum \eqref{eq:ho-rec-main} from $t=0$ to
$T-1$:
\[
g_T-g_0
\le
-\underline a_{\mathrm{ho}}\sum_{t=0}^{T-1}g_t^2
+
\beta_{\mathrm{ho}}\sum_{t=0}^{T-1}g_t
+
T\delta_{\mathrm{ho}}.
\]
Dropping $g_T\ge 0$ and writing $S_T:=\sum_{t=0}^{T-1}g_t$,
\[
\underline a_{\mathrm{ho}}\sum_{t=0}^{T-1}g_t^2
\le
g_0+\beta_{\mathrm{ho}}S_T+T\delta_{\mathrm{ho}}.
\]
By Cauchy--Schwarz, $S_T^2\le T\sum_{t=0}^{T-1}g_t^2$, hence
\[
\frac{\underline a_{\mathrm{ho}}}{T}S_T^2-\beta_{\mathrm{ho}}S_T-(g_0+T\delta_{\mathrm{ho}})\le 0.
\]
Therefore
\[
\frac{S_T}{T}
\le
\frac{
\beta_{\mathrm{ho}}+
\sqrt{\beta_{\mathrm{ho}}^2+4\underline a_{\mathrm{ho}}(\delta_{\mathrm{ho}}+g_0/T)}
}{2\underline a_{\mathrm{ho}}}.
\]
Since $\min_{0\le t\le T-1}g_t\le S_T/T$, the sharper root bound follows.
Applying $\sqrt{x+y}\le \sqrt{x}+\sqrt{y}$ yields
\eqref{eq:ho-explicit-main}. Finally, if $\mathcal H=M=0$, then
$K_{\mathrm{ho}}=0$, so $\beta_{\mathrm{ho}}=0$ and
$\delta_{\mathrm{ho}}=\vartheta^2 v^2/(2L n H b)$, which gives the homogeneous
quadratic benchmark bound.
\end{proof}

\section{Post-local aggregation statements and proofs}

\begin{assumption}[Common round amplitude for the post-local corrected branch]
\label{ass:realized-common-theta}
At round $t$, the corrected local branch uses a common amplitude
\[
\vartheta_t\in[\underline\theta,\bar\theta],
\qquad
\eta_{i,t}=\frac{\vartheta_t}{L H_i},
\qquad
i\in \mathcal S_t.
\]
\end{assumption}

\begin{proposition}[Existence and measurable selection of post-local exact minimizers]
\label{prop:realized-existence}
For every round $t$, the correspondence $\omega\mapsto \Delta_n(\mathcal S_t(\omega))$ has nonempty compact values and measurable graph, and the objective $\Psi_t^{\mathrm{het}}$ is Carath\'eodory on that graph. Therefore an exact minimizer exists for every realization, and there exists an $\mathcal H_t$-measurable selector
\[
w_t^{\mathrm{het}}\in\argmin_{w\in\Delta_n(\mathcal S_t)}\Psi_t^{\mathrm{het}}(w).
\]
\end{proposition}

\paragraph{Expanded labeled forms for the heterogeneous corrected post-local branch.}
The proof references use the following full forms:
\[
g_0\le U_0\le \bar f,
\qquad
Q_0\ge \max_{1\le i\le n}\EE\norm{e_{i,0}}^2.
\]
\[
U_{t+1}
\le
U_t-\underline A_t^{\mathrm{het}}U_t^2+\beta_t^{\mathrm{het}}U_t+d_t^{\mathrm{het}}(a),
\qquad
d_t^{\mathrm{het}}(a):=\gamma_t^{\mathrm{het}}\bar Q+\delta_t^{\mathrm{het}}(a).
\]
If, in addition, $
\vartheta_t\equiv \vartheta$,
$\Lambda_t\equiv \Lambda$,
then
\[
\gamma^{\mathrm{het}}
:=
\frac{1}{2(\Lambda-L)}
+
39\frac{\vartheta}{L}
+
64\frac{\Lambda}{L^2}\vartheta^4,
\qquad
\delta^{\mathrm{het}}(a)
:=
64\frac{\vartheta^3}{L}V_1(a)
+
16\frac{\Lambda\vartheta^2}{L^2}V_2(a),
\]
\[
d^{\mathrm{het}}(a):=\gamma^{\mathrm{het}}\bar Q+\delta^{\mathrm{het}}(a).
\]
If, in addition, $
v_i^2>0$ for every $i\in[n]$,
then the explicit scalar-recursion bound used in the proof is
\[
g_T\le U_T
\le
m^{\mathrm{het}}
+
\frac{1}{((U_0-m^{\mathrm{het}})_+)^{-1}+\underline A^{\mathrm{het}}T},
\]
and, for the variance-optimal comparator $a^\star$,
\[
d^{\mathrm{het}}(a^\star)
=
\gamma^{\mathrm{het}}\bar Q
+
\left(
64n\frac{\vartheta^3}{L}
+
16\frac{\Lambda\vartheta^2}{L^2}
\right)
\left(\sum_{i=1}^n \frac{H_i b_i}{v_i^2}\right)^{-1}.
\]

\begin{assumption}[Homogeneous post-local oracle]
\label{ass:realized-hom-exact}
Assume
\[
F_i\equiv F
\qquad
\text{on }B(x_\star,R)
\qquad
\text{for every }i\in[n].
\]
\end{assumption}

\begin{lemma}[Sufficient realizations of the homogeneous post-local oracle]
\label{lem:postlocal-hom-sufficient}
The homogeneous post-local oracle assumption is implied by the stronger condition $F_i\equiv F$ on $B(x_\star,R)$. It is also implied when each node forms minibatches by uniform sampling from an equal-size IID batch of a common finite sum representing $F$, in which case the conditional mean is $\nabla F(y)$ and the nodewise variance proxy is the within-batch gradient variance.
\end{lemma}

\begin{proof}
If $F_i\equiv F$, then the original oracle assumption gives $\EE[g_{i,t,\ell}\mid\cG_{i,t,\ell}]=\nabla F_i(y_{i,t}^{(\ell)})=\nabla F(y_{i,t}^{(\ell)})$, and the same conditional variance bound is unchanged. The IID-batch claim is the standard unbiased minibatch identity for sampling from a common empirical objective.
\end{proof}

\begin{proposition}[Existence and measurable selection of post-local exact minimizers: homogeneous no-CV branch]
\label{prop:realized-hom-existence}
For every round $t$, the function $w\mapsto \Psi_t^{\mathrm{hom}}(w)$ is a Carath\'eodory function of the round-$t$ realized data and is continuous and convex on the compact simplex $\Delta_n$. Therefore an exact minimizer exists. Moreover, there exists an $\mathcal H_t$-measurable selector
\[
w_t^{\mathrm{hom}}\in\argmin_{w\in\Delta_n}\Psi_t^{\mathrm{hom}}(w).
\]
\end{proposition}

\begin{theorem}[Post-local one-step certificate and exact uniform-comparator identity: homogeneous no-CV branch]
\label{th:realized-hom-one-step}
Fix a round $t$ and define $
\eta_t:=\Lambda_t-L>0$.
Let $
\mathcal H_t
:=
\sigma\!\left(
\mathcal F_t,\ \{\mathcal B_{i,t,\ell}:1\le i\le n,\ 0\le \ell\le H_i-1\}
\right)$.
Define $
g_{i,t}^{\mathrm{loc}}
:=
-\frac{1}{\eta_{i,t}H_i}\Delta_{i,t}$,
$\bar g_t:=\frac1n\sum_{i=1}^n g_{i,t}^{\mathrm{loc}}$.
Let $
u:=\left(\frac1n,\dots,\frac1n\right)\in\Delta_n$.
Then, for every $\mathcal H_t$-measurable $w_t\in\Delta_n$,
\begin{equation}
F(x_t+d_t(w_t))-F_\star
\le
f_t+\Psi_t^{\mathrm{hom}}(w_t)+\frac{1}{2\eta_t}\norm{\nabla F(x_t)-\bar g_t}^2.
\label{eq:realized-hom-pathwise}\end{equation}
If $w_t^{\mathrm{hom}}$ is an exact minimizer, then
\begin{equation}
F(x_{t+1})-F_\star
\le
f_t+\Psi_t^{\mathrm{hom}}(\nu)-\Gamma_t^{\mathrm{hom}}
+\frac{1}{2\eta_t}\norm{\nabla F(x_t)-\bar g_t}^2,
\label{eq:realized-hom-comp}\end{equation}
where
\[
\Gamma_t^{\mathrm{hom}}
:=
\Psi_t^{\mathrm{hom}}(\nu)-\Psi_t^{\mathrm{hom}}(w_t^{\mathrm{hom}})
\ge 0.
\]
Moreover,
\begin{equation}
d_t(\nu)=\frac1n\sum_{i=1}^n \Delta_{i,t}
=
-\frac{\vartheta_t}{L}\bar g_t.
\label{eq:realized-hom-id}\end{equation}
\end{theorem}

\paragraph{Expanded labeled forms for the homogeneous no-CV post-local branch.}
The proof references use the following labeled equations:
\[
g_{t+1}
\le
g_t-\underline A_t^{\mathrm{hom}}g_t^2+\beta_t^{\mathrm{hom}}g_t+\delta_t^{\mathrm{hom}}.
\]
If, in addition, $
\vartheta_t\equiv \vartheta$, $\Lambda_t\equiv \Lambda$,
then
\[
g_T
\le
m^{\mathrm{hom}}
+
\frac{1}{((g_0-m^{\mathrm{hom}})_+)^{-1}+\underline A^{\mathrm{hom}}T}.
\]

\begin{assumption}[PL condition for the homogeneous no-control-variate post-local branch]
\label{ass:realized-hom-PL}
Assume
\[
\norm{\nabla F(x)}^2\ge 2\mu(F(x)-F_\star)
\qquad
\text{for every }x\in B(x_\star,R).
\]
\end{assumption}

\begin{corollary}[PL rate with unequal local horizons: homogeneous no-control-variate post-local controller]

Under Assumptions~\ref{ass:realized-hom-exact}, \ref{ass:realized-common-theta}, and \ref{ass:realized-hom-PL},
\begin{equation}
g_{t+1}
\le
(1-\rho_t^{\mathrm{hom}})g_t+\delta_t^{\mathrm{hom}},
\qquad
\rho_t^{\mathrm{hom}}:=\frac{\mu\vartheta_t}{2L}-\beta_t^{\mathrm{hom}}.\end{equation}
If, in addition, \eqref{eq:realized-het-const-param} holds and $\rho^{\mathrm{hom}}
:=
\frac{\mu\vartheta}{2L}-\beta^{\mathrm{hom}}>0$,
then
\[
g_t
\le
(1-\rho^{\mathrm{hom}})^t g_0+\frac{\delta^{\mathrm{hom}}}{\rho^{\mathrm{hom}}}.
\]
\end{corollary}

\paragraph{Auxiliary existence statements.}

\begin{proof}[Proof of Proposition~\ref{prop:realized-existence}]
For every subset $S\subseteq[n]$, the set
\[
\Delta_n(S)
=
\left\{
w\in\RR^n_+:
\sum_{i=1}^n w_i=1,\ 
w_i=0 \ \text{for } i\notin S
\right\}
\]
is a nonempty compact subset of $\RR^n$. Since $\mathcal S_t$ is finite-valued and
$\mathcal F_t$-measurable, the graph
\[
\operatorname{Gr}(\Delta_n(\mathcal S_t))
=
\bigcup_{S\subseteq[n]}
\Bigl(
\{\omega:\mathcal S_t(\omega)=S\}\times \Delta_n(S)
\Bigr)
\]
is measurable in $\mathcal H_t\otimes \mathcal B(\RR^n)$.

Next, for every $w\in\RR^n$,
\[
d_t(w;\omega)=\sum_{i=1}^n w_i \widetilde\Delta_{i,t}(\omega)
\]
is measurable in $\omega$ and affine in $w$. Since $(c_t,\Lambda_t)$ is $\mathcal F_t$-measurable
and hence $\mathcal H_t$-measurable, it follows that
\[
(\omega,w)\mapsto
\Psi_t^{\mathrm{het}}(w;\omega)
=
\ip{c_t(\omega)}{d_t(w;\omega)}
+
\frac{\Lambda_t(\omega)}{2}\norm{d_t(w;\omega)}^2
\]
is measurable in $\omega$ and continuous in $w$. Thus it is Carath\'eodory on
$\operatorname{Gr}(\Delta_n(\mathcal S_t))$.

Pointwise existence of a minimizer follows from continuity on the compact feasible set
$\Delta_n(\mathcal S_t(\omega))$. The measurable maximum theorem then yields an
$\mathcal H_t$-measurable exact minimizer.
\end{proof}

\begin{proof}[Proof of Proposition~\ref{prop:realized-hom-existence}]
Let $(\omega,w)\mapsto d_t(w;\omega)=\sum_{i=1}^n w_i\Delta_{i,t}(\omega)$.
This map is measurable in $\omega$ and affine in $w$. Hence
\[
(\omega,w)\mapsto
\Psi_t^{\mathrm{hom}}(w;\omega)
=
\ip{\bar g_t(\omega)}{d_t(w;\omega)}
+
\frac{\Lambda_t(\omega)}{2}\norm{d_t(w;\omega)}^2
\]
is measurable in $\omega$ and continuous in $w$, i.e. Carath\'eodory on the compact simplex
$\Delta_n$. Pointwise existence of a minimizer follows from continuity on a compact set, and the
measurable maximum theorem yields an $\mathcal H_t$-measurable exact minimizer.
\end{proof}
\paragraph{Heterogeneous corrected post-local controller proofs.}

\begin{proof}[Proof of Theorem~\ref{th:realized-het-one-step}]
Fix a round $t$ and an $\mathcal H_t$-measurable $w_t\in\Delta_n(\mathcal S_t)$. Write $d:=d_t(w_t)$.
By $L$-smoothness of $F$,
\begin{equation}
F(x_t+d)-F_\star
\le
f_t+\ip{\nabla F(x_t)}{d}+\frac{L}{2}\norm{d}^2.
\label{eq:realized-het-proof-1}\end{equation}
Since $\eta_t=\Lambda_t-L$,
\[
\ip{\nabla F(x_t)}{d}
=
\ip{c_t}{d}+\ip{\nabla F(x_t)-c_t}{d}.
\]
By Young's inequality,
\begin{equation}
\ip{\nabla F(x_t)-c_t}{d}
\le
\frac{1}{2\eta_t}\norm{\nabla F(x_t)-c_t}^2
+
\frac{\eta_t}{2}\norm{d}^2.
\label{eq:realized-het-proof-2}\end{equation}
Substituting \eqref{eq:realized-het-proof-2} into \eqref{eq:realized-het-proof-1} gives
\[
F(x_t+d)-F_\star
\le
f_t+\ip{c_t}{d}+\frac{\Lambda_t}{2}\norm{d}^2+\frac{1}{2\eta_t}\norm{\nabla F(x_t)-c_t}^2
=
f_t+\Psi_t^{\mathrm{het}}(w_t)+\frac{1}{2\eta_t}\norm{\nabla F(x_t)-c_t}^2.
\]
This proves \eqref{eq:realized-het-pathwise}.

If $w_t^{\mathrm{het}}$ minimizes $\Psi_t^{\mathrm{het}}$ on $\Delta_n(\mathcal S_t)$, then for every
comparator $a_t\in\Delta_n(\mathcal S_t)$,
\[
\Psi_t^{\mathrm{het}}(w_t^{\mathrm{het}})
\le
\Psi_t^{\mathrm{het}}(a_t).
\]
Equivalently,
\[
\Psi_t^{\mathrm{het}}(w_t^{\mathrm{het}})
=
\Psi_t^{\mathrm{het}}(a_t)-\Gamma_t^{\mathrm{het}}(a_t),
\qquad
\Gamma_t^{\mathrm{het}}(a_t)\ge 0.
\]
Substituting this identity into \eqref{eq:realized-het-pathwise} yields the comparator inequality.
\end{proof}

\begin{lemma}[Server-average tracking error bound]
\label{lem:realized-het-server-error}
Under Assumption~\ref{ass:init}, if $\max_{1\le i\le n}\EE\norm{e_{i,t}}^2\le Q_t$,
then
\begin{equation}
\EE\norm{c_t-\nabla F(x_t)}^2\le Q_t.
\label{eq:realized-het-server-error}\end{equation}
\end{lemma}

\begin{proof}
By Lemma~\ref{lem:c-average},
\[
c_t-\nabla F(x_t)
=
\frac1n\sum_{i=1}^n c_{i,t}-\frac1n\sum_{i=1}^n \nabla F_i(x_t)
=
\frac1n\sum_{i=1}^n e_{i,t}.
\]
Therefore
\[
\norm{c_t-\nabla F(x_t)}^2
=
\left\|
\frac1n\sum_{i=1}^n e_{i,t}
\right\|^2
\le
\frac1n\sum_{i=1}^n \norm{e_{i,t}}^2.
\]
Taking expectations yields \eqref{eq:realized-het-server-error}.
\end{proof}

\begin{lemma}[Comparator decomposition for the heterogeneous corrected post-local controller]
\label{lem:realized-het-decomp}
Assume Assumption~\ref{ass:full-global}. Fix a deterministic comparator $a\in\Delta_n$ and define
\[
\bar m_t(a):=\sum_{i=1}^n a_i m_{i,t},
\qquad
\bar\zeta_t(a):=\sum_{i=1}^n a_i \zeta_{i,t},
\qquad
d_t(a):=\sum_{i=1}^n a_i \Delta_{i,t}.
\]
Then
\begin{equation}
d_t(a)=\bar m_t(a)+\bar\zeta_t(a),
\label{eq:realized-het-decomp-1}\end{equation}
and
\begin{equation}
\bar m_t(a)
=
-\frac{\vartheta_t}{L}\nabla F(x_t)+\nu_t(a),
\label{eq:realized-het-decomp-2}\end{equation}
where
\begin{equation}
\nu_t(a)
:=
-\frac{\vartheta_t}{L}\sum_{i=1}^n a_i \delta_{i,t}
+
\sum_{i=1}^n a_i \bar r_{i,t}.
\label{eq:realized-het-nu-def}\end{equation}
Moreover,
\begin{equation}
\EE[\Psi_t^{\mathrm{het}}(a)\mid \mathcal F_t]
=
\ip{c_t}{\bar m_t(a)}
+
\frac{\Lambda_t}{2}\norm{\bar m_t(a)}^2
+
\frac{\Lambda_t}{2}\EE\!\left[\norm{\bar\zeta_t(a)}^2\middle|\mathcal F_t\right].
\label{eq:realized-het-psi-cond}\end{equation}
\end{lemma}

\begin{proof}
By Definition~\ref{def:agg}, $\Delta_{i,t}=m_{i,t}+\zeta_{i,t}$.
Averaging with coefficients $a_i$ gives \eqref{eq:realized-het-decomp-1}. Since the amplitude is common,
Lemma~\ref{lem:endpoint} gives
\[
m_{i,t}
=
-\frac{\vartheta_t}{L}\nabla F(x_t)
-
\frac{\vartheta_t}{L}\delta_{i,t}
+
\bar r_{i,t}.
\]
Averaging with coefficients $a_i$ proves \eqref{eq:realized-het-decomp-2} and
\eqref{eq:realized-het-nu-def}.

Finally,
\[
\Psi_t^{\mathrm{het}}(a)
=
\ip{c_t}{\bar m_t(a)+\bar\zeta_t(a)}
+
\frac{\Lambda_t}{2}\norm{\bar m_t(a)+\bar\zeta_t(a)}^2.
\]
Taking conditional expectation given $\mathcal F_t$ and using $\EE[\bar\zeta_t(a)\mid \mathcal F_t]=0$
from Lemma~\ref{lem:endpoint} proves \eqref{eq:realized-het-psi-cond}.
\end{proof}

\begin{lemma}[Moment bounds for the heterogeneous comparator branch]
\label{lem:realized-het-moments}
Assume Assumption~\ref{ass:full-global}. Let
\[
g_t\le U_t,
\qquad
\max_{1\le i\le n}\EE\norm{e_{i,t}}^2\le Q_t.
\]
Fix a deterministic comparator $a\in\Delta_n$. Then
\begin{equation}
\EE\norm{\nu_t(a)}^2
\le
48\frac{\vartheta_t^4}{L}U_t
+
9\frac{\vartheta_t^2}{L^2}Q_t
+
16\frac{\vartheta_t^4}{L^2}V_1(a),
\label{eq:realized-het-moment-nu}\end{equation}
and
\begin{equation}
\EE\norm{\bar\zeta_t(a)}^2
\le
64\frac{\vartheta_t^4}{L}U_t
+
128\frac{\vartheta_t^4}{L^2}Q_t
+
32\frac{\vartheta_t^2}{L^2}V_2(a).
\label{eq:realized-het-moment-zeta}\end{equation}
\end{lemma}

\begin{proof}
From \eqref{eq:realized-het-nu-def},
\[
\norm{\nu_t(a)}^2
\le
2\frac{\vartheta_t^2}{L^2}
\left\|
\sum_{i=1}^n a_i\delta_{i,t}
\right\|^2
+
2
\left\|
\sum_{i=1}^n a_i \bar r_{i,t}
\right\|^2.
\]
Since $a\in\Delta_n$, convexity of $\norm{\cdot}^2$ implies
\[
\left\|
\sum_{i=1}^n a_i\delta_{i,t}
\right\|^2
\le
\sum_{i=1}^n a_i \norm{\delta_{i,t}}^2.
\]
Taking expectations and applying Lemma~\ref{lem:delta-sur},
\begin{equation}
\EE
\left\|
\sum_{i=1}^n a_i\delta_{i,t}
\right\|^2
\le 4Q_t.
\label{eq:realized-het-moment-delta}\end{equation}
Likewise,
\[
\left\|
\sum_{i=1}^n a_i\bar r_{i,t}
\right\|^2
\le
\sum_{i=1}^n a_i \norm{\bar r_{i,t}}^2.
\]
Taking expectations and applying Lemma~\ref{lem:sur-rbar},
\begin{equation}
\EE
\left\|
\sum_{i=1}^n a_i\bar r_{i,t}
\right\|^2
\le
24\frac{\vartheta_t^4}{L}U_t
+
8\frac{\vartheta_t^4}{L^2}V_1(a)
+
48\frac{\vartheta_t^4}{L^2}Q_t.
\label{eq:realized-het-moment-r}\end{equation}
Combining \eqref{eq:realized-het-moment-delta} and \eqref{eq:realized-het-moment-r},
\[
\EE\norm{\nu_t(a)}^2
\le
8\frac{\vartheta_t^2}{L^2}Q_t
+
48\frac{\vartheta_t^4}{L}U_t
+
16\frac{\vartheta_t^4}{L^2}V_1(a)
+
96\frac{\vartheta_t^4}{L^2}Q_t.
\]
Since $288\bar\theta^2<1$, one has $96\vartheta_t^4\le 96\bar\theta^2\vartheta_t^2<\vartheta_t^2$, hence
\[
8\frac{\vartheta_t^2}{L^2}Q_t
+
96\frac{\vartheta_t^4}{L^2}Q_t
\le
9\frac{\vartheta_t^2}{L^2}Q_t.
\]
This proves \eqref{eq:realized-het-moment-nu}.

For the centered term, Lemma~\ref{lem:endpoint} implies
\[
\EE\norm{\bar\zeta_t(a)}^2
=
\sum_{i=1}^n a_i^2\EE\norm{\zeta_{i,t}}^2.
\]
Applying Lemma~\ref{lem:sur-kappa},
\[
\EE\norm{\zeta_{i,t}}^2
\le
32\frac{\vartheta_t^2v_i^2}{L^2H_i b_i}
+
64\frac{\vartheta_t^4}{L}U_t
+
128\frac{\vartheta_t^4}{L^2}Q_t.
\]
Hence
\[
\EE\norm{\bar\zeta_t(a)}^2
\le
32\frac{\vartheta_t^2}{L^2}V_2(a)
+
64\frac{\vartheta_t^4}{L}U_t\sum_{i=1}^n a_i^2
+
128\frac{\vartheta_t^4}{L^2}Q_t\sum_{i=1}^n a_i^2.
\]
Since $\sum_{i=1}^n a_i^2\le 1$, this proves \eqref{eq:realized-het-moment-zeta}.
\end{proof}

\begin{lemma}[Tracking recursion for post-local simplex weights]
\label{lem:realized-het-tracking}
Assume Assumptions~\ref{ass:full-global} and \ref{ass:realized-common-theta}. Let
\[
g_t\le U_t,
\qquad
\max_{1\le i\le n}\EE\norm{e_{i,t}}^2\le Q_t.
\]
Let $w_t$ be any $\mathcal H_t$-measurable random vector satisfying $w_t\in\Delta_n$ almost surely.
Set
\begin{equation}
x_{t+1}=x_t+\sum_{j=1}^n w_{j,t}\Delta_{j,t}.
\label{eq:realized-het-tracking-update}\end{equation}
Then
\begin{equation}
\max_{1\le i\le n}\EE\norm{e_{i,t+1}}^2
\le
A_\chi+B_\chi U_t+C_\chi Q_t,
\label{eq:realized-het-tracking}\end{equation}
where
\[
A_\chi:=6\max_{1\le i\le n}\frac{v_i^2}{H_i b_i},
\qquad
B_\chi:=144L\bar\theta^2,
\qquad
C_\chi:=288\bar\theta^2.
\]
\end{lemma}

\begin{proof}
Fix $i\in[n]$. By Lemma~\ref{lem:avg}, $c_{i,t+1}
=
\frac1{H_i}\sum_{\ell=0}^{H_i-1} g_{i,t,\ell}$.
Therefore
\begin{align}
\EE\norm{e_{i,t+1}}^2
&=
\EE\norm{c_{i,t+1}-\nabla F_i(x_{t+1})}^2
\notag\\
&\le
3\EE\left\|
\frac1{H_i}\sum_{\ell=0}^{H_i-1}
\bigl(g_{i,t,\ell}-\nabla F_i(y_{i,t}^{(\ell)})\bigr)
\right\|^2
\notag\\
&\quad+
3\EE\left\|
\frac1{H_i}\sum_{\ell=0}^{H_i-1}
\bigl(\nabla F_i(y_{i,t}^{(\ell)})-\nabla F_i(x_t)\bigr)
\right\|^2
\notag\\
&\quad+
3\EE\norm{\nabla F_i(x_t)-\nabla F_i(x_{t+1})}^2.
\label{eq:realized-het-tracking-split}
\end{align}

For the first term, conditional orthogonality across local steps yields
\[
\EE\left\|
\frac1{H_i}\sum_{\ell=0}^{H_i-1}
\bigl(g_{i,t,\ell}-\nabla F_i(y_{i,t}^{(\ell)})\bigr)
\right\|^2
\le
\frac{v_i^2}{H_i b_i}
\le
\max_{1\le j\le n}\frac{v_j^2}{H_j b_j}.
\]

For the second term,
\[
\left\|
\frac1{H_i}\sum_{\ell=0}^{H_i-1}
\bigl(\nabla F_i(y_{i,t}^{(\ell)})-\nabla F_i(x_t)\bigr)
\right\|^2
\le
\frac1{H_i}\sum_{\ell=0}^{H_i-1}
L^2\norm{y_{i,t}^{(\ell)}-x_t}^2.
\]
Taking expectations and applying Lemma~\ref{lem:sur-radius} with $\theta_j=\vartheta_t$ for all $j$,
\[
\EE\left\|
\frac1{H_i}\sum_{\ell=0}^{H_i-1}
\bigl(\nabla F_i(y_{i,t}^{(\ell)})-\nabla F_i(x_t)\bigr)
\right\|^2
\le
24L\vartheta_t^2U_t
+
8\vartheta_t^2\max_{1\le j\le n}\frac{v_j^2}{H_j b_j}
+
48\vartheta_t^2Q_t.
\]

For the third term, by $L$-smoothness,
\[
\EE\norm{\nabla F_i(x_t)-\nabla F_i(x_{t+1})}^2
\le
L^2\EE\norm{x_{t+1}-x_t}^2.
\]
Using \eqref{eq:realized-het-tracking-update} and the fact that $w_t\in\Delta_n$ almost surely,
\[
\norm{x_{t+1}-x_t}^2
=
\left\|
\sum_{j=1}^n w_{j,t}\Delta_{j,t}
\right\|^2
\le
\sum_{j=1}^n w_{j,t}\norm{\Delta_{j,t}}^2
\le
\max_{1\le j\le n}\norm{\Delta_{j,t}}^2.
\]
Taking expectations and using $\EE\norm{\Delta_{j,t}}^2\le \widetilde R_{j,t}^2$,
we have $\EE\norm{x_{t+1}-x_t}^2
\le
\max_{1\le j\le n}\widetilde R_{j,t}^2$.
Applying Lemma~\ref{lem:sur-radius},
\[
L^2\EE\norm{x_{t+1}-x_t}^2
\le
24L\vartheta_t^2U_t
+
8\vartheta_t^2\max_{1\le j\le n}\frac{v_j^2}{H_j b_j}
+
48\vartheta_t^2Q_t.
\]

Substituting the three bounds into \eqref{eq:realized-het-tracking-split},
\begin{align*}
\EE\norm{e_{i,t+1}}^2
&\le
3\max_{1\le j\le n}\frac{v_j^2}{H_j b_j}
+
48\vartheta_t^2\max_{1\le j\le n}\frac{v_j^2}{H_j b_j}
+
144L\vartheta_t^2U_t
+
288\vartheta_t^2Q_t.
\end{align*}
Since $288\bar\theta^2<1$, one has $48\vartheta_t^2\le 48\bar\theta^2<1$, so
\[
3\max_j\frac{v_j^2}{H_j b_j}
+
48\vartheta_t^2\max_j\frac{v_j^2}{H_j b_j}
\le
6\max_j\frac{v_j^2}{H_j b_j}.
\]
Using $\vartheta_t\le\bar\theta$ in the remaining terms proves \eqref{eq:realized-het-tracking}.
\end{proof}

\begin{proof}[Proof of Theorem~\ref{th:realized-het-global}]
We argue by induction on $t$. At $t=0$, \eqref{eq:realized-het-init} is exactly the desired upper-state
condition.

Fix a deterministic comparator $a\in\Delta_n$. Assume
\[
g_t\le U_t\le \bar f,
\qquad
\max_{1\le i\le n}\EE\norm{e_{i,t}}^2\le Q_t.
\]
By Theorem~\ref{th:realized-het-one-step},
\[
F(x_{t+1})-F_\star
\le
f_t+\Psi_t^{\mathrm{het}}(a)-\Gamma_t^{\mathrm{het}}(a)
+\frac{1}{2(\Lambda_t-L)}\norm{\nabla F(x_t)-c_t}^2.
\]
Taking expectations, dropping the nonnegative term $\Gamma_t^{\mathrm{het}}(a)$, and applying
Lemma~\ref{lem:realized-het-server-error},
\begin{equation}
g_{t+1}
\le
g_t+\EE[\Psi_t^{\mathrm{het}}(a)]
+\frac{1}{2(\Lambda_t-L)}Q_t.
\label{eq:realized-het-global-start}\end{equation}

Define $\bar e_t:=c_t-\nabla F(x_t)$.
By Lemma~\ref{lem:realized-het-decomp},
\begin{align*}
\EE[\Psi_t^{\mathrm{het}}(a)]
&=
\EE\ip{\nabla F(x_t)+\bar e_t}{-\frac{\vartheta_t}{L}\nabla F(x_t)+\nu_t(a)}\\
&\quad+
\frac{\Lambda_t}{2}
\EE\left\|-\frac{\vartheta_t}{L}\nabla F(x_t)+\nu_t(a)\right\|^2
+
\frac{\Lambda_t}{2}\EE\norm{\bar\zeta_t(a)}^2.
\end{align*}
Expanding,
\begin{align*}
\EE[\Psi_t^{\mathrm{het}}(a)]
=
&-\frac{\vartheta_t}{L}\EE\norm{\nabla F(x_t)}^2
+\EE\ip{\nabla F(x_t)}{\nu_t(a)}
-\frac{\vartheta_t}{L}\EE\ip{\bar e_t}{\nabla F(x_t)}
+\EE\ip{\bar e_t}{\nu_t(a)}
\notag\\
&+
\frac{\Lambda_t\vartheta_t^2}{2L^2}\EE\norm{\nabla F(x_t)}^2
-
\frac{\Lambda_t\vartheta_t}{L}\EE\ip{\nabla F(x_t)}{\nu_t(a)}\\
&+
\frac{\Lambda_t}{2}\EE\norm{\nu_t(a)}^2
+
\frac{\Lambda_t}{2}\EE\norm{\bar\zeta_t(a)}^2.
\end{align*}
Since $\Lambda_t\vartheta_t\le L/2$,
\[
-\frac{\vartheta_t}{L}+\frac{\Lambda_t\vartheta_t^2}{2L^2}
\le
-\frac{3\vartheta_t}{4L},
\qquad
\left|1-\frac{\Lambda_t\vartheta_t}{L}\right|\le 1.
\]
Therefore
\begin{align}
\EE[\Psi_t^{\mathrm{het}}(a)]
&\le
-\frac{3\vartheta_t}{4L}\EE\norm{\nabla F(x_t)}^2
+
\left|\EE\ip{\nabla F(x_t)}{\nu_t(a)}\right|
+
\frac{\vartheta_t}{L}\left|\EE\ip{\bar e_t}{\nabla F(x_t)}\right|
+
\left|\EE\ip{\bar e_t}{\nu_t(a)}\right|
\notag\\
&\quad+
\frac{\Lambda_t}{2}\EE\norm{\nu_t(a)}^2
+
\frac{\Lambda_t}{2}\EE\norm{\bar\zeta_t(a)}^2.
\label{eq:realized-het-reduced}
\end{align}
Using $\left|\EE Z\right|\le \EE|Z|$
and then Young's inequality pointwise,
\begin{equation}
\left|\EE\ip{\nabla F(x_t)}{\nu_t(a)}\right|
\le
\frac{\vartheta_t}{8L}\EE\norm{\nabla F(x_t)}^2
+
\frac{2L}{\vartheta_t}\EE\norm{\nu_t(a)}^2,
\label{eq:realized-het-young1}\end{equation}
\[
\frac{\vartheta_t}{L}\left|\EE\ip{\bar e_t}{\nabla F(x_t)}\right|
\le
\frac{\vartheta_t}{8L}\EE\norm{\nabla F(x_t)}^2
+
\frac{2\vartheta_t}{L}\EE\norm{\bar e_t}^2,
\]
\begin{equation}
\left|\EE\ip{\bar e_t}{\nu_t(a)}\right|
\le
\frac{\vartheta_t}{L}\EE\norm{\bar e_t}^2
+
\frac{L}{4\vartheta_t}\EE\norm{\nu_t(a)}^2.
\label{eq:realized-het-young3}\end{equation}
Substituting \eqref{eq:realized-het-young1}--\eqref{eq:realized-het-young3} into
\eqref{eq:realized-het-reduced},
\begin{align*}
\EE[\Psi_t^{\mathrm{het}}(a)]
&\le
-\frac{\vartheta_t}{2L}\EE\norm{\nabla F(x_t)}^2
+
\left(
\frac{9L}{4\vartheta_t}
+
\frac{\Lambda_t}{2}
\right)\EE\norm{\nu_t(a)}^2
\notag\\
&\quad+
3\frac{\vartheta_t}{L}\EE\norm{\bar e_t}^2
+
\frac{\Lambda_t}{2}\EE\norm{\bar\zeta_t(a)}^2.
\end{align*}
By Lemma~\ref{lem:realized-het-server-error}, $\EE\norm{\bar e_t}^2\le Q_t$.
Since $\Lambda_t\le 2L$ and $\vartheta_t\le 1$,
\[
\frac{9L}{4\vartheta_t}+\frac{\Lambda_t}{2}
\le
\frac{9L}{4\vartheta_t}+L
\le
\frac{13L}{4\vartheta_t}
\le
\frac{4L}{\vartheta_t}.
\]
Therefore
\begin{align}
\EE[\Psi_t^{\mathrm{het}}(a)]
&\le
-\frac{\vartheta_t}{2L}\EE\norm{\nabla F(x_t)}^2
+
\frac{4L}{\vartheta_t}\EE\norm{\nu_t(a)}^2
+
3\frac{\vartheta_t}{L}Q_t
+
\frac{\Lambda_t}{2}\EE\norm{\bar\zeta_t(a)}^2.
\label{eq:realized-het-before-moments}
\end{align}
Applying Lemma~\ref{lem:realized-het-moments} gives
\begin{align*}
\frac{4L}{\vartheta_t}\EE\norm{\nu_t(a)}^2
&\le
192\vartheta_t^3 U_t
+
36\frac{\vartheta_t}{L}Q_t
+
64\frac{\vartheta_t^3}{L}V_1(a),
\\
\frac{\Lambda_t}{2}\EE\norm{\bar\zeta_t(a)}^2
&\le
32\frac{\Lambda_t}{L}\vartheta_t^4 U_t
+
64\frac{\Lambda_t}{L^2}\vartheta_t^4 Q_t
+
16\frac{\Lambda_t\vartheta_t^2}{L^2}V_2(a).
\end{align*}
Substituting these estimates into \eqref{eq:realized-het-before-moments},
\[
\EE[\Psi_t^{\mathrm{het}}(a)]
\le
-\frac{\vartheta_t}{2L}\EE\norm{\nabla F(x_t)}^2
+
\beta_t^{\mathrm{het}}U_t
+
\left(
39\frac{\vartheta_t}{L}
+
64\frac{\Lambda_t}{L^2}\vartheta_t^4
\right)Q_t
+
\delta_t^{\mathrm{het}}(a).
\]
Combining this with \eqref{eq:realized-het-global-start},
\[
g_{t+1}
\le
g_t
-
\frac{\vartheta_t}{2L}\EE\norm{\nabla F(x_t)}^2
+
\beta_t^{\mathrm{het}}U_t
+
\gamma_t^{\mathrm{het}}Q_t
+
\delta_t^{\mathrm{het}}(a).
\]
By Lemma~\ref{lem:gap}, $\EE\norm{\nabla F(x_t)}^2\ge \frac{g_t^2}{R^2}$.
Therefore
\[
g_{t+1}
\le
g_t-A_t^{\mathrm{het}}g_t^2+\beta_t^{\mathrm{het}}U_t+\gamma_t^{\mathrm{het}}Q_t+\delta_t^{\mathrm{het}}(a).
\]
By Lemma~\ref{lem:Ta-majorizes},
\[
g_t-A_t^{\mathrm{het}}g_t^2
\le
T_{A_t^{\mathrm{het}}}(g_t)
\le
T_{A_t^{\mathrm{het}}}(U_t),
\]
since $g_t\le U_t$. Hence
\[
g_{t+1}
\le
T_{A_t^{\mathrm{het}}}(U_t)
+
\beta_t^{\mathrm{het}}U_t
+
\gamma_t^{\mathrm{het}}Q_t
+
\delta_t^{\mathrm{het}}(a).
\]
Because also $g_{t+1}\le \bar f$ by Lemma~\ref{lem:gap}, the definition \eqref{eq:realized-het-U}
implies
\[
g_{t+1}\le U_{t+1}\le \bar f.
\]
By Lemma~\ref{lem:realized-het-tracking}, applied to the $\mathcal H_t$-measurable exact minimizer
$w_t^{\mathrm{het}}\in\Delta_n(\mathcal S_t)\subseteq\Delta_n$,
\[
\max_{1\le i\le n}\EE\norm{e_{i,t+1}}^2
\le
A_\chi+B_\chi U_t+C_\chi Q_t
=
Q_{t+1}.
\]
Thus \eqref{eq:realized-het-upper-state} holds for $t+1$, completing the induction.
\end{proof}

\begin{proof}[Proof of Corollary~\ref{cor:realized-het-rate}]
The proof that $Q_t\le \bar Q$ is identical to the proof of Lemma~\ref{lem:sur-chi-cap}, because
\eqref{eq:realized-het-Q} has the same form as the original tracking recursion.

Next, since $U_t\le \bar f$,
\[
T_{A_t^{\mathrm{het}}}(U_t)
=
U_t-\frac{A_t^{\mathrm{het}}U_t^2}{1+A_t^{\mathrm{het}}U_t}
\le
U_t-\underline A_t^{\mathrm{het}}U_t^2.
\]
Hence
\[
U_{t+1}
=
\min\!\left\{
\bar f,\
T_{A_t^{\mathrm{het}}}(U_t)+\beta_t^{\mathrm{het}}U_t+\gamma_t^{\mathrm{het}}Q_t+\delta_t^{\mathrm{het}}(a)
\right\}
\le
U_t-\underline A_t^{\mathrm{het}}U_t^2+\beta_t^{\mathrm{het}}U_t+\gamma_t^{\mathrm{het}}Q_t+\delta_t^{\mathrm{het}}(a).
\]
Using $Q_t\le \bar Q$ gives \eqref{eq:realized-het-quadratic-linear}. Under
\eqref{eq:realized-het-const-param}, Lemma~\ref{lem:scalar} applied to the scalar recursion
\[
U_{t+1}\le U_t-\underline A^{\mathrm{het}}U_t^2+\beta^{\mathrm{het}}U_t+d^{\mathrm{het}}(a)
\]
yields \eqref{eq:realized-het-rate}.

Finally, assume \eqref{eq:realized-het-positive-v}. Let
\[
\alpha:=\left(\sum_{j=1}^n \frac{H_j b_j}{v_j^2}\right)^{-1},
\qquad
a_i^\star=\alpha\frac{H_i b_i}{v_i^2}.
\]
Then
\[
V_2(a^\star)
=
\sum_{i=1}^n \alpha^2 \frac{H_i^2 b_i^2}{v_i^4}\frac{v_i^2}{H_i b_i}
=
\alpha^2 \sum_{i=1}^n \frac{H_i b_i}{v_i^2}
=
\alpha,
\]
and
\[
V_1(a^\star)
=
\sum_{i=1}^n \alpha \frac{H_i b_i}{v_i^2}\frac{v_i^2}{H_i b_i}
=
n\alpha.
\]
Substituting these identities into \eqref{eq:realized-het-const-coeff-2} and
\eqref{eq:realized-het-const-coeff-3} proves \eqref{eq:realized-het-rate-astar}.
\end{proof}
\paragraph{Homogeneous no-control-variate post-local controller proofs.}

\begin{proof}[Proof of Theorem~\ref{th:realized-hom-one-step}]
The proof of \eqref{eq:realized-hom-pathwise} is identical to the proof of
Theorem~\ref{th:realized-het-one-step}, with $c_t$ replaced by $\bar g_t$ and
$\Psi_t^{\mathrm{het}}$ replaced by $\Psi_t^{\mathrm{hom}}$.

If $w_t^{\mathrm{hom}}$ minimizes $\Psi_t^{\mathrm{hom}}$ on $\Delta_n$, then
\[
\Psi_t^{\mathrm{hom}}(w_t^{\mathrm{hom}})\le \Psi_t^{\mathrm{hom}}(u),
\]
hence
\[
\Psi_t^{\mathrm{hom}}(w_t^{\mathrm{hom}})
=
\Psi_t^{\mathrm{hom}}(u)-\Gamma_t^{\mathrm{hom}},
\qquad
\Gamma_t^{\mathrm{hom}}\ge 0.
\]
Substituting this identity into \eqref{eq:realized-hom-pathwise} proves
\eqref{eq:realized-hom-comp}.

To prove \eqref{eq:realized-hom-id}, use \eqref{eq:realized-hom-gbar}:
\[
g_{i,t}^{\mathrm{loc}}
=
-\frac{1}{\eta_{i,t}H_i}\Delta_{i,t}.
\]
Since $\eta_{i,t}H_i=\vartheta_t/L$ for every $i$,
\[
g_{i,t}^{\mathrm{loc}}
=
-\frac{L}{\vartheta_t}\Delta_{i,t}.
\]
Averaging over $i$ gives
\[
\bar g_t
=
-\frac{L}{\vartheta_t}\frac1n\sum_{i=1}^n \Delta_{i,t},
\]
which is equivalent to \eqref{eq:realized-hom-id}.
\end{proof}

\begin{lemma}[Plain local branch radius bound]
\label{lem:realized-hom-radius}
Assume Assumption~\ref{ass:realized-hom-exact} and \ref{ass:realized-common-theta}. Let
$g_t\le U_t$ and define
\[
R_{i,t}^{\mathrm{pl}\,2}:=
\max_{0\le \ell\le H_i}\EE\norm{y_{i,t}^{(\ell)}-x_t}^2.
\]
Then, for every $i\in[n]$,
\begin{equation}
R_{i,t}^{\mathrm{pl}\,2}
\le
16\frac{\vartheta_t^2}{L}U_t
+
4\frac{\vartheta_t^2 v_i^2}{L^2H_i b_i}.
\label{eq:realized-hom-radius}\end{equation}
\end{lemma}

\begin{proof}
For $\ell\in\{0,\dots,H_i\}$,
\[
y_{i,t}^{(\ell)}-x_t
=
-\eta_{i,t}\sum_{s=0}^{\ell-1}
\left(
\nabla F(y_{i,t}^{(s)})+\varepsilon_{i,t,s}
\right),
\qquad
\varepsilon_{i,t,s}:=g_{i,t,s}-\nabla F(y_{i,t}^{(s)}).
\]
Hence
\begin{align}
\EE\norm{y_{i,t}^{(\ell)}-x_t}^2
&\le
2\eta_{i,t}^2
\EE\left\|
\sum_{s=0}^{\ell-1}\nabla F(y_{i,t}^{(s)})
\right\|^2
+
2\eta_{i,t}^2
\EE\left\|
\sum_{s=0}^{\ell-1}\varepsilon_{i,t,s}
\right\|^2.
\label{eq:realized-hom-radius-split}
\end{align}
Write
\[
\nabla F(y_{i,t}^{(s)})=\nabla F(x_t)+\bigl(\nabla F(y_{i,t}^{(s)})-\nabla F(x_t)\bigr).
\]
Then
\[
\left\|
\sum_{s=0}^{\ell-1}\nabla F(y_{i,t}^{(s)})
\right\|^2
\le
2\ell^2\norm{\nabla F(x_t)}^2
+
2
\left\|
\sum_{s=0}^{\ell-1}
\bigl(\nabla F(y_{i,t}^{(s)})-\nabla F(x_t)\bigr)
\right\|^2.
\]
Since $\ell\le H_i$,
\[
2\ell^2\norm{\nabla F(x_t)}^2
\le
2H_i^2\norm{\nabla F(x_t)}^2.
\]
Also,
\[
\left\|
\sum_{s=0}^{\ell-1}
\bigl(\nabla F(y_{i,t}^{(s)})-\nabla F(x_t)\bigr)
\right\|^2
\le
\ell\sum_{s=0}^{\ell-1}
L^2\norm{y_{i,t}^{(s)}-x_t}^2
\le
H_i^2L^2 R_{i,t}^{\mathrm{pl}\,2}.
\]
Therefore
\[
\EE\left\|
\sum_{s=0}^{\ell-1}\nabla F(y_{i,t}^{(s)})
\right\|^2
\le
2H_i^2\EE\norm{\nabla F(x_t)}^2
+
2H_i^2L^2 R_{i,t}^{\mathrm{pl}\,2}.
\]
For the stochastic term, conditional orthogonality across local steps gives
\[
\EE\left\|
\sum_{s=0}^{\ell-1}\varepsilon_{i,t,s}
\right\|^2
\le
H_i\frac{v_i^2}{b_i}.
\]
Substituting these estimates into \eqref{eq:realized-hom-radius-split} and using
$\eta_{i,t}=\vartheta_t/(L H_i)$,
\[
\EE\norm{y_{i,t}^{(\ell)}-x_t}^2
\le
4\frac{\vartheta_t^2}{L^2}\EE\norm{\nabla F(x_t)}^2
+
4\vartheta_t^2 R_{i,t}^{\mathrm{pl}\,2}
+
2\frac{\vartheta_t^2 v_i^2}{L^2H_i b_i}.
\]
Take the maximum over $\ell$. By Lemma~\ref{lem:gap},
\[
\EE\norm{\nabla F(x_t)}^2\le 2L g_t\le 2L U_t.
\]
Applying Lemma~\ref{lem:cum-gronwall} directly to the scalar recursion
\[
r_{i,t,\ell}:=\bigl(\EE\norm{y_{i,t}^{(\ell)}-x_t}^2\bigr)^{1/2}
\le
\frac{\vartheta_t}{L}(\EE\norm{\nabla F(x_t)}^2)^{1/2}
+
\frac{\vartheta_t v_i}{L\sqrt{H_i b_i}}
+
\frac{\vartheta_t}{H_i}\sum_{s=0}^{\ell-1} r_{i,t,s}
\]
yields
\[
R_{i,t}^{\mathrm{pl}\,2}
\le
4 e^{2\vartheta_t}\frac{\vartheta_t^2}{L}U_t
+
2 e^{2\vartheta_t}\frac{\vartheta_t^2 v_i^2}{L^2H_i b_i},
\]
which proves \eqref{eq:realized-hom-radius}.
\end{proof}

\begin{lemma}[Averaged normalized proxy error bound]
\label{lem:realized-hom-q}
Assume Assumption~\ref{ass:realized-hom-exact} and \ref{ass:realized-common-theta}. Then
\begin{equation}
\EE\norm{\bar g_t-\nabla F(x_t)}^2
\le
32L\vartheta_t^2 g_t
+
8\vartheta_t^2\bar V_1
+
2V_u.
\label{eq:realized-hom-q}\end{equation}
\end{lemma}

\begin{proof}
Define
\[
q_t:=\bar g_t-\nabla F(x_t).
\]
By \eqref{eq:realized-hom-gbar},
\[
\bar g_t
=
\frac1n\sum_{i=1}^n \frac1{H_i}\sum_{\ell=0}^{H_i-1} g_{i,t,\ell}.
\]
Insert and subtract $\nabla F(y_{i,t}^{(\ell)})$:
\[
q_t
=
\frac1n\sum_{i=1}^n b_{i,t}
+
\frac1n\sum_{i=1}^n \epsilon_{i,t},
\]
where
\[
b_{i,t}
:=
\frac1{H_i}\sum_{\ell=0}^{H_i-1}
\bigl(\nabla F(y_{i,t}^{(\ell)})-\nabla F(x_t)\bigr),\quad\epsilon_{i,t}
:=
\frac1{H_i}\sum_{\ell=0}^{H_i-1}
\bigl(g_{i,t,\ell}-\nabla F(y_{i,t}^{(\ell)})\bigr).
\]
Hence
\[
\EE\norm{q_t}^2
\le
2\EE\left\|
\frac1n\sum_{i=1}^n b_{i,t}
\right\|^2
+
2\EE\left\|
\frac1n\sum_{i=1}^n \epsilon_{i,t}
\right\|^2.
\]
For the drift term, convexity of $\norm{\cdot}^2$ gives
\[
\left\|
\frac1n\sum_{i=1}^n b_{i,t}
\right\|^2
\le
\frac1n\sum_{i=1}^n \norm{b_{i,t}}^2.
\]
By Jensen and $L$-smoothness of $\nabla F$,
\[
\norm{b_{i,t}}^2
\le
\frac1{H_i}\sum_{\ell=0}^{H_i-1}
\norm{\nabla F(y_{i,t}^{(\ell)})-\nabla F(x_t)}^2
\le
\frac{L^2}{H_i}\sum_{\ell=0}^{H_i-1}\norm{y_{i,t}^{(\ell)}-x_t}^2.
\]
Now apply Lemma~\ref{lem:realized-hom-radius} with $U_t=g_t$:
\[
\EE\left\|
\frac1n\sum_{i=1}^n b_{i,t}
\right\|^2
\le
\frac1n\sum_{i=1}^n L^2 R_{i,t}^{\mathrm{pl}\,2}
\le
16L\vartheta_t^2 g_t
+
4\vartheta_t^2\bar V_1.
\]
For the stochastic term, orthogonality across nodes and local steps gives
\[
\EE\left\|
\frac1n\sum_{i=1}^n \epsilon_{i,t}
\right\|^2
=
\frac1{n^2}\sum_{i=1}^n \EE\norm{\epsilon_{i,t}}^2.
\]
Also,
\[
\EE\norm{\epsilon_{i,t}}^2
=
\EE\left\|
\frac1{H_i}\sum_{\ell=0}^{H_i-1}
\bigl(g_{i,t,\ell}-\nabla F(y_{i,t}^{(\ell)})\bigr)
\right\|^2
\le
\frac1{H_i^2}\sum_{\ell=0}^{H_i-1}
\EE\norm{g_{i,t,\ell}-\nabla F(y_{i,t}^{(\ell)})}^2
\le
\frac{v_i^2}{H_i b_i}.
\]
Therefore
\[
\EE\left\|
\frac1n\sum_{i=1}^n \epsilon_{i,t}
\right\|^2
\le
\frac1{n^2}\sum_{i=1}^n \frac{v_i^2}{H_i b_i}
=
V_u.
\]
Combining the two estimates proves \eqref{eq:realized-hom-q}.
\end{proof}

\begin{proof}[Proof of Theorem~\ref{th:realized-hom-bellman}]
By Theorem~\ref{th:realized-hom-one-step},
\begin{equation}
g_{t+1}
\le
g_t+\EE[\Psi_t^{\mathrm{hom}}(u)]
+
\frac{1}{2(\Lambda_t-L)}\EE\norm{\bar g_t-\nabla F(x_t)}^2.
\label{eq:realized-hom-start}\end{equation}
By \eqref{eq:realized-hom-id},
\[
\Psi_t^{\mathrm{hom}}(u)
=
\ip{\bar g_t}{-\frac{\vartheta_t}{L}\bar g_t}
+
\frac{\Lambda_t}{2}
\left\|-\frac{\vartheta_t}{L}\bar g_t\right\|^2
=
-\frac{\vartheta_t}{L}
\left(
1-\frac{\Lambda_t\vartheta_t}{2L}
\right)\norm{\bar g_t}^2.
\]
Since $\Lambda_t\vartheta_t\le L/2$, it holds that $1-\frac{\Lambda_t\vartheta_t}{2L}\ge \frac34$,
hence
\begin{equation}
\Psi_t^{\mathrm{hom}}(u)\le -\frac{\vartheta_t}{2L}\norm{\bar g_t}^2.
\label{eq:realized-hom-uniform}\end{equation}
Write
\[
\bar g_t=\nabla F(x_t)+q_t,
\qquad
q_t:=\bar g_t-\nabla F(x_t).
\]
Then
\[
\norm{\bar g_t}^2
=
\norm{\nabla F(x_t)+q_t}^2
\ge
\frac12\norm{\nabla F(x_t)}^2-\norm{q_t}^2.
\]
Combining this with \eqref{eq:realized-hom-uniform},
\[
\EE[\Psi_t^{\mathrm{hom}}(u)]
\le
-\frac{\vartheta_t}{4L}\EE\norm{\nabla F(x_t)}^2
+
\frac{\vartheta_t}{2L}\EE\norm{q_t}^2.
\]
Substituting into \eqref{eq:realized-hom-start},
\[
g_{t+1}
\le
g_t
-
\frac{\vartheta_t}{4L}\EE\norm{\nabla F(x_t)}^2
+
\left(
\frac{\vartheta_t}{2L}
+
\frac{1}{2(\Lambda_t-L)}
\right)\EE\norm{q_t}^2.
\]
Applying Lemma~\ref{lem:realized-hom-q},
\begin{align*}
g_{t+1}
&\le
g_t
-
\frac{\vartheta_t}{4L}\EE\norm{\nabla F(x_t)}^2\\
&\quad+
\left(
\frac{\vartheta_t}{2L}
+
\frac{1}{2(\Lambda_t-L)}
\right)
\left(
32L\vartheta_t^2 g_t
+
8\vartheta_t^2\bar V_1
+
2V_u
\right).
\end{align*}
Therefore
\[
g_{t+1}
\le
g_t
-
\frac{\vartheta_t}{4L}\EE\norm{\nabla F(x_t)}^2
+
\beta_t^{\mathrm{hom}}g_t
+
\delta_t^{\mathrm{hom}}.
\]
By Lemma~\ref{lem:gap},
\[
\EE\norm{\nabla F(x_t)}^2\ge \frac{g_t^2}{R^2}.
\]
Hence
\[
g_{t+1}
\le
g_t-A_t^{\mathrm{hom}}g_t^2+\beta_t^{\mathrm{hom}}g_t+\delta_t^{\mathrm{hom}}.
\]
By Lemma~\ref{lem:Ta-majorizes},
\[
g_t-A_t^{\mathrm{hom}}g_t^2\le T_{A_t^{\mathrm{hom}}}(g_t),
\]
which proves \eqref{eq:realized-hom-bellman}.
\end{proof}

\begin{proof}[Proof of Corollary~\ref{cor:realized-hom-rate}]
Since $g_t\le \bar f$ by Lemma~\ref{lem:gap},
\[
T_{A_t^{\mathrm{hom}}}(g_t)
=
g_t-\frac{A_t^{\mathrm{hom}}g_t^2}{1+A_t^{\mathrm{hom}}g_t}
\le
g_t-\underline A_t^{\mathrm{hom}}g_t^2.
\]
Substituting this into \eqref{eq:realized-hom-bellman} gives
\eqref{eq:realized-hom-quadratic-linear}. Under \eqref{eq:realized-het-const-param},
Lemma~\ref{lem:scalar} applied directly to the scalar recursion
\[
g_{t+1}\le g_t-\underline A^{\mathrm{hom}}g_t^2+\beta^{\mathrm{hom}}g_t+\delta^{\mathrm{hom}}
\]
proves \eqref{eq:realized-hom-rate}.
\end{proof}

\begin{proof}[Proof of Corollary~\ref{cor:realized-hom-PL}]
Under Assumption~\ref{ass:realized-hom-PL},
\[
\EE\norm{\nabla F(x_t)}^2\ge 2\mu g_t.
\]
Returning to the inequality
\[
g_{t+1}
\le
g_t
-
\frac{\vartheta_t}{4L}\EE\norm{\nabla F(x_t)}^2
+
\beta_t^{\mathrm{hom}}g_t
+
\delta_t^{\mathrm{hom}}
\]
proved in the proof of Theorem~\ref{th:realized-hom-bellman}, we obtain
\[
g_{t+1}
\le
g_t-\frac{\mu\vartheta_t}{2L}g_t+\beta_t^{\mathrm{hom}}g_t+\delta_t^{\mathrm{hom}}
=
(1-\rho_t^{\mathrm{hom}})g_t+\delta_t^{\mathrm{hom}}.
\]
This proves \eqref{eq:realized-hom-PL-rec}. Under \eqref{eq:realized-het-const-param} and
\eqref{eq:realized-hom-rho}, Lemma~\ref{lem:lin-scalar} applied to
\[
g_{t+1}\le (1-\rho^{\mathrm{hom}})g_t+\delta^{\mathrm{hom}}
\]
yields \eqref{eq:realized-hom-PL-rate}.
\end{proof}

\section{Generalized theory}

\begin{definition}[Bregman divergence]
Let $\mathcal X\subset\mathbb R^d$ be nonempty and convex.
Let $h:\mathcal X\to\mathbb R$ be differentiable and convex.
Fix $x_\star\in\mathcal X$.
Define
\begin{equation}
D_h(x_\star,x)
:=
h(x_\star)-h(x)-\langle \nabla h(x),x_\star-x\rangle,
\qquad x\in\mathcal X.
\label{eq:Bregman}
\end{equation}
\end{definition}

\begin{lemma}
For every $x\in\mathcal X$,
\begin{equation}
D_h(x_\star,x)\ge 0.
\label{eq:Bregman_nonnegative}
\end{equation}
\end{lemma}

\begin{proof}
\[
h(x_\star)\ge h(x)+\langle \nabla h(x),x_\star-x\rangle.
\]
Subtract.
\end{proof}

\begin{assumption}[Raw one-step product envelopes]
\label{ass:raw}
Let $(\Omega,\mathcal F,(\mathcal F_t)_{t\ge 0},\mathbb P)$ be a filtered probability space.
Let $(x_t)_{t\ge 0}$ be an adapted $\mathcal X$-valued process.
Let $(z_t)_{t\ge 0}$ be an adapted $\mathbb R_+^m$-valued process, $m\in\mathbb N$, such that each coordinate of $z_t$ is integrable.
For each $t\ge 0$, let $\mathcal U_t$ be a nonempty compact metric space.

Fix integers $J_0,\dots,J_m\in\mathbb N\cup\{0\}$.
For each $t\ge 0$, $i\in\{0,\dots,m\}$, and $j\in\{1,\dots,J_i\}$, let
\[
A_{t,i}:[0,\infty)\times\mathbb R_+^m\times\mathcal U_t\to[0,\infty),
\]
\[
a_{t,i,j}:[0,\infty)\times\mathbb R_+^m\times\mathcal U_t\to[0,\infty),
\qquad
b_{t,i,j}:[0,\infty)\times\mathbb R_+^m\to[0,\infty)
\]
be Borel.
Fix proper convex functions
\[
\phi_{i,j}:[0,\infty)\to[0,\infty],
\qquad
\phi_{i,j}^\ast(\alpha):=\sup_{\beta\ge 0}\{\alpha\beta-\phi_{i,j}(\beta)\}.
\]

Assume that for every $\mathcal F_t$-measurable control $u_t:\Omega\to\mathcal U_t$,
\begin{align}
\mathbb E[D_h(x_\star,x_{t+1})\mid\mathcal F_t]
&\le
A_{t,0}(D_h(x_\star,x_t),z_t,u_t)
+
\sum_{j=1}^{J_0}
a_{t,0,j}(D_h(x_\star,x_t),z_t,u_t)\,
b_{t,0,j}(D_h(x_\star,x_t),z_t),
\label{eq:raw0}
\\
\mathbb E[z_{t+1}^{(i)}\mid\mathcal F_t]
&\le
A_{t,i}(D_h(x_\star,x_t),z_t,u_t)
+
\sum_{j=1}^{J_i}
a_{t,i,j}(D_h(x_\star,x_t),z_t,u_t)\,
b_{t,i,j}(D_h(x_\star,x_t),z_t),
\notag\\
&\hspace{8.3cm}
i=1,\dots,m.
\label{eq:rawi}
\end{align}
\end{assumption}

\begin{theorem}[Fenchel--Young Bellman closure {\cite{Bellman}}]
\label{thm:FY_closure}
Assume Assumption~\ref{ass:raw}.
Define
\[
E:=\mathbb R^{m+1},
\qquad
K:=\mathbb R_+^{m+1},
\qquad
Y_t:=\bigl(D_h(x_\star,x_t),z_t^{(1)},\dots,z_t^{(m)}\bigr)\in K.
\]
For $y=(y_0,y_1,\dots,y_m)\in K$ and $u\in\mathcal U_t$, define
\begin{align*}
U_{t,0}(y,u)
&:=
A_{t,0}(y_0,y_{1:m},u)
+
\sum_{j=1}^{J_0}
\phi_{0,j}^\ast(a_{t,0,j}(y_0,y_{1:m},u))
+
\sum_{j=1}^{J_0}
\phi_{0,j}(b_{t,0,j}(y_0,y_{1:m})),
\\
U_{t,i}(y,u)
&:=
A_{t,i}(y_0,y_{1:m},u)
+
\sum_{j=1}^{J_i}
\phi_{i,j}^\ast(a_{t,i,j}(y_0,y_{1:m},u))
+
\sum_{j=1}^{J_i}
\phi_{i,j}(b_{t,i,j}(y_0,y_{1:m})),
\notag\\
&\hspace{8.1cm}
i=1,\dots,m,
\end{align*}
and
\begin{equation*}
\mathcal B_t(y,u)
:=
\bigl(U_{t,0}(y,u),U_{t,1}(y,u),\dots,U_{t,m}(y,u)\bigr)\in K.
\end{equation*}
Then $Y_t\in K$ a.s. for all $t$, and for every $\mathcal F_t$-measurable control $u_t$,
\begin{equation}
\mathbb E[Y_{t+1}\mid\mathcal F_t]
\preceq_K
\mathcal B_t(Y_t,u_t)
\qquad
\text{a.s.}
\label{eq:FY_upperized}
\end{equation}
\end{theorem}

\begin{proof}
By \eqref{eq:Bregman_nonnegative}, $Y_t\in K$ a.s.
For every $i\in\{0,\dots,m\}$ and $j\in\{1,\dots,J_i\}$,
\begin{align*}
a_{t,i,j}(D_h(x_\star,x_t),z_t,u_t)\,b_{t,i,j}(D_h(x_\star,x_t),z_t)
\le&
\phi_{i,j}^\ast(a_{t,i,j}(D_h(x_\star,x_t),z_t,u_t))\\
&+
\phi_{i,j}(b_{t,i,j}(D_h(x_\star,x_t),z_t)).
\end{align*}
Substitute into \eqref{eq:raw0} and \eqref{eq:rawi}.
Then
\[
\bigl(\mathbb E[Y_{t+1}\mid\mathcal F_t]\bigr)_i
\le
\bigl(\mathcal B_t(Y_t,u_t)\bigr)_i,
\qquad
i=0,\dots,m.
\]
Hence \eqref{eq:FY_upperized}.
\end{proof}

For the remainder of the section, $E=\mathbb R^{m+1}$, $K=\mathbb R_+^{m+1}$, $Y_t=(D_h(x_\star,x_t),z_t)$, and $\mathcal B_t$ are as in Theorem~\ref{thm:FY_closure}.

\begin{definition}[Ordered state space]
Define
\[
x\preceq_K y
\quad\Longleftrightarrow\quad
y-x\in K.
\]
Let
\[
K^\ast:=\{\ell\in E^\ast:\ \ell(x)\ge 0\ \forall x\in K\}.
\]
Fix a nonzero functional $\Psi\in K^\ast$.
\end{definition}

\begin{definition}[Gauge state]
Define
\begin{equation*}
s_t:=\Psi(Y_t).
\end{equation*}
\end{definition}

\begin{assumption}[Regularity]
\label{ass:regularity}
For every $t\ge 0$:
\begin{enumerate}
\item $Y_t\in L^1(\Omega;E)$;
\item $\mathcal U_t$ is a nonempty compact metric space;
\item all maps introduced below are Borel in all variables and continuous in the control variable.
\end{enumerate}
\end{assumption}

\begin{assumption}[Generator]
\label{ass:generator}
Let $\varphi_\theta:[0,\infty)\to[0,\infty)$ satisfy:
\begin{enumerate}
\item $\varphi_\theta\in C([0,\infty))\cap C^2((0,\infty))$;
\item $\varphi_\theta(0)=0$;
\item $\varphi_\theta(s)>0$ for every $s>0$;
\item $\varphi_\theta$ is convex on $[0,\infty)$.
\end{enumerate}
\end{assumption}

\begin{definition}[Scalar flow]
Under Assumption~\ref{ass:generator}, for each $s\ge 0$ let $v(\cdot;s)$ be the unique global solution of
\begin{equation}
\dot v(\tau)=-\varphi_\theta(v(\tau)),
\qquad
v(0)=s.
\label{eq:scalar_ode}
\end{equation}
Define
\begin{equation*}
R_a^\theta(s):=v(a;s),
\qquad
a,s\ge 0.
\end{equation*}
\end{definition}

\begin{proposition}
\label{prop:flow}
Under Assumption~\ref{ass:generator}, the map
\[
(a,s)\mapsto R_a^\theta(s)
\]
is continuous on $[0,\infty)^2$, and the following hold:
\begin{align}
R_0^\theta(s)&=s,
\label{eq:R0}
\\
R_{a+b}^\theta(s)&=R_a^\theta(R_b^\theta(s)),
\qquad a,b,s\ge 0,
\label{eq:Rsemigroup}
\\
R_a^\theta(0)&=0,
\qquad a\ge 0,
\label{eq:Rzero}
\\
0\le R_a^\theta(s)&\le s,
\qquad a,s\ge 0,
\label{eq:Rcontract}
\\
s_1\le s_2 &\Longrightarrow R_a^\theta(s_1)\le R_a^\theta(s_2),
\qquad a\ge 0,
\label{eq:Rmono}
\end{align}
and, for each $a\ge 0$, the map $s\mapsto R_a^\theta(s)$ is concave on $[0,\infty)$.
\end{proposition}

\begin{proof}
Fix $T,S>0$.
For $0\le t\le T$ and $0\le s\le S$, one has
\[
0\le R_t^\theta(s)\le S
\]
by \eqref{eq:Rcontract}.
Let
\[
M_{T,S}:=\max_{r\in[0,S]}\varphi_\theta(r),
\qquad
L_{T,S}:=\sup_{r\in[0,S]}|\varphi_\theta'(r)|.
\]
Then, for $0\le t,r\le T$,
\[
|R_t^\theta(s)-R_r^\theta(s)|
=
\left|
\int_r^t \varphi_\theta(R_\tau^\theta(s))\,d\tau
\right|
\le
M_{T,S}|t-r|.
\]
For $0\le s_1,s_2\le S$, let
\[
u(\tau):=R_\tau^\theta(s_1),
\qquad
v(\tau):=R_\tau^\theta(s_2).
\]
Then
\[
\frac{d}{d\tau}(u(\tau)-v(\tau))
=
-\varphi_\theta(u(\tau))+\varphi_\theta(v(\tau)).
\]
Hence
\[
\frac{d}{d\tau}|u(\tau)-v(\tau)|
\le
L_{T,S}|u(\tau)-v(\tau)|.
\]
Gronwall yields
\[
|R_t^\theta(s_1)-R_t^\theta(s_2)|
\le
e^{L_{T,S}t}|s_1-s_2|.
\]
Thus $(a,s)\mapsto R_a^\theta(s)$ is jointly continuous.

Global existence follows from $0\le v(\tau;s)\le s$.
Equations \eqref{eq:R0}, \eqref{eq:Rsemigroup}, and \eqref{eq:Rzero} follow from the definition, uniqueness, and the constant zero solution.
Equation \eqref{eq:Rcontract} follows from $\dot v(\tau)\le 0$.
For \eqref{eq:Rmono}, let $u(\tau):=R_\tau^\theta(s_1)$ and $v(\tau):=R_\tau^\theta(s_2)$ with $s_1\le s_2$.
If $u(\tau_0)>v(\tau_0)$ for some $\tau_0>0$, let
\[
\tau_*:=\inf\{\tau\ge 0:\ u(\tau)>v(\tau)\}.
\]
Then $u(\tau_*)=v(\tau_*)$ and
\[
\dot u(\tau_*)-\dot v(\tau_*)
=
-\varphi_\theta(u(\tau_*))+\varphi_\theta(v(\tau_*))
=
0.
\]
Uniqueness implies $u(\tau)=v(\tau)$ for $\tau\ge \tau_*$, contradiction.

For concavity, define
\[
u(\tau,s):=R_\tau^\theta(s),
\qquad
p(\tau,s):=\partial_su(\tau,s),
\qquad
q(\tau,s):=\partial_{ss}u(\tau,s),
\]
for $\tau>0$, $s>0$.
Differentiating \eqref{eq:scalar_ode},
\[
\partial_\tau p(\tau,s)
=
-\varphi_\theta'(u(\tau,s))\,p(\tau,s),
\qquad
p(0,s)=1,
\]
hence
\[
p(\tau,s)
=
\exp\!\left(
-\int_0^\tau \varphi_\theta'(u(r,s))\,dr
\right)>0.
\]
Differentiating again,
\[
\partial_\tau q(\tau,s)
=
-\varphi_\theta''(u(\tau,s))\,p(\tau,s)^2
-
\varphi_\theta'(u(\tau,s))\,q(\tau,s),
\qquad
q(0,s)=0.
\]
Therefore
\[
\frac{d}{d\tau}\left(\frac{q(\tau,s)}{p(\tau,s)}\right)
=
-\varphi_\theta''(u(\tau,s))\,p(\tau,s)
\le 0.
\]
Since $q(0,s)=0$, one gets $q(\tau,s)\le 0$.
Thus $R_a^\theta$ is concave on $(0,\infty)$, hence on $[0,\infty)$ by continuity.
\end{proof}

\begin{definition}[Conjugacy coordinate {\cite{Bellman}}]
\label{def:chi}
Fix an interval $I\subset(0,\infty)$ and $s_{\mathrm{ref}}\in I$.
Define
\begin{equation*}
\chi_\theta(s)
:=
\int_s^{s_{\mathrm{ref}}}\frac{d\xi}{\varphi_\theta(\xi)},
\qquad
s\in I.
\end{equation*}
\end{definition}

\begin{proposition}
\label{prop:chi}
If $s\in I$ and $R_a^\theta(s)\in I$, then
\begin{equation}
\chi_\theta(R_a^\theta(s))=\chi_\theta(s)+a.
\label{eq:chi_conj}
\end{equation}
\end{proposition}

\begin{proof}
Let $u(\tau):=R_\tau^\theta(s)$. Then
\[
\dot u(\tau)=-\varphi_\theta(u(\tau)).
\]
Hence
\[
\frac{d}{d\tau}\chi_\theta(u(\tau))
=
\chi_\theta'(u(\tau))\dot u(\tau)
=
\left(-\frac{1}{\varphi_\theta(u(\tau))}\right)\!\left(-\varphi_\theta(u(\tau))\right)
=
1.
\]
Integrating over $[0,a]$ yields \eqref{eq:chi_conj}.
\end{proof}

\begin{assumption}[Cone semigroup]
\label{ass:cone_semigroup}
There exists a family of maps
\[
\Phi_a^\theta:K\to K,
\qquad a\ge 0,
\]
such that
\begin{align*}
\Phi_0^\theta&=\mathrm{id}_K,
\\
\Phi_{a+b}^\theta&=\Phi_a^\theta\circ\Phi_b^\theta,
\\
x\preceq_K y &\Longrightarrow \Phi_a^\theta(x)\preceq_K \Phi_a^\theta(y),
\end{align*}
and
\begin{equation}
\Psi(\Phi_a^\theta(y))
=
R_a^\theta(\Psi(y))
\qquad
\forall a\ge 0,\ \forall y\in K.
\label{eq:cone_shadow}
\end{equation}
\end{assumption}

\begin{assumption}[Vector Bellman comparator]
\label{ass:vector_comp}
For each $t\ge 0$ there exist maps
\[
\alpha_t,\beta_t:\mathcal U_t\to[0,\infty),
\qquad
M_t:\mathcal U_t\to \mathcal L(E),
\qquad
q_t:\mathcal U_t\to K,
\]
such that:
\begin{enumerate}
\item $M_t(u)K\subseteq K$ for every $u\in\mathcal U_t$;
\item
\begin{equation}
\mathcal B_t(y,u)
\preceq_K
\Phi_{\alpha_t(u)}^\theta(y)+M_t(u)y+q_t(u)
\qquad
\forall y\in K,\ \forall u\in\mathcal U_t;
\label{eq:vector_comp}
\end{equation}
\item
\begin{equation}
\Psi(M_t(u)y)\le \beta_t(u)\Psi(y)
\qquad
\forall y\in K,\ \forall u\in\mathcal U_t.
\label{eq:M_bound}
\end{equation}
\end{enumerate}
Define
\begin{equation}
\delta_t(u):=\Psi(q_t(u)).
\label{eq:delta}
\end{equation}
\end{assumption}

\begin{proposition}
\label{prop:vector_scalar}
Under the hypotheses of Theorem~\ref{thm:FY_closure}, Assumptions~\ref{ass:regularity}, \ref{ass:cone_semigroup}, and \ref{ass:vector_comp}, for every $\mathcal F_t$-measurable control $u_t$,
\begin{align}
\mathbb E[Y_{t+1}\mid\mathcal F_t]
&\preceq_K
\Phi_{\alpha_t(u_t)}^\theta(Y_t)+M_t(u_t)Y_t+q_t(u_t),
\label{eq:vector_rec}
\\
\mathbb E[s_{t+1}\mid\mathcal F_t]
&\le
R_{\alpha_t(u_t)}^\theta(s_t)+\beta_t(u_t)s_t+\delta_t(u_t).
\label{eq:scalar_rec}
\end{align}
\end{proposition}

\begin{proof}
Equation \eqref{eq:vector_rec} follows from \eqref{eq:FY_upperized} and \eqref{eq:vector_comp}.
Since $\Psi$ is linear and $Y_{t+1}\in L^1(\Omega;E)$,
\[
\mathbb E[s_{t+1}\mid\mathcal F_t]
=
\Psi(\mathbb E[Y_{t+1}\mid\mathcal F_t]).
\]
Applying $\Psi\in K^\ast$ to \eqref{eq:vector_rec},
\[
\mathbb E[s_{t+1}\mid\mathcal F_t]
\le
\Psi(\Phi_{\alpha_t(u_t)}^\theta(Y_t))
+
\Psi(M_t(u_t)Y_t)
+
\Psi(q_t(u_t)).
\]
Using \eqref{eq:cone_shadow}, \eqref{eq:M_bound}, and \eqref{eq:delta}, we obtain \eqref{eq:scalar_rec}.
\end{proof}

\begin{definition}[Bellman kernel and Bellman operator]
\label{def:Bellman}
For $t\ge 0$, $s\ge 0$, and $u\in\mathcal U_t$, define
\begin{equation*}
G_t(s,u)
:=
R_{\alpha_t(u)}^\theta(s)+\beta_t(u)s+\delta_t(u).
\end{equation*}
Define
\begin{equation*}
\mathfrak B_t(s)
:=
\inf_{u\in\mathcal U_t}G_t(s,u).
\end{equation*}
\end{definition}

\begin{proposition}[Measurable $\varepsilon$-selectors {\cite{Bellman}}]
\label{prop:eps}
Assume the hypotheses of Definition~\ref{def:Bellman}, and assume that $u\mapsto \alpha_t(u),\beta_t(u),\delta_t(u)$ are continuous on compact $\mathcal U_t$.
Fix $\varepsilon_t>0$.
Then there exists a Borel map
\[
\pi_t^{\varepsilon_t}:[0,\infty)\to\mathcal U_t
\]
such that
\begin{equation}
G_t(s,\pi_t^{\varepsilon_t}(s))
\le
\mathfrak B_t(s)+\varepsilon_t
\qquad
\forall s\ge 0.
\label{eq:eps_selector}
\end{equation}
\end{proposition}

\begin{proof}
For fixed $s\ge 0$, the map
\[
u\mapsto G_t(s,u)=R_{\alpha_t(u)}^\theta(s)+\beta_t(u)s+\delta_t(u)
\]
is continuous because $u\mapsto \alpha_t(u),\beta_t(u),\delta_t(u)$ are continuous and $(a,s)\mapsto R_a^\theta(s)$ is continuous by Proposition~\ref{prop:flow}.
Let $D_t=\{u_{t,n}\}_{n\ge 1}$ be a countable dense subset of compact $\mathcal U_t$.
Then
\[
\mathfrak B_t(s)=\inf_{n\ge 1}G_t(s,u_{t,n}).
\]
Hence $s\mapsto \mathfrak B_t(s)$ is Borel.
Define
\[
N_t^{\varepsilon_t}(s)
:=
\min\left\{
n\ge 1:\ G_t(s,u_{t,n})\le \mathfrak B_t(s)+\varepsilon_t
\right\}.
\]
The set on the right is nonempty.
For each $n\ge 1$,
\[
\{s:\ N_t^{\varepsilon_t}(s)=n\}
=
\{s:\ G_t(s,u_{t,n})\le \mathfrak B_t(s)+\varepsilon_t\}
\cap
\bigcap_{k=1}^{n-1}
\{s:\ G_t(s,u_{t,k})>\mathfrak B_t(s)+\varepsilon_t\},
\]
hence it is Borel.
Therefore $N_t^{\varepsilon_t}$ is Borel.
Define
\[
\pi_t^{\varepsilon_t}(s):=u_{t,N_t^{\varepsilon_t}(s)}.
\]
Then \eqref{eq:eps_selector} holds.
\end{proof}

\begin{theorem}[Adaptive Bellman reduction]
\label{thm:adaptive}
Assume the hypotheses of Proposition~\ref{prop:eps}.
Let
\[
u_t^\varepsilon:=\pi_t^{\varepsilon_t}(s_t).
\]
Then $u_t^\varepsilon$ is $\mathcal F_t$-measurable and
\begin{equation}
\mathbb E[s_{t+1}\mid\mathcal F_t]
\le
\mathfrak B_t(s_t)+\varepsilon_t
\qquad
\text{a.s.}
\label{eq:adaptive_reduction}
\end{equation}
\end{theorem}

\begin{proof}
Because $s_t$ is $\mathcal F_t$-measurable and $\pi_t^{\varepsilon_t}$ is Borel, $u_t^\varepsilon$ is $\mathcal F_t$-measurable.
By \eqref{eq:scalar_rec},
\[
\mathbb E[s_{t+1}\mid\mathcal F_t]\le G_t(s_t,u_t^\varepsilon).
\]
By \eqref{eq:eps_selector},
\[
G_t(s_t,u_t^\varepsilon)\le \mathfrak B_t(s_t)+\varepsilon_t.
\]
Hence \eqref{eq:adaptive_reduction}.
\end{proof}

\begin{theorem}[Master surrogate recursion]
\label{thm:master}
Assume the hypotheses of Theorem~\ref{thm:adaptive}.
Fix a deterministic comparator sequence $(\bar u_t)_{t\ge 0}$ with $\bar u_t\in\mathcal U_t$.
Define
\begin{equation*}
a_t:=\alpha_t(\bar u_t),
\qquad
b_t:=\beta_t(\bar u_t),
\qquad
d_t:=\delta_t(\bar u_t),
\qquad
S_t:=\mathbb E[s_t].
\end{equation*}
Then
\begin{equation}
S_{t+1}
\le
R_{a_t}^\theta(S_t)+b_tS_t+d_t+\varepsilon_t
\qquad
\forall t\ge 0.
\label{eq:master_rec}
\end{equation}
\end{theorem}

\begin{proof}
By definition of $\mathfrak B_t$,
\[
\mathfrak B_t(s_t)\le G_t(s_t,\bar u_t)=R_{a_t}^\theta(s_t)+b_t s_t+d_t.
\]
Hence \eqref{eq:adaptive_reduction} gives
\[
\mathbb E[s_{t+1}\mid\mathcal F_t]
\le
R_{a_t}^\theta(s_t)+b_t s_t+d_t+\varepsilon_t.
\]
Taking expectations,
\[
S_{t+1}
\le
\mathbb E[R_{a_t}^\theta(s_t)] + b_tS_t+d_t+\varepsilon_t.
\]
By concavity of $R_{a_t}^\theta$,
\[
\mathbb E[R_{a_t}^\theta(s_t)]
\le
R_{a_t}^\theta(\mathbb E[s_t])
=
R_{a_t}^\theta(S_t).
\]
Thus \eqref{eq:master_rec}.
\end{proof}

\begin{theorem}[Noiseless master theorem {\cite{Bellman}}]
\label{thm:noiseless}
Assume the hypotheses of Theorem~\ref{thm:master} and
\begin{equation}
b_t=0,
\qquad
d_t=0,
\qquad
\varepsilon_t=0
\qquad
\forall t\ge 0.
\label{eq:noiseless_ass}
\end{equation}
Define
\begin{equation*}
A_T:=\sum_{t=0}^{T-1}a_t.
\end{equation*}
Then
\begin{equation}
S_T\le R_{A_T}^\theta(S_0)
\qquad
\forall T\ge 1.
\label{eq:noiseless_bound}
\end{equation}
If $S_0>0$ and $R_{A_T}^\theta(S_0)\in I$, then
\begin{equation}
\chi_\theta(S_T)\ge \chi_\theta(S_0)+A_T.
\label{eq:noiseless_chi}
\end{equation}
If there exists $C_{\mathrm{obj}}>0$ such that
\begin{equation}
F(x_t)-F_\star\le C_{\mathrm{obj}}\,s_t
\qquad
\text{a.s. for all }t,
\label{eq:obj_dom}
\end{equation}
then
\begin{equation}
\mathbb E[F(x_T)-F_\star]
\le
C_{\mathrm{obj}}\,R_{A_T}^\theta(S_0).
\label{eq:noiseless_obj}
\end{equation}
\end{theorem}

\begin{proof}
Under \eqref{eq:noiseless_ass},
\[
S_{t+1}\le R_{a_t}^\theta(S_t).
\]
Induction gives \eqref{eq:noiseless_bound}: if $S_T\le R_{A_T}^\theta(S_0)$, then
\[
S_{T+1}
\le
R_{a_T}^\theta(S_T)
\le
R_{a_T}^\theta(R_{A_T}^\theta(S_0))
=
R_{A_{T+1}}^\theta(S_0)
\]
by \eqref{eq:Rmono} and \eqref{eq:Rsemigroup}.
Since $\chi_\theta$ is decreasing,
\[
\chi_\theta(S_T)\ge \chi_\theta(R_{A_T}^\theta(S_0)).
\]
By \eqref{eq:chi_conj},
\[
\chi_\theta(R_{A_T}^\theta(S_0))=\chi_\theta(S_0)+A_T.
\]
Hence \eqref{eq:noiseless_chi}.
Finally,
\[
\mathbb E[F(x_T)-F_\star]\le C_{\mathrm{obj}}\mathbb E[s_T]=C_{\mathrm{obj}}S_T,
\]
and \eqref{eq:noiseless_obj} follows from \eqref{eq:noiseless_bound}.
\end{proof}

\begin{definition}[Noisy drift]
\label{def:noisy_drift}
For the recursion \eqref{eq:master_rec}, define
\begin{equation*}
g_t(s):=R_{a_t}^\theta(s)+b_t s+d_t+\varepsilon_t,
\qquad s\ge 0.
\end{equation*}
\end{definition}

\begin{definition}[Slope modulus]
\label{def:slope_modulus}
For $a,m\ge 0$, define
\begin{equation}
L_a(m):=
\sup_{r>0}
\frac{R_a^\theta(m+r)-R_a^\theta(m)}{r}.
\label{eq:L}
\end{equation}
\end{definition}

\begin{lemma}
\label{lem:L}
For every $a,m\ge 0$,
\begin{equation}
0\le L_a(m)\le 1.
\label{eq:L_bound}
\end{equation}
If $m>0$, then
\begin{equation}
L_a(m)=\partial_+R_a^\theta(m)=\partial_sR_a^\theta(m).
\label{eq:L_derivative}
\end{equation}
If $m=0$, then
\begin{equation}
L_a(0)=\sup_{r>0}\frac{R_a^\theta(r)}{r}.
\label{eq:L_zero}
\end{equation}
\end{lemma}

\begin{proof}
By \eqref{eq:Rcontract},
\[
0\le R_a^\theta(m+r)-R_a^\theta(m)\le r
\qquad
\forall r>0,
\]
hence \eqref{eq:L_bound}.
If $m>0$, concavity of $R_a^\theta$ implies that
\[
r\mapsto \frac{R_a^\theta(m+r)-R_a^\theta(m)}{r}
\]
is nonincreasing on $(0,\infty)$.
Thus the supremum is the right derivative.
Because the derivative exists on $(0,\infty)$, \eqref{eq:L_derivative} follows.
Equation \eqref{eq:L_zero} is the definition of $L_a(0)$.
\end{proof}

\begin{assumption}[Natural noisy floor]
\label{ass:noisy_floor}
There exists $m\ge 0$ such that
\begin{equation}
g_t(m)\le m
\qquad
\forall t\ge 0.
\label{eq:floor}
\end{equation}
Define
\begin{equation}
\lambda_t:=L_{a_t}(m)+b_t.
\label{eq:lambda}
\end{equation}
Assume
\begin{equation}
0\le \lambda_t<1
\qquad
\forall t\ge 0.
\label{eq:lambda_contract}
\end{equation}
\end{assumption}

\begin{definition}[Nonautonomous evolution family]
\label{def:evolution}
For integers $t\ge s\ge 0$, define
\begin{equation}
\Lambda_{t,s}:=\prod_{\tau=s}^{t-1}\lambda_\tau,
\qquad
\Lambda_{s,s}:=1,
\label{eq:Lambda}
\end{equation}
and
\begin{equation*}
\mathcal E_{t,s}^{(m)}(r):=\Lambda_{t,s}\,r,
\qquad
r\ge 0.
\end{equation*}
\end{definition}

\begin{proposition}
\label{prop:evolution}
For all integers $t\ge u\ge s\ge 0$ and all $r\ge 0$,
\begin{equation*}
\mathcal E_{t,u}^{(m)}(\mathcal E_{u,s}^{(m)}(r))
=
\mathcal E_{t,s}^{(m)}(r).
\end{equation*}
\end{proposition}

\begin{proof}
\[
\mathcal E_{t,u}^{(m)}(\mathcal E_{u,s}^{(m)}(r))
=
\Lambda_{t,u}\Lambda_{u,s}r
=
\Lambda_{t,s}r
=
\mathcal E_{t,s}^{(m)}(r).
\]
\end{proof}

\begin{theorem}[Noisy master theorem]
\label{thm:noisy}
Assume the hypotheses of Theorem~\ref{thm:master} and Assumption~\ref{ass:noisy_floor}.
Define
\begin{equation*}
r_t:=(S_t-m)_+.
\end{equation*}
Then
\begin{equation}
r_{t+1}\le \lambda_t r_t
\qquad
\forall t\ge 0.
\label{eq:noisy_onestep}
\end{equation}
Consequently,
\begin{equation}
r_T
\le
\Lambda_{T,0}r_0
=
\mathcal E_{T,0}^{(m)}(r_0)
\qquad
\forall T\ge 1,
\label{eq:noisy_evolution}
\end{equation}
and
\begin{equation}
S_T
\le
m+\mathcal E_{T,0}^{(m)}\bigl((S_0-m)_+\bigr).
\label{eq:noisy_bound}
\end{equation}
If \eqref{eq:obj_dom} holds, then
\begin{equation}
\mathbb E[F(x_T)-F_\star]
\le
C_{\mathrm{obj}}
\left(
m+\mathcal E_{T,0}^{(m)}\bigl((S_0-m)_+\bigr)
\right).
\label{eq:noisy_obj}
\end{equation}
\end{theorem}

\begin{proof}
Because $R_{a_t}^\theta$ is nondecreasing and $b_t\ge 0$, $g_t$ is nondecreasing.
If $S_t\le m$, then
\[
S_{t+1}\le g_t(S_t)\le g_t(m)\le m,
\]
hence
\[
r_{t+1}=0\le \lambda_t r_t.
\]
If $S_t>m$, then $r_t=S_t-m$ and
\[
S_{t+1}-m
\le
R_{a_t}^\theta(S_t)-R_{a_t}^\theta(m)+b_t(S_t-m)+\bigl(g_t(m)-m\bigr).
\]
By \eqref{eq:floor},
\[
S_{t+1}-m
\le
R_{a_t}^\theta(m+r_t)-R_{a_t}^\theta(m)+b_t r_t.
\]
By \eqref{eq:L},
\[
R_{a_t}^\theta(m+r_t)-R_{a_t}^\theta(m)\le L_{a_t}(m)\,r_t.
\]
Hence
\[
S_{t+1}-m\le (L_{a_t}(m)+b_t)r_t=\lambda_t r_t.
\]
Taking positive parts gives \eqref{eq:noisy_onestep}.
Iterating,
\[
r_T\le \left(\prod_{t=0}^{T-1}\lambda_t\right)r_0=\Lambda_{T,0}r_0=\mathcal E_{T,0}^{(m)}(r_0).
\]
Thus \eqref{eq:noisy_evolution}.
Since
\[
S_T\le m+(S_T-m)_+=m+r_T,
\]
\eqref{eq:noisy_bound} follows.
Finally,
\[
\mathbb E[F(x_T)-F_\star]\le C_{\mathrm{obj}}\,S_T,
\]
and \eqref{eq:noisy_obj} follows from \eqref{eq:noisy_bound}.
\end{proof}

\begin{corollary}[Quadratic semigroup]
\label{cor:quadratic}
Assume the hypotheses of Theorem~\ref{thm:noiseless}.
Let
\[
\varphi_\theta(s)=\kappa s^2,
\qquad
\kappa>0.
\]
Then
\begin{equation}
R_a^\theta(s)=\frac{s}{1+\kappa a s}.
\label{eq:quadratic_R}
\end{equation}
If $S_0>0$, then
\begin{equation*}
S_T\le \frac{1}{S_0^{-1}+\kappa A_T}.
\end{equation*}
If \eqref{eq:obj_dom} holds, then
\begin{equation*}
\mathbb E[F(x_T)-F_\star]
\le
\frac{C_{\mathrm{obj}}}{S_0^{-1}+\kappa A_T}.
\end{equation*}
\end{corollary}

\begin{proof}
Solving
\[
\dot v=-\kappa v^2,\qquad v(0)=s
\]
gives
\[
\frac{1}{R_a^\theta(s)}=\frac{1}{s}+\kappa a.
\]
Hence \eqref{eq:quadratic_R}. Substitute into \eqref{eq:noiseless_bound} and \eqref{eq:noiseless_obj}.
\end{proof}

\begin{corollary}[Quadratic noisy contraction]
\label{cor:quadratic_noisy}
Assume the hypotheses of Theorem~\ref{thm:noisy}.
Let
\[
\varphi_\theta(s)=\kappa s^2,
\qquad
\kappa>0.
\]
Then
\begin{equation}
L_{a_t}(m)
=
\begin{cases}
1, & m=0,\\[4pt]
\dfrac{1}{(1+\kappa a_t m)^2}, & m>0,
\end{cases}
\qquad
\lambda_t=L_{a_t}(m)+b_t.
\label{eq:quadratic_lambda}
\end{equation}
If $m=0$, then
\[
\lambda_t=1+b_t\ge 1,
\]
hence the contracting regime \eqref{eq:lambda_contract} is impossible unless $b_t<0$.
Therefore every nontrivial contracting quadratic noisy regime requires $m>0$.
For $m>0$,
\begin{equation}
S_T
\le
m+
\left(
\prod_{t=0}^{T-1}
\left[
\dfrac{1}{(1+\kappa a_t m)^2}+b_t
\right]
\right)
(S_0-m)_+.
\label{eq:quadratic_noisy_rate}
\end{equation}
\end{corollary}

\begin{proof}
From \eqref{eq:quadratic_R},
\[
R_a^\theta(s)=\frac{s}{1+\kappa a s}.
\]
If $m>0$, then
\[
\partial_sR_a^\theta(s)=\frac{1}{(1+\kappa a s)^2},
\]
hence \eqref{eq:quadratic_lambda} follows from \eqref{eq:L_derivative}.
If $m=0$, then
\[
L_a(0)=\sup_{r>0}\frac{R_a^\theta(r)}{r}
=
\sup_{r>0}\frac{1}{1+\kappa a r}
=
1.
\]
Thus \eqref{eq:quadratic_lambda}. The statement for $m=0$ is immediate. The bound \eqref{eq:quadratic_noisy_rate} follows from \eqref{eq:noisy_bound}.
\end{proof}

\begin{corollary}[Power-law semigroup]
\label{cor:power}
Assume the hypotheses of Theorem~\ref{thm:noiseless}.
Let
\[
\varphi_\theta(s)=\kappa s^{1+p},
\qquad
\kappa>0,\quad p>0.
\]
Then
\begin{equation}
R_a^\theta(s)=\left(s^{-p}+\kappa p a\right)^{-1/p}
\qquad
(s>0).
\label{eq:power_R}
\end{equation}
If $S_0>0$, then
\begin{equation}
S_T
\le
\left(S_0^{-p}+\kappa p A_T\right)^{-1/p}.
\label{eq:power_rate}
\end{equation}
If \eqref{eq:obj_dom} holds, then
\begin{equation*}
\mathbb E[F(x_T)-F_\star]
\le
C_{\mathrm{obj}}
\left(S_0^{-p}+\kappa p A_T\right)^{-1/p}.
\end{equation*}
\end{corollary}

\begin{proof}
Solving
\[
\dot v=-\kappa v^{1+p},\qquad v(0)=s
\]
gives
\[
\frac{d}{d\tau}(v(\tau)^{-p})=\kappa p,
\]
hence \eqref{eq:power_R}. Substitute into \eqref{eq:noiseless_bound} and \eqref{eq:noiseless_obj}.
\end{proof}

\begin{corollary}[Power-law noisy contraction]
\label{cor:power_noisy}
Assume the hypotheses of Theorem~\ref{thm:noisy}.
Let
\[
\varphi_\theta(s)=\kappa s^{1+p},
\qquad
\kappa>0,\quad p>0.
\]
Then
\begin{equation}
L_{a_t}(m)
=
\begin{cases}
1, & m=0,\\[4pt]
\bigl(1+\kappa p a_t m^p\bigr)^{-1-\frac1p}, & m>0,
\end{cases}
\qquad
\lambda_t=L_{a_t}(m)+b_t.
\label{eq:power_lambda}
\end{equation}
If $m=0$, then
\[
\lambda_t=1+b_t\ge 1,
\]
hence the contracting regime \eqref{eq:lambda_contract} is impossible unless $b_t<0$.
Therefore every nontrivial contracting power-law noisy regime requires $m>0$.
For $m>0$,
\begin{equation}
S_T
\le
m+
\left(
\prod_{t=0}^{T-1}\lambda_t
\right)
(S_0-m)_+.
\label{eq:power_noisy_rate}
\end{equation}
\end{corollary}

\begin{proof}
By \eqref{eq:power_R},
\[
R_a^\theta(s)=s(1+\kappa p a s^p)^{-1/p}.
\]
If $m>0$, then
\[
\partial_sR_a^\theta(s)
=
(1+\kappa p a s^p)^{-1-\frac1p},
\]
hence \eqref{eq:power_lambda} follows from \eqref{eq:L_derivative}.
If $m=0$, then
\[
L_a(0)=\sup_{r>0}\frac{R_a^\theta(r)}{r}
=
\sup_{r>0}(1+\kappa p a r^p)^{-1/p}
=
1.
\]
Thus \eqref{eq:power_lambda}. The statement for $m=0$ is immediate. The bound \eqref{eq:power_noisy_rate} follows from \eqref{eq:noisy_bound}.
\end{proof}

\begin{corollary}[Exponential / PL semigroup]
\label{cor:exp}
Assume the hypotheses of Theorem~\ref{thm:noiseless}.
Let
\[
\varphi_\theta(s)=\rho s,
\qquad
\rho>0.
\]
Then
\begin{equation}
R_a^\theta(s)=e^{-\rho a}s.
\label{eq:exp_R}
\end{equation}
Hence
\begin{equation*}
S_T\le e^{-\rho A_T}S_0.
\end{equation*}
If \eqref{eq:obj_dom} holds, then
\begin{equation*}
\mathbb E[F(x_T)-F_\star]
\le
C_{\mathrm{obj}}\,e^{-\rho A_T}S_0.
\end{equation*}
\end{corollary}

\begin{proof}
Solving
\[
\dot v=-\rho v,\qquad v(0)=s
\]
gives \eqref{eq:exp_R}. Substitute into \eqref{eq:noiseless_bound} and \eqref{eq:noiseless_obj}.
\end{proof}

\begin{corollary}[Exponential / PL noisy contraction]
\label{cor:exp_noisy}
Assume the hypotheses of Theorem~\ref{thm:noisy}.
Let
\[
\varphi_\theta(s)=\rho s,
\qquad
\rho>0.
\]
Then
\begin{equation}
L_{a_t}(m)=e^{-\rho a_t}
\qquad
\forall m\ge 0,
\qquad
\lambda_t=e^{-\rho a_t}+b_t.
\label{eq:exp_lambda}
\end{equation}
Consequently,
\begin{equation*}
S_T
\le
m+
\left(
\prod_{t=0}^{T-1}(e^{-\rho a_t}+b_t)
\right)
(S_0-m)_+.
\end{equation*}
\end{corollary}

\begin{proof}
Differentiate \eqref{eq:exp_R} in $s$:
\[
\partial_sR_a^\theta(s)=e^{-\rho a}.
\]
Hence \eqref{eq:exp_lambda}. Then use \eqref{eq:noisy_bound}.
\end{proof}

In particular, if
\[
\varphi_\theta(s)=s^2,
\]
then
\[
R_a^\theta(s)=\frac{s}{1+as}=T_a(s).
\]
Thus the kernel $T_a$ from the main part of the paper is the quadratic specialization of the present master line.

\section{Experiments for deep learning tasks}
\label{app:text_experiments}

\subsection{DistilBERT with head-only training}
\label{app:text_distilbert_head_only}

This experiment uses the same federated text-classification protocol on AG News. The original training split is partitioned into \(30\) clients with Dirichlet label skew with concentration \(\alpha=0.3\), with a minimum of \(64\) training examples per client. We reserve \(10\%\) of the original training split for validation, which yields \(108{,}000\) training examples and \(12{,}000\) validation examples; the standard AG News test split contains \(7{,}600\) examples. Every main run lasts \(30\) communication rounds and samples \(10\) clients per round. The common optimization settings are batch size \(16\), evaluation batch size \(64\), maximum sequence length \(192\), weight decay \(10^{-2}\), gradient clipping norm \(1.0\), and seeds \(\{42,43,44,45\}\).

Client-side computation is heterogeneous. For each selected client in each round, the realized number of local steps \(H\) is drawn from a three-component Poisson mixture with rates \(\{1,4,8\}\), mixture probabilities \((0.34,0.33,0.33)\), and truncation to the range \([1,12]\). The communication axis in all accuracy plots is cumulative communicated scalars. The compared methods are \textsc{Uniform-LocalSGD}, \textsc{FedAvg}, \textsc{FedProx}, \textsc{FedNova}, \textsc{MBSGD}, \textsc{SCAFFOLD}, \textsc{HEW}, and \textsc{HEW-Fixed}.

Hyperparameters are tuned separately for each method within each model line. The tuning procedure has two stages. First, candidate configurations are screened in short \(4\)-round runs on seed \(42\). Second, the top candidates are reevaluated in \(8\)-round runs on seeds \(42\) and \(43\). Selection is based on mean validation accuracy, with near-ties resolved by validation loss, test accuracy, communication cost, and cumulative local work. The search space includes method-specific learning-rate multipliers for all methods, \(\mu\) for \textsc{FedProx}, and the horizon-weight scale parameter for \textsc{HEW} and \textsc{HEW-Fixed}.

This line uses \texttt{distilbert-base-uncased} with a frozen encoder and a trainable classification head. The base learning rate is \(5 \times 10^{-4}\), and the number of trainable parameters is \(593{,}668\). The selected hyperparameters for the completed four-seed run are
\textsc{Uniform-LocalSGD} \((lr\_scale=1.5)\),
\textsc{FedAvg} \((2.0)\),
\textsc{FedProx} \((2.0,\mu=0.05)\),
\textsc{FedNova} \((1.25)\),
\textsc{MBSGD} \((2.0)\),
\textsc{SCAFFOLD} \((2.0)\),
\textsc{HEW} \((lr\_scale=1.5,\lambda=1.75)\), and
\textsc{HEW-Fixed} \((lr\_scale=2.0,\lambda=1.0)\).

\begin{figure*}[ht!]
\centering
\setlength{\tabcolsep}{4pt}
\begin{tabular}{cc}
\makecell[c]{\footnotesize DistilBERT head-only\\ \footnotesize Validation accuracy by communication} &
\makecell[c]{\footnotesize DistilBERT head-only\\ \footnotesize Test accuracy by communication} \\[-1mm]
\includegraphics[width=0.33\textwidth]{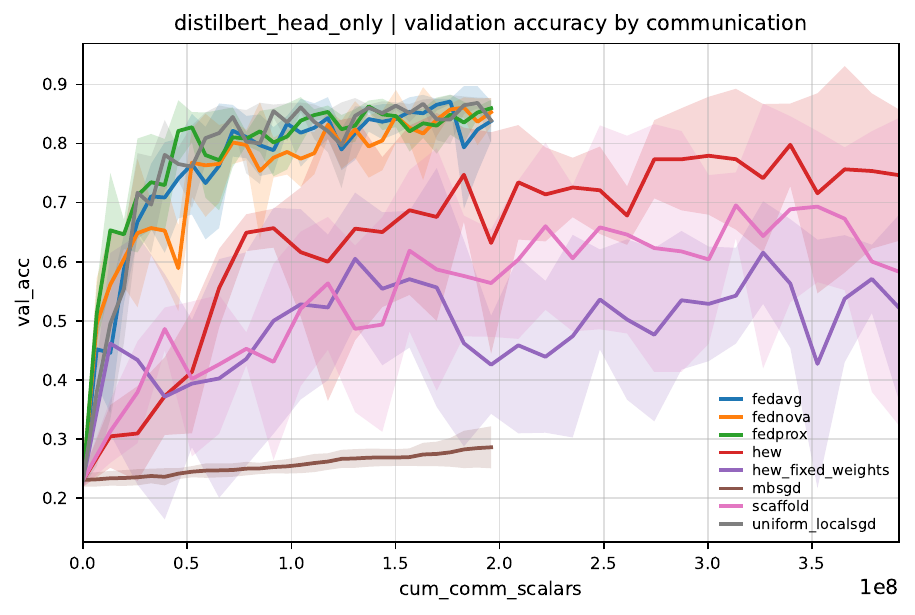} &
\includegraphics[width=0.33\textwidth]{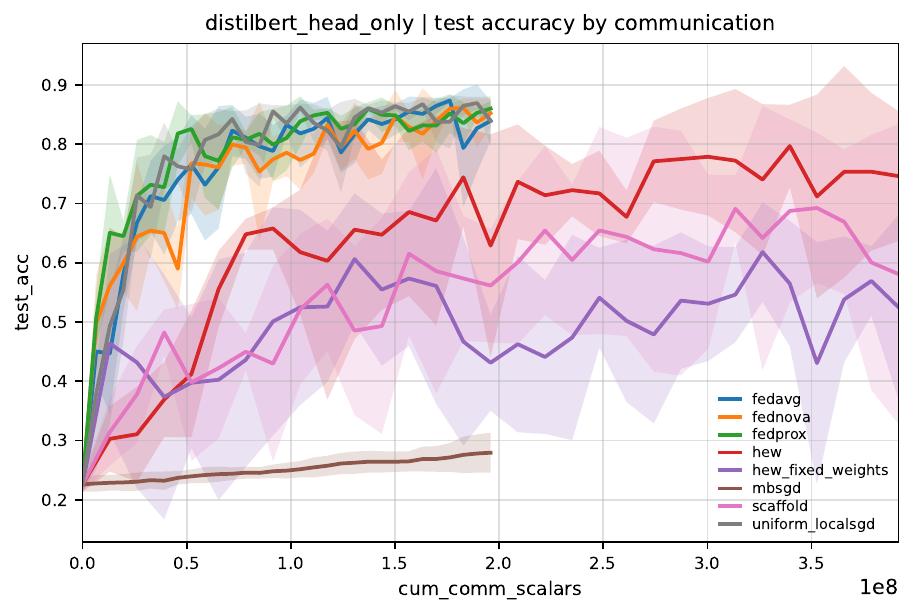}
\end{tabular}
\caption{
Validation and test accuracy for the DistilBERT head-only AG News experiment under heterogeneous local computation. Curves show the mean over seeds \(\{42,43,44,45\}\), and the shaded bands show one standard deviation.
}
\label{fig:distilbert_head_only_acc}
\end{figure*}

Figure~\ref{fig:distilbert_head_only_acc} reports validation and test accuracy as functions of cumulative communication. The two panels induce the same ranking. Over the overlapping communication range, the strongest group is formed by \textsc{FedAvg}, \textsc{FedProx}, \textsc{FedNova}, and \textsc{Uniform-LocalSGD}; these methods track closely and remain in the low-to-mid \(0.8\) range after the initial transient. \textsc{HEW} is below this top group throughout the shared budget. At the same time, \textsc{HEW} is consistently above \textsc{HEW-Fixed} once the trajectories separate, and the gap is large on both validation and test accuracy. \textsc{SCAFFOLD} lies between the two horizon-aware variants for most of the run, while \textsc{MBSGD} is uniformly worst.

The main comparison inside the horizon-aware pair is therefore unambiguous in this experiment: the learned weighting rule in \textsc{HEW} yields substantially better optimization trajectories than the fixed horizon-based rule in \textsc{HEW-Fixed}. This conclusion is supported by both validation and test accuracy and is not an artifact of one panel only.

\begin{figure*}[ht!]
\centering
\setlength{\tabcolsep}{4pt}
\begin{tabular}{cc}
\makecell[c]{\footnotesize DistilBERT head-only\\ \footnotesize \textsc{HEW}: final grouped mass by realized \(H\)} &
\makecell[c]{\footnotesize DistilBERT head-only\\ \footnotesize \textsc{HEW-Fixed}: final grouped mass by realized \(H\)} \\[-1mm]
\includegraphics[width=0.33\textwidth]{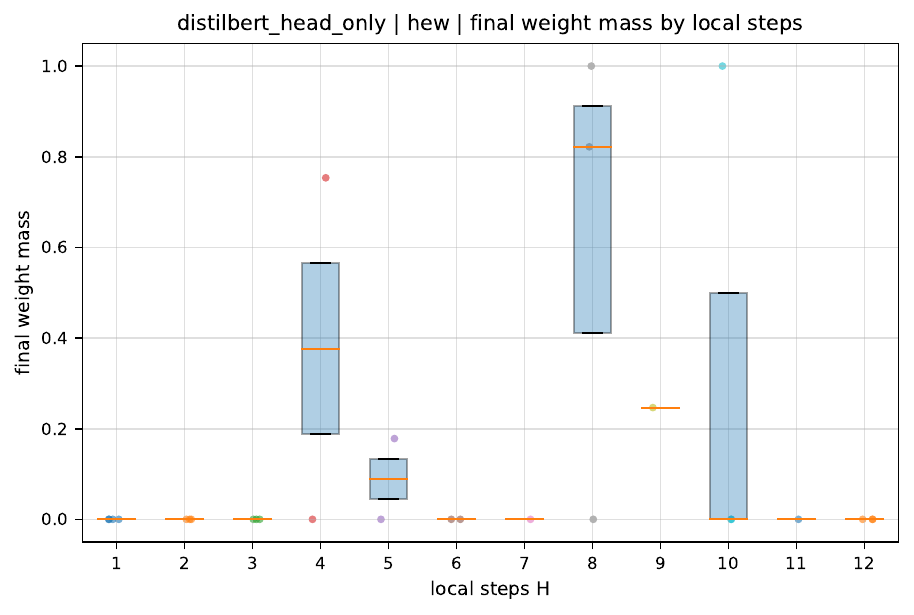} &
\includegraphics[width=0.33\textwidth]{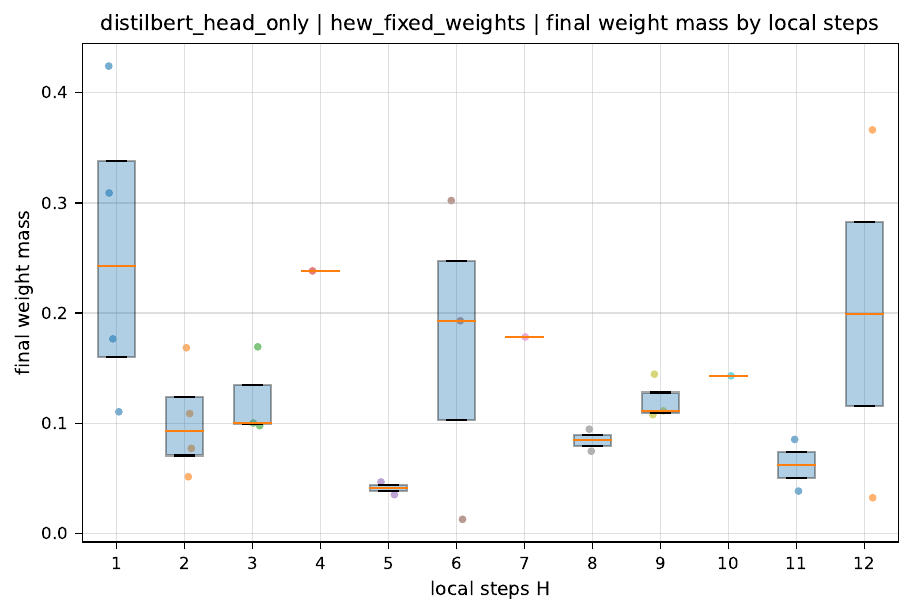}
\end{tabular}
\caption{
Final aggregation mass grouped by realized local-step count \(H\) for the two horizon-aware variants in the DistilBERT head-only experiment.
}
\label{fig:distilbert_head_only_weights}
\end{figure*}

Figure~\ref{fig:distilbert_head_only_weights} reports the final grouped aggregation mass assigned to clients with realized local-step count \(H\). The two variants produce qualitatively different allocation profiles. \textsc{HEW} is sparse and concentrated: for most seeds, the final mass is carried by a small subset of horizon groups, with the largest contributions appearing at intermediate-to-large realized horizons, especially around \(H=4\), \(H=8\), and \(H=10\). By contrast, \textsc{HEW-Fixed} remains diffuse. Its final mass is spread across a broad range of realized horizons, including both low-\(H\) and high-\(H\) groups, with no comparable concentration on a narrow subset.

This figure should be read as grouped final mass, not as a direct plot of a raw per-client weighting formula. For a fixed value of \(H\), the plotted mass depends on the realized composition of the selected client set in the final round: if more selected clients fall into the same horizon group, the total mass assigned to that group changes even when the underlying rule is fixed. The relevant empirical point is the difference between the two grouped profiles. In this line, \textsc{HEW} concentrates the final aggregate on a much narrower subset of realized horizon groups than \textsc{HEW-Fixed}, and this sharper allocation is accompanied by a clear accuracy advantage in Figure~\ref{fig:distilbert_head_only_acc}.

\subsection{Medium-CNN image classification with horizon-aware aggregation}

This experiment uses Fashion-MNIST \cite{fashionmnist} with a medium-sized convolutional network under heterogeneous client-side computation. The original training split is partitioned into \(40\) clients with Dirichlet label skew with concentration \(\alpha=0.5\), with a minimum of \(100\) training examples per client. We reserve \(10\%\) of the original training split for validation, which yields \(54{,}000\) training examples and \(6{,}000\) validation examples; the standard Fashion-MNIST test split contains \(10{,}000\) examples. Every main run lasts \(160\) communication rounds and samples \(10\) clients per round. The common optimization settings are batch size \(64\), evaluation batch size \(512\), base learning rate \(10^{-3}\), weight decay \(10^{-4}\), dropout \(0.15\), gradient clipping norm \(5.0\), and seeds \(\{42,43,44,45\}\).

Client-side computation is heterogeneous at the level of realized local epochs. For each selected client in each round, the realized local epoch count \(H\) varies across clients, and the grouped-mass plots below summarize the final-round aggregation mass by realized \(H\). The communication axis in all communication-based plots is cumulative communicated scalars. The compared methods in this line are \textsc{FedAvg} and \textsc{HEW}.

Hyperparameters are tuned separately for the two methods within this model line. For \textsc{FedAvg}, tuning covers the learning-rate multiplier. For \textsc{HEW}, tuning covers the learning-rate multiplier and the horizon-weight scale parameter. Thus the comparison isolates the effect of replacing uniform aggregation by horizon-aware aggregation under the same data partition, model class, and optimization protocol.

\begin{figure*}[ht!]
\centering
\setlength{\tabcolsep}{4pt}
\begin{tabular}{cc}
\makecell[c]{\footnotesize Medium CNN\\ \footnotesize Validation accuracy by communication} &
\makecell[c]{\footnotesize Medium CNN\\ \footnotesize Test accuracy by communication} \\[-1mm]
\includegraphics[width=0.33\textwidth]{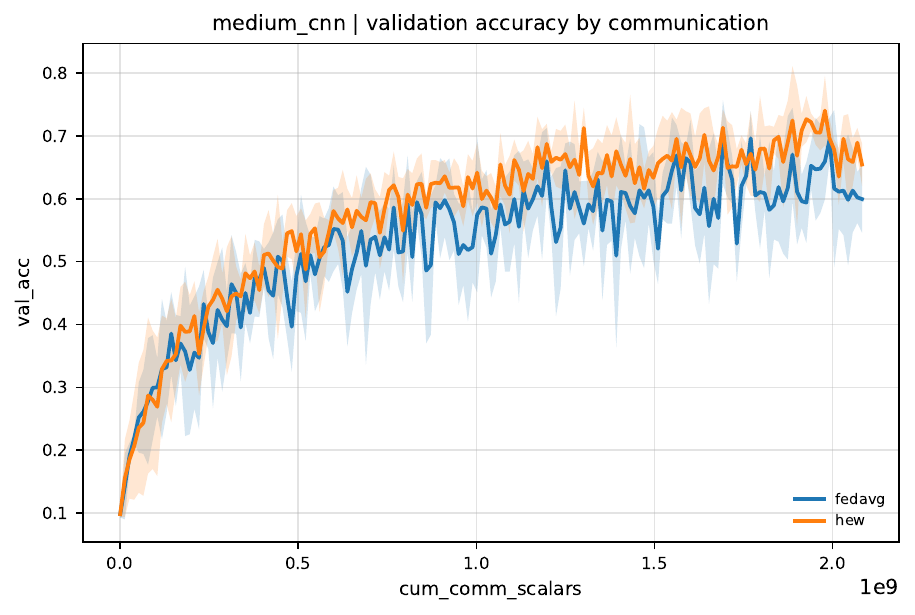} &
\includegraphics[width=0.33\textwidth]{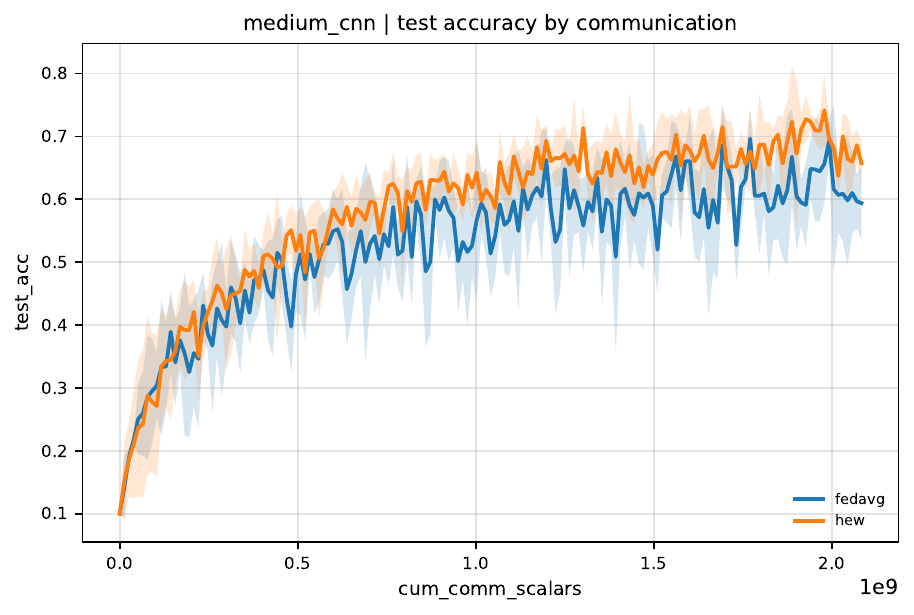} \\[1.5mm]

\makecell[c]{\footnotesize Medium CNN\\ \footnotesize Final validation accuracy across seeds} &
\makecell[c]{\footnotesize Medium CNN\\ \footnotesize Final test accuracy across seeds} \\[-1mm]
\includegraphics[width=0.33\textwidth]{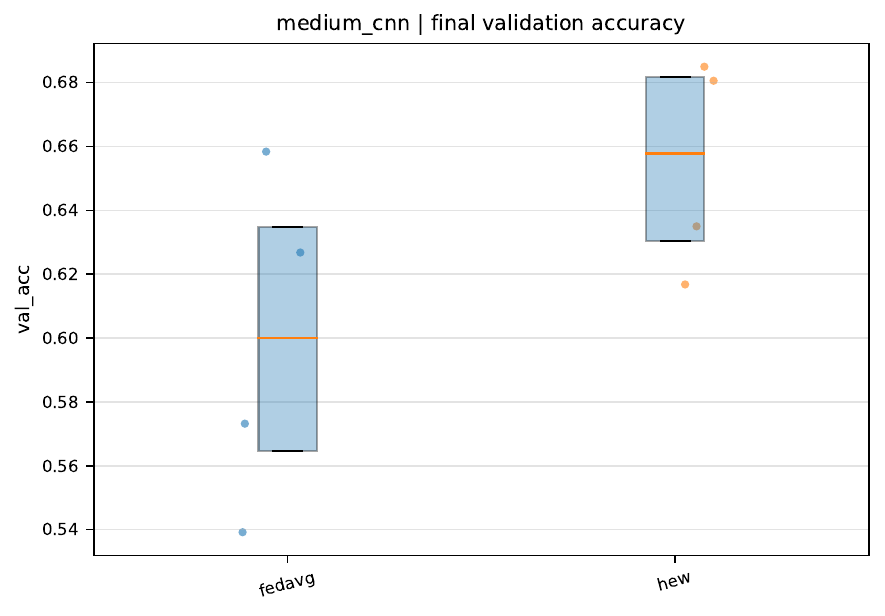} &
\includegraphics[width=0.33\textwidth]{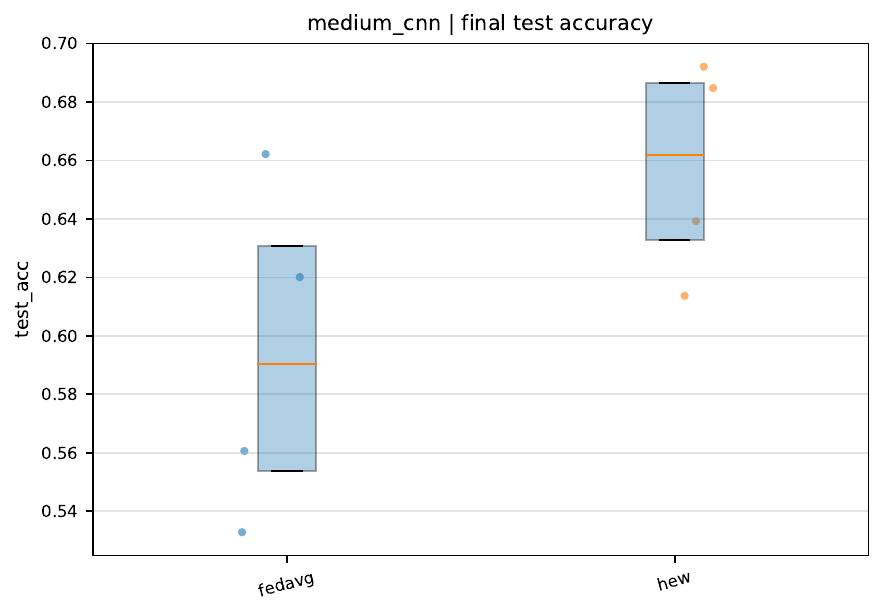}
\end{tabular}
\caption{
Accuracy results for the medium-CNN Fashion-MNIST experiment. The top row shows mean validation and test accuracy as functions of cumulative communication; shaded bands denote one standard deviation across seeds \(\{42,43,44,45\}\). The bottom row reports the final validation and test accuracies across the same four seeds.
}
\label{fig:medium_cnn_hew_acc}
\end{figure*}

Figure~\ref{fig:medium_cnn_hew_acc} reports the main accuracy comparison. The two communication trajectories induce the same qualitative ranking. After the initial transient, \textsc{HEW} remains above \textsc{FedAvg} on both validation and test accuracy for most of the run, and the gap is not confined to a short portion of the communication budget. The final across-seed boxplots support the same conclusion. On both validation and test accuracy, the \textsc{HEW} distribution is shifted upward relative to \textsc{FedAvg}, with only limited overlap between the two four-seed summaries. In this line, the horizon-aware rule therefore yields a clear and stable improvement over uniform aggregation.

Because the communicated model size and the number of sampled clients per round are fixed in this experiment, cumulative communicated scalars are proportional to round number. The round-index plots therefore do not add qualitatively new information beyond the communication plots and are omitted from the main appendix discussion.

\begin{figure*}[ht!]
\centering
\setlength{\tabcolsep}{4pt}
\begin{tabular}{cc}
\makecell[c]{\footnotesize Medium CNN\\ \footnotesize Validation loss by communication} &
\makecell[c]{\footnotesize Medium CNN\\ \footnotesize Test loss by communication} \\[-1mm]
\includegraphics[width=0.33\textwidth]{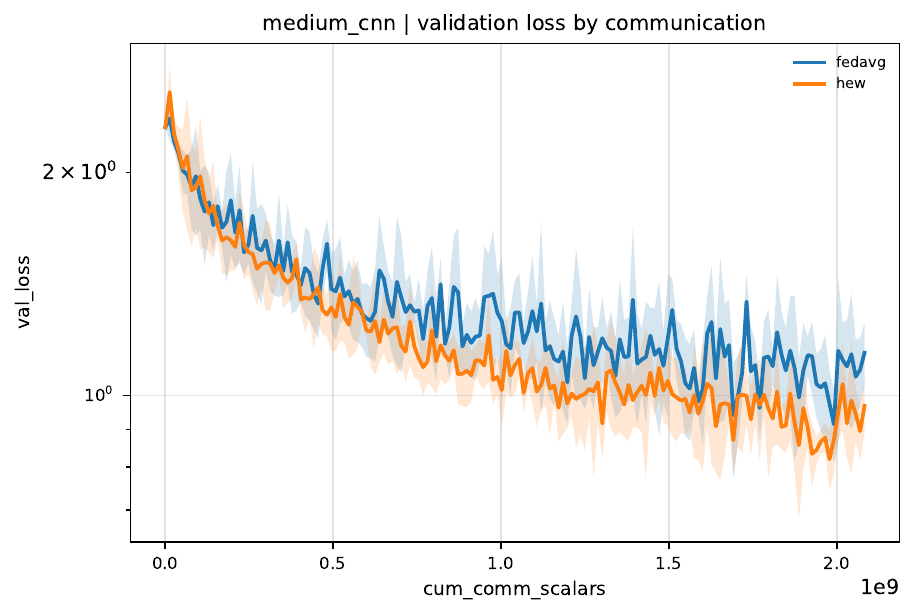} &
\includegraphics[width=0.33\textwidth]{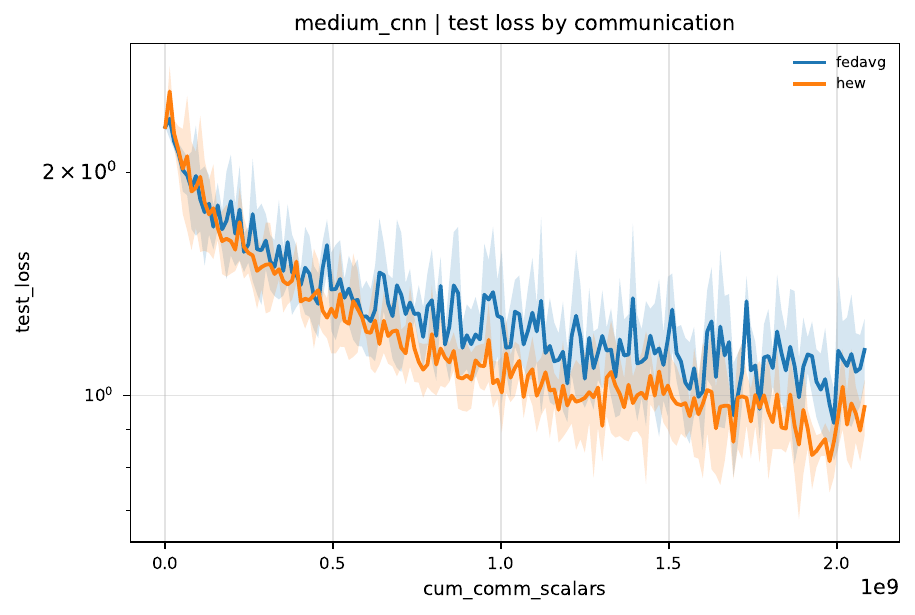} \\[1.5mm]

\makecell[c]{\footnotesize Medium CNN\\ \footnotesize \textsc{FedAvg}: final grouped mass by realized \(H\)} &
\makecell[c]{\footnotesize Medium CNN\\ \footnotesize \textsc{HEW}: final grouped mass by realized \(H\)} \\[-1mm]
\includegraphics[width=0.33\textwidth]{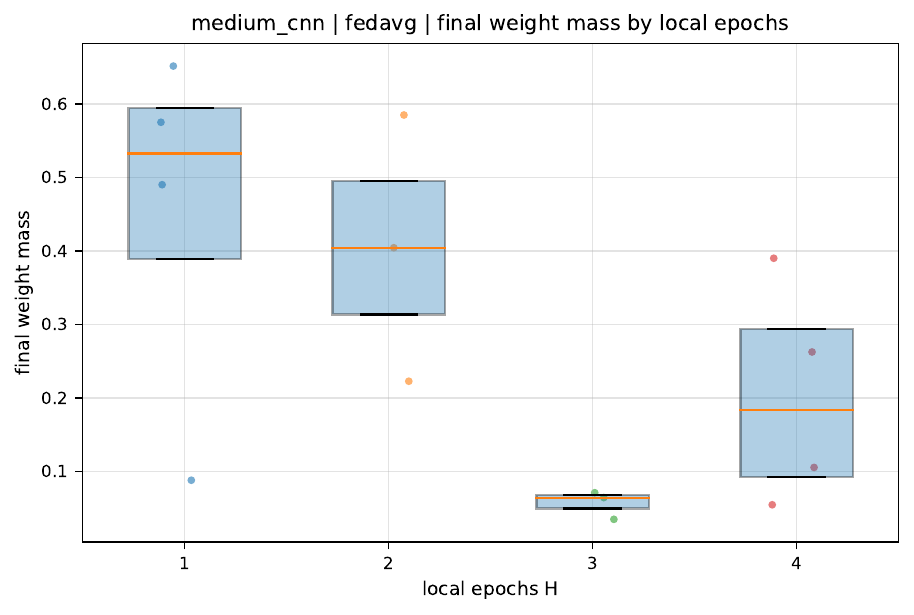} &
\includegraphics[width=0.33\textwidth]{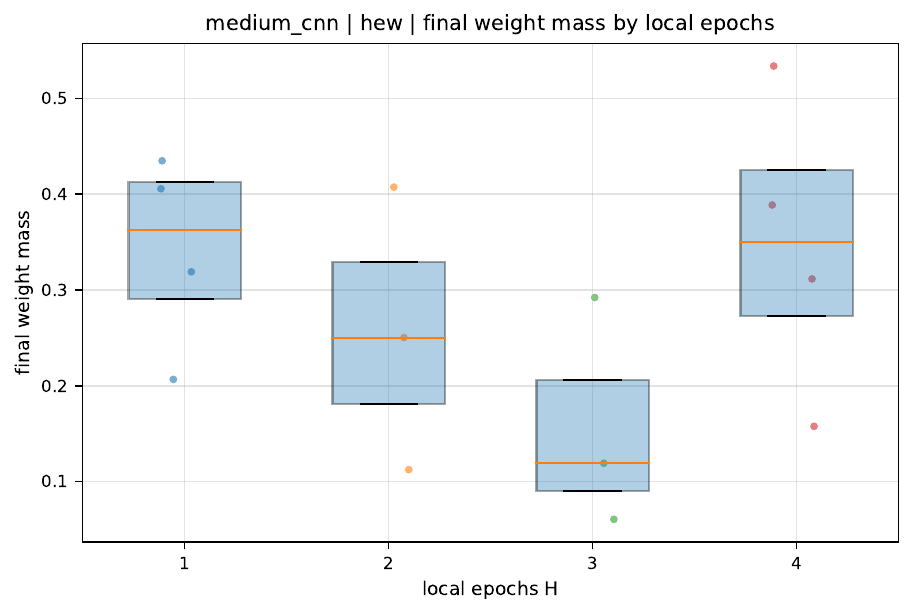}
\end{tabular}
\caption{
Loss trajectories and final grouped aggregation-mass summaries for the medium-CNN Fashion-MNIST experiment. The top row shows mean validation and test loss by cumulative communication, with one-standard-deviation bands across seeds \(\{42,43,44,45\}\). The bottom row shows the final-round aggregation mass grouped by realized local epoch count \(H\), summarized across the same seeds.
}
\label{fig:medium_cnn_hew_loss_weights}
\end{figure*}

Figure~\ref{fig:medium_cnn_hew_loss_weights} provides the corresponding loss and weighting view. The loss trajectories mirror the accuracy results: \textsc{HEW} stays below \textsc{FedAvg} on both validation and test loss over most of the run, so the accuracy advantage is accompanied by consistently better optimization rather than by a thresholding artifact in the accuracy metric.

The grouped-mass panels clarify the mechanism. Under \textsc{FedAvg}, the final grouped mass is concentrated primarily on \(H=1\) and \(H=2\), with much smaller mass on \(H=3\) and a secondary contribution at \(H=4\). Since \textsc{FedAvg} assigns uniform per-client aggregation weights within each round, this panel should be read mainly as a summary of the realized composition of the selected clients in the final round, aggregated by horizon group. The \textsc{HEW} panel is qualitatively different. It still assigns substantial mass to \(H=1\), but it shifts markedly more mass toward the largest realized local-epoch group \(H=4\), while reducing the relative contribution of the intermediate group \(H=2\); \(H=3\) remains the least emphasized group. Thus the horizon-aware rule does not merely reproduce the realized client composition. It selectively reallocates aggregation mass across horizon groups, and in this experiment that reallocation is associated with uniformly better validation and test performance.

The main empirical conclusion of this line is therefore straightforward. For medium-CNN image classification on Fashion-MNIST under heterogeneous local computation, \textsc{HEW} is clearly stronger than \textsc{FedAvg} both in final performance and along the full optimization trajectory. The grouped-mass analysis indicates that this gain is accompanied by a non-uniform reweighting pattern that assigns substantially more influence to high-\(H\) clients than the uniform baseline does, while still avoiding a trivial monotone rule that would simply place all mass on the largest realized horizon.

\section{Symmetric homogeneous exact-gradient degeneration}

\begin{assumption}[Symmetric exact-gradient regime]
\label{ass:symmetric_exact_regime}
Throughout this subsection, for every round $t \ge 0$, we assume:
\begin{enumerate}[label=(\roman*)]
    \item full participation:
    \[
    S_t = [n];
    \]
    \item identical local objectives on the invariant ball:
    \[
    F_i \equiv F \qquad \text{on } B(x^\star,R), \qquad i \in [n];
    \]
    \item equal local horizons:
    \[
    H_i \equiv H \in \mathbb{N}, \qquad i \in [n];
    \]
    \item for the corrected branches (Algorithm~1 and Algorithm~2), identical initial control variates:
    \[
    c_{i,0} = c_0, \qquad i \in [n];
    \]
    \item exact gradients along the realized local paths:
    \[
    g_{i,t,\ell} = \nabla F\!\left(y_{i,t}^{(\ell)}\right),
    \qquad i \in [n], \quad \ell = 0,\dots,H-1 .
    \]
\end{enumerate}
\end{assumption}

\begin{proposition}[Corrected branches: exact centralized microstep representation]
\label{prop:corrected_exact_microsteps}
Assume Assumption~\ref{ass:symmetric_exact_regime}. Consider either:
\begin{enumerate}[label=(\alph*)]
    \item Algorithm~2, or
    \item Algorithm~1 restricted to the common-amplitude slice
    \[
    \theta_t = \vartheta_t \one \in [\theta,\bar\theta]^n
    \qquad \text{for some } \vartheta_t \in [\theta,\bar\theta].
    \]
\end{enumerate}
Define
\[
\eta_t := \frac{\vartheta_t}{LH},
\qquad
z_{t,0} := x_t,
\qquad
z_{t,\ell+1} := z_{t,\ell} - \eta_t \nabla F(z_{t,\ell}),
\qquad \ell = 0,\dots,H-1.
\]
Then, for every round $t \ge 0$, every $\ell = 0,\dots,H$, and every $i \in [n]$,
\[
y_{i,t}^{(\ell)} = z_{t,\ell}.
\]
Moreover,
\[
\Delta_{i,t} = z_{t,H} - x_t,
\qquad i \in [n],
\]
and the corrected control variates satisfy
\[
c_{i,t+1} = c_{t+1} = \frac1H \sum_{\ell=0}^{H-1} \nabla F(z_{t,\ell}),
\qquad i \in [n].
\]
Consequently:
\begin{enumerate}[label=(\roman*)]
    \item for Algorithm~1 on the slice $\theta_t=\vartheta_t\one$,
    \[
    x_{t+1} = x_t + \sum_{i=1}^n w_{i,t}\Delta_{i,t} = z_{t,H}
    \qquad \text{for every } w_t \in \Delta_n;
    \]
    \item for Algorithm~2,
    \[
    d_t(w) = z_{t,H} - x_t,
    \qquad
    x_{t+1} = z_{t,H}
    \qquad \text{for every } w \in \Delta_n,
    \]
    and
    \[
    \Psi_t^{\mathrm{het}}(w)
    =
    \ip{c_t}{z_{t,H}-x_t}
    +
    \frac{\Lambda_t}{2}\norm{z_{t,H}-x_t}^2
    \qquad \text{for every } w \in \Delta_n.
    \]
    In particular, every $w \in \Delta_n$ is an exact minimizer of $\Psi_t^{\mathrm{het}}$.
\end{enumerate}
\end{proposition}

\begin{proof}
We first prove by induction on $t$ that
\[
c_{i,t} = c_t
\qquad \text{for all } i \in [n].
\]

For $t=0$, Assumption~\ref{ass:symmetric_exact_regime}(iv) gives
\[
c_{i,0} = c_0
\qquad \text{for all } i \in [n].
\]

Assume now that, for some fixed round $t \ge 0$,
\[
c_{i,t} = c_t
\qquad \text{for all } i \in [n].
\]
For both branches under consideration, the local step size is the common scalar
\[
\eta_t = \frac{\vartheta_t}{LH},
\]
and the corrected local recursion is
\[
y_{i,t}^{(\ell+1)}
=
y_{i,t}^{(\ell)} - \eta_t\bigl(g_{i,t,\ell} - c_{i,t} + c_t\bigr),
\qquad
\ell = 0,\dots,H-1,
\]
with
\[
y_{i,t}^{(0)} = x_t.
\]

We prove by induction on $\ell$ that
\[
y_{i,t}^{(\ell)} = z_{t,\ell}
\qquad \text{for all } i \in [n].
\]
For $\ell=0$, this is immediate from the definition $z_{t,0}=x_t$.
Assume that the claim holds for some $\ell \in \{0,\dots,H-1\}$. Then, by Assumption~\ref{ass:symmetric_exact_regime}(ii) and (v),
\[
g_{i,t,\ell}
=
\nabla F\!\left(y_{i,t}^{(\ell)}\right)
=
\nabla F(z_{t,\ell}),
\qquad i \in [n].
\]
Using also the induction hypothesis $c_{i,t}=c_t$, we obtain
\[
y_{i,t}^{(\ell+1)}
=
z_{t,\ell} - \eta_t \nabla F(z_{t,\ell})
=
z_{t,\ell+1},
\qquad i \in [n].
\]
This closes the induction on $\ell$ and proves
\[
y_{i,t}^{(\ell)} = z_{t,\ell},
\qquad i \in [n], \quad \ell=0,\dots,H.
\]

Therefore
\[
\Delta_{i,t}
=
y_{i,t}^{(H)} - x_t
=
z_{t,H} - x_t,
\qquad i \in [n].
\]

Next, Lemma~B.23 gives, for every active node $i \in [n]$,
\[
c_{i,t+1}
=
\frac1H \sum_{\ell=0}^{H-1} g_{i,t,\ell}.
\]
Using the already proved identity $y_{i,t}^{(\ell)} = z_{t,\ell}$ and Assumption~\ref{ass:symmetric_exact_regime}(v), we obtain
\[
g_{i,t,\ell}
=
\nabla F(z_{t,\ell}),
\qquad i \in [n], \quad \ell=0,\dots,H-1,
\]
and hence
\[
c_{i,t+1}
=
\frac1H \sum_{\ell=0}^{H-1} \nabla F(z_{t,\ell}),
\qquad i \in [n].
\]
Thus all vectors $c_{i,t+1}$ are equal. Lemma~B.22 gives
\[
c_{t+1} = \frac1n \sum_{i=1}^n c_{i,t+1},
\]
and therefore
\[
c_{i,t+1} = c_{t+1}
\qquad \text{for all } i \in [n].
\]
This closes the induction on $t$.

We now prove the update identities. Since $\Delta_{i,t}=z_{t,H}-x_t$ for all $i$,
\[
x_{t+1}
=
x_t + \sum_{i=1}^n w_{i,t}\Delta_{i,t}
=
x_t + \left(\sum_{i=1}^n w_{i,t}\right)(z_{t,H}-x_t)
=
z_{t,H}
\]
for every $w_t \in \Delta_n$. This proves the Algorithm~1 claim.

For Algorithm~2, full participation implies $\Delta_{i,t}^e = \Delta_{i,t}$ for every $i$, so
\[
d_t(w)
=
\sum_{i=1}^n w_i \Delta_{i,t}^e
=
\sum_{i=1}^n w_i (z_{t,H}-x_t)
=
z_{t,H}-x_t
\]
for every $w \in \Delta_n$. Therefore
\[
x_{t+1} = x_t + d_t(w_t^{\mathrm{het}}) = z_{t,H}.
\]
Substituting the identity for $d_t(w)$ into the definition of $\Psi_t^{\mathrm{het}}$ yields
\[
\Psi_t^{\mathrm{het}}(w)
=
\ip{c_t}{z_{t,H}-x_t}
+
\frac{\Lambda_t}{2}\norm{z_{t,H}-x_t}^2,
\qquad w \in \Delta_n,
\]
which is independent of $w$. Hence every feasible $w$ is an exact minimizer.
\end{proof}

\begin{proposition}[Plain local branch: exact centralized microstep representation]
\label{prop:plain_exact_microsteps}
Assume Assumption~\ref{ass:symmetric_exact_regime} and consider Algorithm~3. Define
\[
\eta_t := \frac{\vartheta_t}{LH},
\qquad
z_{t,0} := x_t,
\qquad
z_{t,\ell+1} := z_{t,\ell} - \eta_t \nabla F(z_{t,\ell}),
\qquad \ell = 0,\dots,H-1.
\]
Then, for every round $t \ge 0$, every $\ell = 0,\dots,H$, and every $i \in [n]$,
\[
y_{i,t}^{(\ell)} = z_{t,\ell}.
\]
Moreover,
\[
\Delta_{i,t} = z_{t,H} - x_t,
\qquad i \in [n],
\]
and
\[
g_{i,t}^{\mathrm{loc}}
=
\frac1H \sum_{\ell=0}^{H-1} \nabla F(z_{t,\ell}),
\qquad i \in [n].
\]
Hence
\[
\bar g_t
=
\frac1H \sum_{\ell=0}^{H-1} \nabla F(z_{t,\ell}).
\]
Consequently,
\[
x_{t+1} = z_{t,H},
\]
and
\[
\Psi_t^{\mathrm{hom}}(w)
=
\ip{\bar g_t}{z_{t,H}-x_t}
+
\frac{\Lambda_t}{2}\norm{z_{t,H}-x_t}^2
\qquad \text{for every } w \in \Delta_n.
\]
In particular, every $w \in \Delta_n$ is an exact minimizer of $\Psi_t^{\mathrm{hom}}$.
\end{proposition}

\begin{proof}
Algorithm~3 uses the recursion
\[
y_{i,t}^{(\ell+1)} = y_{i,t}^{(\ell)} - \eta_t g_{i,t,\ell},
\qquad
y_{i,t}^{(0)} = x_t.
\]
We prove by induction on $\ell$ that
\[
y_{i,t}^{(\ell)} = z_{t,\ell}
\qquad \text{for all } i \in [n].
\]
For $\ell=0$, this is immediate from $z_{t,0}=x_t$.
Assume that the claim holds for some $\ell \in \{0,\dots,H-1\}$. Then, by Assumption~\ref{ass:symmetric_exact_regime}(ii) and (v),
\[
g_{i,t,\ell}
=
\nabla F\!\left(y_{i,t}^{(\ell)}\right)
=
\nabla F(z_{t,\ell}),
\qquad i \in [n].
\]
Hence
\[
y_{i,t}^{(\ell+1)}
=
z_{t,\ell} - \eta_t \nabla F(z_{t,\ell})
=
z_{t,\ell+1},
\qquad i \in [n].
\]
This proves
\[
y_{i,t}^{(\ell)} = z_{t,\ell},
\qquad i \in [n], \quad \ell=0,\dots,H.
\]

Therefore
\[
\Delta_{i,t}
=
y_{i,t}^{(H)} - x_t
=
z_{t,H} - x_t,
\qquad i \in [n].
\]
By the definition of $g_{i,t}^{\mathrm{loc}}$,
\[
g_{i,t}^{\mathrm{loc}}
=
-\frac{1}{\eta_t H}\Delta_{i,t}
=
-\frac{1}{\eta_t H}(z_{t,H}-x_t).
\]
Also, by telescoping the recursion defining $z_{t,\ell}$,
\[
z_{t,H} - x_t
=
-\eta_t \sum_{\ell=0}^{H-1} \nabla F(z_{t,\ell}).
\]
Substituting this identity gives
\[
g_{i,t}^{\mathrm{loc}}
=
\frac1H \sum_{\ell=0}^{H-1} \nabla F(z_{t,\ell}),
\qquad i \in [n].
\]
Therefore
\[
\bar g_t
=
\frac1n \sum_{i=1}^n g_{i,t}^{\mathrm{loc}}
=
\frac1H \sum_{\ell=0}^{H-1} \nabla F(z_{t,\ell}).
\]

Since $\Delta_{i,t}=z_{t,H}-x_t$ for all $i$,
\[
x_{t+1}
=
x_t + \sum_{i=1}^n w_{i,t}^{\mathrm{hom}}\Delta_{i,t}
=
x_t + \left(\sum_{i=1}^n w_{i,t}^{\mathrm{hom}}\right)(z_{t,H}-x_t)
=
z_{t,H}.
\]
Likewise, for every $w \in \Delta_n$,
\[
\sum_{i=1}^n w_i \Delta_{i,t} = z_{t,H}-x_t.
\]
Substituting this into the definition of $\Psi_t^{\mathrm{hom}}$ yields
\[
\Psi_t^{\mathrm{hom}}(w)
=
\ip{\bar g_t}{z_{t,H}-x_t}
+
\frac{\Lambda_t}{2}\norm{z_{t,H}-x_t}^2,
\qquad w \in \Delta_n,
\]
which is independent of $w$. Hence every feasible $w$ is an exact minimizer.
\end{proof}

\begin{corollary}[Exact horizon representation]
\label{cor:exact_horizon_representation}
Under Assumption~\ref{ass:symmetric_exact_regime}, the following statements hold.

\begin{enumerate}[label=(\roman*)]
    \item For Algorithm~2 and Algorithm~3, the round-$t$ update with horizon $H$ is exactly the composition of $H$ centralized gradient-descent microsteps with common microstep size
    \[
    \eta_t = \frac{\vartheta_t}{LH}.
    \]
    \item For Algorithm~1 on the common-amplitude slice $\theta_t=\vartheta_t\one$, the same conclusion holds.
\end{enumerate}

Equivalently, in the symmetric exact-gradient regime, the distinction between local computation and server aggregation disappears: one round is exactly a centralized gradient-descent trajectory segment.
\end{corollary}

\begin{proof}
Immediate from Propositions~\ref{prop:corrected_exact_microsteps} and \ref{prop:plain_exact_microsteps}.
\end{proof}

\begin{assumption}[Symmetric certificate slice for the exact local-control objective]
\label{ass:symmetric_certificate_slice}
Fix a round $t \ge 0$. Assume
\[
S_t = [n],
\qquad
H_i \equiv H,
\qquad
b_i \equiv b,
\qquad
v_i \equiv v,
\qquad
\theta_t = \vartheta_t \one
\]
for some $H \in \mathbb{N}$, $b>0$, $v \ge 0$, and $\vartheta_t \in [\theta,\bar\theta]$.
\end{assumption}

\begin{proposition}[Degeneration of the exact local-control objective]
\label{prop:local_control_degeneration}
Assume Assumption~\ref{ass:symmetric_certificate_slice}. Then the coefficients in Theorem~3.1 satisfy
\[
\rho_i(\vartheta_t;U_t,Q_t) = \rho_t,
\qquad
\kappa_i(\vartheta_t;U_t,Q_t) = \kappa_t,
\qquad
\mu_i(\vartheta_t;U_t,Q_t) = \mu_t,
\qquad i \in [n],
\]
for some scalars $\rho_t,\kappa_t,\mu_t$ with $\kappa_t \ge 0$, and
\[
J_t^{\mathrm{id}}(w,\vartheta_t\one)
=
U_t^\sharp - \mu_t + \frac{L\kappa_t}{2}\sum_{i=1}^n w_i^2,
\qquad w \in \Delta_n.
\]
Consequently, the uniform vector
\[
u := \frac1n \one
\]
belongs to
\[
\argmin_{w \in \Delta_n} J_t^{\mathrm{id}}(w,\vartheta_t\one).
\]
More precisely:
\begin{enumerate}[label=(\roman*)]
    \item if $\kappa_t>0$, then the minimizer is unique and equals $u$;
    \item if $\kappa_t=0$, then $J_t^{\mathrm{id}}(w,\vartheta_t\one)$ is constant on $\Delta_n$.
\end{enumerate}
\end{proposition}

\begin{proof}
Under Assumption~\ref{ass:symmetric_certificate_slice}, the defining formulas of Theorem~3.1 depend on the node index $i$ only through the quadruple $(H_i,b_i,v_i,\theta_i)$, which is constant across $i$. Hence
\[
\rho_i(\vartheta_t;U_t,Q_t) = \rho_t,
\qquad
\kappa_i(\vartheta_t;U_t,Q_t) = \kappa_t,
\qquad i \in [n],
\]
for some scalars $\rho_t,\kappa_t$. Likewise,
\[
A_i(\vartheta_t) = \frac{\vartheta_t}{2LR^2}
\]
is independent of $i$, and therefore so is
\[
s_i(U_t^\sharp;\vartheta_t)
=
U_t^\sharp - T_{A_i(\vartheta_t)}(U_t^\sharp).
\]
Hence
\[
\mu_i(\vartheta_t;U_t,Q_t)
=
s_i(U_t^\sharp;\vartheta_t) - \rho_i(\vartheta_t;U_t,Q_t)
=
\mu_t
\qquad \text{for all } i \in [n].
\]

By the explicit formula for $\kappa_i$ in Theorem~3.1, every term in $\kappa_t$ is nonnegative; therefore
\[
\kappa_t \ge 0.
\]

Substituting the equalities above into the formula of Theorem~3.1 yields
\[
J_t^{\mathrm{id}}(w,\vartheta_t\one)
=
U_t^\sharp
-
\sum_{i=1}^n w_i \mu_t
+
\frac{L}{2}\sum_{i=1}^n w_i^2 \kappa_t.
\]
Since $w \in \Delta_n$, one has $\sum_{i=1}^n w_i=1$, and therefore
\[
J_t^{\mathrm{id}}(w,\vartheta_t\one)
=
U_t^\sharp - \mu_t + \frac{L\kappa_t}{2}\sum_{i=1}^n w_i^2.
\]

If $\kappa_t=0$, then the displayed formula is independent of $w$, proving (ii).

Assume now that $\kappa_t>0$. Minimizing $J_t^{\mathrm{id}}(w,\vartheta_t\one)$ over $\Delta_n$ is equivalent to minimizing $\sum_{i=1}^n w_i^2$ over $\Delta_n$. By Cauchy--Schwarz,
\[
1
=
\left(\sum_{i=1}^n w_i\right)^2
\le
n\sum_{i=1}^n w_i^2,
\]
hence
\[
\sum_{i=1}^n w_i^2 \ge \frac1n,
\]
with equality if and only if
\[
w_1 = \cdots = w_n = \frac1n.
\]
Therefore the unique minimizer is $u=\frac1n\one$, proving (i).
\end{proof}

\begin{proposition}[Algorithm~1 at $H=1$: exact gradient-step form for arbitrary amplitudes]
\label{prop:alg1_H1_general}
Assume Assumption~\ref{ass:symmetric_exact_regime}, and in addition let
\[
H=1.
\]
Consider Algorithm~1 with an arbitrary feasible pair
\[
w_t \in \Delta_n,
\qquad
\theta_t = (\theta_{1,t},\dots,\theta_{n,t}) \in [\theta,\bar\theta]^n.
\]
Then
\[
c_{i,t} = c_t
\qquad \text{for all } i \in [n], \ t \ge 0,
\]
and the primal update satisfies
\[
x_{t+1}
=
x_t - \frac{\vartheta_t^{\mathrm{eff}}}{L}\nabla F(x_t),
\qquad
\vartheta_t^{\mathrm{eff}} := \sum_{i=1}^n w_{i,t}\theta_{i,t}.
\]
Thus, at $H=1$, Algorithm~1 is exactly one ordinary gradient-descent step with effective scalar amplitude $\vartheta_t^{\mathrm{eff}}$.
\end{proposition}

\begin{proof}
We prove by induction on $t$ that
\[
c_{i,t}=c_t
\qquad \text{for all } i \in [n].
\]

For $t=0$, this is exactly Assumption~\ref{ass:symmetric_exact_regime}(iv).

Assume now that, for some fixed round $t \ge 0$,
\[
c_{i,t}=c_t
\qquad \text{for all } i \in [n].
\]
Since $H=1$, Algorithm~1 uses
\[
\eta_{i,t} = \frac{\theta_{i,t}}{L},
\qquad
y_{i,t}^{(0)} = x_t,
\qquad
y_{i,t}^{(1)}
=
x_t - \eta_{i,t}(g_{i,t,0} - c_{i,t} + c_t).
\]
By Assumption~\ref{ass:symmetric_exact_regime}(v),
\[
g_{i,t,0} = \nabla F(x_t).
\]
Using the induction hypothesis $c_{i,t}=c_t$, we obtain
\[
y_{i,t}^{(1)}
=
x_t - \frac{\theta_{i,t}}{L}\nabla F(x_t),
\qquad
\Delta_{i,t}
=
-\frac{\theta_{i,t}}{L}\nabla F(x_t).
\]

Next, Lemma~B.23 gives, since $H=1$,
\[
c_{i,t+1} = g_{i,t,0} = \nabla F(x_t),
\qquad i \in [n].
\]
Thus all vectors $c_{i,t+1}$ are equal. Lemma~B.22 implies
\[
c_{t+1} = \frac1n \sum_{i=1}^n c_{i,t+1},
\]
and hence
\[
c_{i,t+1}=c_{t+1}
\qquad \text{for all } i \in [n].
\]
This closes the induction on $t$.

Finally,
\[
x_{t+1}
=
x_t + \sum_{i=1}^n w_{i,t}\Delta_{i,t}
=
x_t - \frac{1}{L}\left(\sum_{i=1}^n w_{i,t}\theta_{i,t}\right)\nabla F(x_t)
=
x_t - \frac{\vartheta_t^{\mathrm{eff}}}{L}\nabla F(x_t).
\]
This proves the claim.
\end{proof}

\begin{corollary}[Exact reduction to ordinary gradient descent]
\label{cor:exact_reduction_gd}
Assume Assumption~\ref{ass:symmetric_exact_regime} and, in addition, let
\[
H=1.
\]
Then:
\begin{enumerate}[label=(\roman*)]
    \item for Algorithm~2,
    \[
    x_{t+1} = x_t - \frac{\vartheta_t}{L}\nabla F(x_t);
    \]
    \item for Algorithm~3,
    \[
    x_{t+1} = x_t - \frac{\vartheta_t}{L}\nabla F(x_t);
    \]
    \item for Algorithm~1 on the common-amplitude slice $\theta_t=\vartheta_t\one$,
    \[
    x_{t+1} = x_t - \frac{\vartheta_t}{L}\nabla F(x_t).
    \]
\end{enumerate}
\end{corollary}

\begin{proof}
For Algorithm~2, Proposition~\ref{prop:corrected_exact_microsteps} with $H=1$ gives
\[
z_{t,1}
=
z_{t,0} - \frac{\vartheta_t}{L}\nabla F(z_{t,0})
=
x_t - \frac{\vartheta_t}{L}\nabla F(x_t),
\]
and
\[
x_{t+1}=z_{t,1}.
\]
This proves (i).

For Algorithm~3, Proposition~\ref{prop:plain_exact_microsteps} with $H=1$ gives the same identity, proving (ii).

For Algorithm~1 on the common-amplitude slice, Proposition~\ref{prop:alg1_H1_general} gives
\[
\vartheta_t^{\mathrm{eff}}
=
\sum_{i=1}^n w_{i,t}\theta_{i,t}
=
\vartheta_t \sum_{i=1}^n w_{i,t}
=
\vartheta_t.
\]
Substituting this into the update formula of Proposition~\ref{prop:alg1_H1_general} yields (iii).
\end{proof}

\section{Predictive solver and complexity proofs}

\begin{algorithm}[ht]
\caption{Alternating local-control solver}
\label{alg:solver}
\begin{algorithmic}[1]
\Require Initial feasible amplitudes $\theta^{(0)}$, tolerance $\varepsilon$
\For{$r=0,1,2,\dots$}
    \State Solve the exact weight subproblem
    \[
    w^{(r+1)}\in \argmin_{w\in\Delta} \mathcal O(w,\theta^{(r)})
    \]
    \State by the KKT threshold law
    \For{$i=1,\dots,S$}
        \State Solve the scalar amplitude problem
        \[
        \theta_i^{(r+1)}
        \in
        \argmin_{\theta_i\in[\underline\theta,\bar\theta]}
        \mathcal O(w^{(r+1)},\theta_1^{(r+1)},\dots,\theta_{i-1}^{(r+1)},\theta_i,\theta_{i+1}^{(r)},\dots,\theta_S^{(r)})
        \]
    \EndFor
    \If{$\mathcal O(w^{(r)},\theta^{(r)})-\mathcal O(w^{(r+1)},\theta^{(r+1)})\le \varepsilon$}
        \State \textbf{stop}
    \EndIf
\EndFor
\end{algorithmic}
\end{algorithm}

\paragraph{Complexity proofs.}

\begin{proposition}[Solver properties]
\label{prop:solver}
Consider any local objective of the quadratic-risk / linear-gain form
\[
\mathcal O(w,\theta)
=
\text{constant}
-
\sum_i w_i \mu_i(\theta_i)
+
\frac{L}{2}\sum_i w_i^2 \kappa_i(\theta_i)
\]
over a simplex in $w$ and a box in $\theta$, where each $\kappa_i>0$ and each
$\mu_i,\kappa_i$ depends only on $\theta_i$. Then:
\begin{enumerate}[leftmargin=2em]
\item every weight subproblem is a strictly convex quadratic program with the closed-form KKT threshold law of Proposition~\ref{prop:KKT};
\item for fixed weights $w$, the amplitude subproblem is separable across the nodes and reduces to independent one-dimensional convex minimizations;
\item exact block-coordinate descent produces a monotonically nonincreasing sequence of objective values;
\item every limit point is a coordinatewise minimum.
\end{enumerate}
\end{proposition}

\begin{proof}
Item (1) is Proposition~\ref{prop:KKT}. For item (2), the objective decomposes
as a sum of scalar coordinate functions because each $\mu_i$ and $\kappa_i$
depends only on $\theta_i$. It remains to prove convexity of the scalar block.
In all local objectives used in the paper, the only non-polynomial part is
$-s_i(u;\theta_i)$, where
\[
s_i(u;\theta_i)=u-T_{A_i(\theta_i)}(u)
=\frac{A_i(\theta_i)u^2}{1+A_i(\theta_i)u}
=\frac{c_i\theta_i}{1+d_i\theta_i}
\]
for nonnegative constants $c_i,d_i$ depending on $u,L,R$. Differentiating
shows that $s_i$ is concave in $\theta_i$, hence $-s_i$ is convex. The
remaining terms are nonnegative linear, quadratic, or quartic monomials in
$\theta_i$, hence convex. Item (3) holds because each block is minimized
exactly. Item (4) follows from compactness of the feasible set and continuity
of the objective.
\end{proof}

\begin{theorem}[Arithmetic and communication complexity]
\label{th:complexity-main}
At a round with active set size $S_t$:
\begin{enumerate}[leftmargin=2em]
\item node $i$ computes exactly $H_i$ minibatch gradients of size $b_i$ and performs $O(H_i d)$ vector operations;
\item for fixed amplitudes, the exact weight subproblem is solved in $O(S_t\log S_t)$ time by sorting the KKT thresholds;
\item for fixed weights, the amplitude block decomposes into $S_t$ independent one-dimensional convex minimizations;
\item one alternating local-control sweep has optimization overhead $O(S_t\log S_t)$, plus the scalar accuracy cost of the amplitude solves, and server-side vector arithmetic $O(S_t d)$.
\end{enumerate}
Assume broadcast is available. Then the total communication at a round with
active set size $S_t$ is
\begin{equation}
\mathrm{CommRound}(S_t,d)=2d+2dS_t+S_t+\nu_t,
\label{eq:comm-main}\end{equation}
where $\nu_t$ is the number of additional scalars needed to instantiate the
executable variance proxies $\widehat v_{i,t}$. In particular,
$\nu_t=0$ in the idealized model and in the fixed-known-proxy model, whereas
$\nu_t=S_t$ if one scalar proxy is communicated per active node.
The alternating local solver itself is recorded in the appendix.
\end{theorem}

\begin{proof}
Item (1) follows directly from Algorithm~\ref{alg:main}. Item (2) is
Proposition~\ref{prop:KKT}. Item (3) is Proposition~\ref{prop:solver}. Item
(4) combines the optimizer overhead from items (2) and (3) with the server-side
vector sums needed to form the primal and control updates. The communication
count follows from Algorithm~\ref{alg:main}: broadcasting $(x_t,c_t)$ costs
$2d$ numbers, the amplitudes cost $S_t$ scalars, the active uploads cost
$2dS_t$ numbers, and $\nu_t$ counts the additional proxy communication.
\end{proof}

% \newpage
% \input{checklist.tex}

\end{document}